\documentclass[english]{article}
\usepackage{lmodern}

\usepackage[T1]{fontenc}
\usepackage[utf8]{inputenc}
\usepackage[a4paper]{geometry}
\usepackage{dsfont}
\usepackage{color}
\usepackage{babel}
\usepackage{amsmath}
\usepackage{amsthm}
\usepackage{amssymb}
\usepackage{graphicx}
\usepackage{setspace}
\usepackage{esint}
\usepackage{caption}
\usepackage{multirow}
\usepackage{booktabs}
\usepackage{natbib}
\usepackage[unicode=true,bookmarks=true,bookmarksnumbered=false,bookmarksopen=false,breaklinks=false,pdfborder={0 0 0},pdfborderstyle={},backref=page,colorlinks=true]{hyperref}
\hypersetup{pdftitle={On IS and IMH with unbounded weight},pdfauthor={Deligiannidis, Jacob, Khribch, Wang},linkcolor=blue,citecolor=blue}
\usepackage{float}
\usepackage{subfig}
\usepackage[dvipsnames,svgnames,x11names,hyperref,table]{xcolor}
\usepackage{algorithm}
\floatstyle{ruled}
\newfloat{algorithm}{tbp}{loa}
\providecommand{\algorithmname}{Algorithm}
\floatname{algorithm}{\protect\algorithmname}
\newtheorem{thm}{Theorem}[section]

\newtheorem{lem}{Lemma}[section]
\newtheorem{remark}{Remark}[section]
\newtheorem{cor}{Corollary}[section]
\newtheorem{prop}{Proposition}[section]
\newtheorem{exa}{Example}
\newtheorem{asm}{Assumption}

\def \bP {\mathbb{P}}
\def \bE {\mathbb{E}}

\newcommand{\diff}{\mathrm{d}} 

\newcommand{\ZN}{\hat{Z}_N}
\newcommand{\FN}{\hat{F}_N}
\newcommand{\x}{X} % particle
\newcommand{\y}{Y} % particle
\newcommand{\bfy}{\mathbf{Y}^N}
\newcommand{\bfx}{\mathbf{X}^N} % N-vector of particles
\newcommand{\bfxstar}{\mathbf{X}^{N,\star}} % proposed N-vector of particles
\newcommand{\bfxprime}{\mathbf{X}^{N,\prime}} % proposed N-vector of particles
\newcommand{\bfxsmall}{\mathbf{x}^N} % N-vector of particles
\newcommand{\bfxsmallstar}{\mathbf{x}^{N,\star}} % proposed N-vector of particles
\newcommand{\bfysmall}{\mathbf{y}^N}

\newcommand{\FNu}{\hat{F}_{N}^{u}}
\newcommand{\FNusym}{\tilde{F}_{N}^{u}}

\usepackage{todonotes}
\newcounter{todocounter}

\begin{document}
\tolerance=1000
\title{On importance sampling and independent Metropolis--Hastings with an unbounded weight function}
\author{George Deligiannidis\thanks{george.deligiannidis@stats.ox.ac.uk}~ (University of Oxford), Pierre E. Jacob\thanks{pierre.jacob@essec.edu}~ (ESSEC Business
School),\\ El Mahdi Khribch\thanks{elmahdi.khribch@essec.edu}~ (ESSEC Business
School), Guanyang Wang\thanks{guanyang.wang@rutgers.edu}~ (Rutgers University)} 

\maketitle

\begin{abstract}
    Importance sampling and independent Metropolis--Hastings are among the
	  fundamental building blocks of Monte Carlo methods.  Both require a
	  proposal distribution that globally approximates the target
	  distribution, and pointwise evaluation of the Radon--Nikodym
	  derivative of the target distribution relative to the proposal, also
	  called the weight function.  We study the bias of importance sampling
	  and independent Metropolis--Hastings, without assuming that the
	  weight function is bounded.  We show that the common random numbers
	  coupling of independent Metropolis--Hastings is maximal.  Using that
	  coupling, we derive polynomial bounds on the total variation distance
	  of the chain to its target distribution.  We further consider bias
	  removal techniques using couplings, and provide conditions under
	  which the resulting unbiased estimators have finite moments, and
	  under which their efficiency is comparable to that of importance
	  sampling. Experiments illustrate unbiased estimators of the inverse
	  of a normalizing constant, estimators of nested expectations, and
	  combination of importance sampling with robust mean estimation
	  methods.
\end{abstract}

% \tableofcontents
\section{Introduction\label{sec:intro}}%

\subsection{Context and contributions\label{subsec:intro:context}}

\subsubsection{Monte Carlo with global proposals\label{subsec:intro:globalproposals}} Monte Carlo methods aim to approximate a target distribution $\pi$ on a measurable space $(\mathbb{X},\mathcal{X})$, for example $(\mathbb{R}^d,\mathcal{B}(\mathbb{R}^d))$. These techniques are crucial when analytical computation of expectations under $\pi$ is infeasible. 
The goal is to evaluate integrals of functions \(f: \mathbb{X} \longrightarrow \mathbb{R}\) with respect to  \(\pi\):
\begin{align}
  \pi(f) := \mathbb{E}_{\pi}[f] = \int_{\mathcal{X}} f(x) \pi(x) \, \diff x.
\end{align}
Two primary approaches are Markov chain Monte Carlo (MCMC) methods, that construct a Markov chain with $\pi$ as its stationary distribution, and importance sampling (IS) methods, where the target distribution is approximated by weighted samples. Among MCMC methods, the independent Metropolis--Hastings (IMH) algorithm is a specialized form of the Metropolis--Rosenbluth--Teller--Hastings algorithm~\citep{metropolis1953equation,hastings:1970}, in which proposals are drawn from a proposal distribution $q$, independently of the current state of the chain.
We assume that we can sample from $q$ on $(\mathbb{X},\mathcal{X})$, such that
$\pi$ is absolutely continuous with respect to $q$, and we can evaluate the Radon--Nikodym derivative $\omega(x) = \pi(x)/q(x)$ up to a multiplicative constant. We call $\omega$ the \emph{weight} function. 
The same proposal $q$ can be employed in importance sampling \citep[IS,][]{kahn1949stochastic}, where draws from $q$ are weighted by $\omega$.
Rejection sampling could be implemented if $\omega$ was bounded by a known constant, but in this article we let  $\omega$ be unbounded.

\subsubsection{Comparing MH with independent proposals and importance sampling\label{subsec:intro:bias}}
Independent MH and IS are two approaches to account for the discrepancy between a proposal $q$ and a target $\pi$, and their comparison is a natural question. The asymptotic variance of self-normalized IS (SNIS) estimators is lower than that of IMH estimators \citep{deligiannidis2018ergodic}, see \eqref{eq:imh_asymvar_greater} below. In both cases, the squared bias is negligible relative to the variance as the computing effort goes to infinity, and thus the bias can be ignored from the perspective of the asymptotic mean squared error. Perhaps for this reason,
the comparison of the bias of SNIS and IMH has received little attention, particularly in the general setting where the weight $\omega$ may be unbounded. 
However, when the resulting biased estimators are used for downstream tasks, the non-asymptotic bias matters and can even become the dominant issue. 
We review multiple such settings in Section~\ref{subsec:advantages_low_bias}.

The question of bias can be rephrased as follows for bounded test functions.
Among $N$ independent draws $(\x_n)_{n=1}^N$ from $q$, consider the selection of an index $I\in[N]=\{1,\ldots,N\}$ such that $\x_I$ is as close as possible to $\pi$ in total variation distance. Is it preferable to perform $N$ steps of IMH and select the terminal state, or to draw $I$ from a Categorical distribution with probabilities proportional to $\omega(\x_1),\ldots,\omega(\x_N)$, as in Sampling-Importance-Resampling (SIR) \citep{db1988using}? Our developments lead to the conclusion that, in some generality, IMH leads to a smaller bias than SIR; see Section~\ref{subsec:comparison} for a summary.
We further explore methods of entirely removing such biases, following \citet{middleton2019unbiased}, and study their asymptotic efficiency in Section~\ref{sec:biasremoval_snis}.

\subsubsection{Effect of the bias of Monte Carlo estimators\label{subsec:advantages_low_bias}}
We review scenarios where the bias of Monte Carlo estimators has noticeable and often detrimental effects on downstream tasks, motivating both its study and its removal. In these settings, the bias interacts with downstream tasks in a different way than the variance, and thus matters even though its square may be small compared to the variance. 

\paragraph{Optimization with stochastic gradients}
Monte Carlo estimators are often used
in the approximation of gradients within stochastic gradient algorithms. The effect of the bias of gradient estimators in the encompassing optimization has received attention 
\citep[e.g.][]{doucet2017asymptotic,Hu2021,demidovich2023zoo,surendran2024biasedadaptive}.
The literature shows that the bias can induce a non-vanishing error, affect the convergence rate, or have no significant effect, depending on its magnitude. The case of 
IS-based gradient estimators employed in training importance-weighted autoencoders (IWAEs) has received
particular attention: the bias of gradient estimators
is studied in \citet[Theorem~1]{daudel2024learning},
and is shown to propagate 
into the encompassing optimization in \citet[Corollary~4.5, Theorem~4.6 and Theorem~B.1]{surendran2024biasedadaptive},
and in \citet[Theorem~3.7]{surendran2025vae}.
A number of bias-reduction techniques have been proposed
in the context of IWAEs, including jackknife correction \citep[Proposition~3]{nowozin2018debiasing}, randomized multilevel Monte Carlo \citep[Theorems~2 and~3]{ShiCornish}, and the iterated SIR \citep[Theorems~3--4]{cardoso2022br}, assuming bounded weights. Experiments in \citet{ShiCornish,cardoso2022br,surendran2024biasedadaptive} show that, relative to SNIS gradient estimators, reducing or removing the bias can be overall beneficial at matched computational cost.

\paragraph{Nested expectations\label{subsec:nested_expectations_bias}}
The bias of Monte Carlo estimators has a particularly clear effect in the context of nested expectations \citep{andradottir2016computing}.
The quantity of interest is of the form
\begin{equation}\label{eq:nested_def}
  \bE_{1}\left[\bE_{2|1}[f(\Theta_1,\Theta_2)|\Theta_1]\right],
\end{equation}
where $\bE_{2|1}$ denotes a conditional expectation of $\Theta_2$ given $\Theta_1$,
and $f$ is a test function. We assume that we are able to sample $\Theta_1$ from its marginal distribution, and that we can approximate $\bE_{2|1}[f(\Theta_1,\Theta_2)|\Theta_1]$ for each $\Theta_1$ by a numerical method, such as SNIS or IMH.
Such nested expectations arise in modular Bayesian inference \citep{liu2009,plummer2014cuts}, Bayesian inference after multiple imputation \citep{zhou2010note} or experimental design \citep{rainforth2018nesting}.

Assume that, for each realization $\theta_1$ of $\Theta_1$, the conditional expectation 
$\bE_{2|1}[f(\Theta_1,\Theta_2)|\Theta_1=\theta_1]$ is approximated by an estimator $\FN(\theta_1)$  with cost $N$, 
bias of order $N^{-\gamma}$ for some $\gamma\in(0,\infty)$, and variance of order $N^{-1}$.
Consider the estimator $M^{-1}\sum_{m=1}^M \FN(\Theta_{1,m})$, where $\Theta_{1,1},\ldots,\Theta_{1,M}$ are i.i.d.\ draws from the marginal distribution of $\Theta_1$. The total cost is $C = NM$, and optimizing the budget allocation leads to an overall mean squared error of order $C^{-2\gamma/(2\gamma+1)}$ \citep[Remarks 3.5-3.6,][]{andradottir2016computing}. Therefore, the magnitude of the bias of the inner estimator directly impacts the overall rate of convergence of the estimator of the nested expectation of interest,
and therefore any reduction in bias brings concrete benefits in the overall mean squared error.
Furthermore, unbiased estimators for the inner expectation recover the canonical Monte Carlo rate for the nested expectation. We illustrate this in Subsection~\ref{subsec:applications:nestedexpectations}.

\paragraph{Plug-and-play property of unbiased estimators\label{subsec:plug_and_play}}
The effect of the bias of a Monte Carlo estimator in an encompassing method is \emph{a priori} unknown, and the magnitude of that bias is typically hard to bound explicitly. Removing the bias offers a plug-and-play solution. We illustrate this in Section~\ref{subsec:robust_mean_estimation} with the setting of robust mean estimation. Robust mean estimators, such as median-of-means, typically assume access to variables that have expectation equal to the quantity of interest. Biased estimators can be used instead, as done by \citet{Dau2022} with SNIS within median-of-means, but not without adjustments in both theory and implementation.

\subsubsection{Our contributions\label{subsec:intro:contributions}} We consider the bias of self-normalized importance sampling and independent Metropolis--Hastings, and methods to remove that bias. 
Our key assumption is that the \emph{weight} function $\omega$ has $p$ finite moments under the proposal $q$. We obtain results on the bias and higher moments of self-normalized importance sampling estimators in Section~\ref{sec:asbiasis}.
Using a common random numbers coupling studied in Section~\ref{sec:crncoupling}, we show in Section~\ref{sec:imh} that the total variation distance between $\pi$ and the IMH chain at iteration $t$ decays as $t^{-(p-1)}$; this provides polynomial bounds on the bias for bounded test functions. We provide an example for which we obtain matching lower bounds, up to logarithmic terms,  suggesting that our working assumption of $p$ finite moments of $\omega$ under $q$ is close to necessary for polynomial convergence rates.
We further obtain explicit dependence in $N$ for the particle IMH algorithm \citep{andrieu:doucet:holenstein:2010}, a variant of IMH where $N$ proposals are sampled at each iteration.
To establish precise bounds that account for both $t$ and $N$, we develop novel methods to analyze the average weight, $\ZN = N^{-1} \sum_{n=1}^{N} \omega(\x_n)$, and thereby control the rejection probability. We first use the \textit{Paley--Zygmund inequality} ~\citep{petrov2007lower} (an anticoncentration inequality) to provide a lower bound on the probability of $\ZN$ being small, specifically controlling its behavior when $\ZN \le 2$. For larger values of $\ZN$, we divide the range into an intermediate section $(2, 1+t)$ and a large section $[1+t, \infty)$. The \textit{Markov inequality} directly handles the large range. For the intermediate range, we employ a \textit{peeling argument}, breaking it down into a union of smaller intervals and applying the Markov inequality to each. Our multiscale analysis, combining both concentration and anticoncentration inequalities, allows us to control the $t$-th moment of the rejection probability, leading to a total-variation bound via the coupling inequality.
In Section~\ref{sec:biasremoval_snis} we consider the bias removal technique of \citet{glynn2014exact} applied by \citet{middleton2019unbiased} to the particle IMH algorithm. 
We provide conditions under which these unbiased estimators have finite moments,
and conditions under which their efficiency is asymptotically equivalent to that of self-normalized importance sampling. This distinguishes the bias removal technique of \citet{middleton2019unbiased} from other debiasing strategies for IS \citep{blanchet2015unbiasedtaylor,blanchet2015unbiasedmultilevel,ShiCornish}.
Section~\ref{sec:applications} presents numerical illustrations of these unbiased variants of importance sampling, for the estimation of the inverse of the normalizing constant of a distribution, for a nested expectation as in Section~\ref{subsec:nested_expectations_bias}, and in the context of robust mean estimation
as in Section~\ref{subsec:plug_and_play}. The proofs are provided in Apprendix~\ref{appx:proofs}.

\subsection{Importance sampling\label{subsec:intro:is}}

Self-normalized importance sampling (IS) is described in Algorithm~\ref{alg:snis}, see also Chapter 9.2 in \citet{mcbook}. 
Central 
to IS is the weight function:
\begin{equation}\label{eq:weightfunction}
  \omega:x\mapsto \frac{\pi(x)}{q(x)},\quad \text{so that} \quad q(\omega) = 1.
\end{equation}
Since multiplicative constants in $\omega$ have no effect on the IS estimator \eqref{eq:def:Fhat},
it can be computed as long as the user can evaluate a function proportional to $\omega$ in \eqref{eq:weightfunction}.
Unless specified otherwise, by IS we refer to the self-normalized version as in Algorithm~\ref{alg:snis}
rather than the more basic estimator $N^{-1} \sum_{n=1}^N \omega(\x_n) f(\x_n)$
that depends on the multiplicative constant in $\omega$.

\begin{algorithm}[ht]
  \begin{enumerate}
    \item Sample \(N\) particles, \(\x_{1}, \ldots, \x_{N}\), independently from \(q\).
    \item Compute the importance weights: \(\omega(\x_{n}) = \pi(\x_{n}) / q(\x_{n})\) for \(n \in [N]=\{1,\ldots,N\}\).
    \item Compute the normalizing constant estimator:
    \begin{equation}\label{eq:def:Zhat}
      \ZN(\x_{1}, \ldots, \x_{N}) = N^{-1} \sum_{n=1}^{N} \omega(\x_{n}).
    \end{equation}
    \item For any test function \(f\), compute the SNIS estimator
    \begin{equation}\label{eq:def:Fhat}
      \FN(\x_{1}, \ldots, \x_{N}) = \frac{\sum_{n=1}^{N} \omega(\x_{n}) f(\x_{n})}{\sum_{n=1}^{N} \omega(\x_{n})}.
    \end{equation}
  \end{enumerate}
  \caption{Self-normalized importance sampling.
  \label{alg:snis}
  }
\end{algorithm}

For brevity we sometimes write $\bfx$ for $(\x_1,\ldots,\x_N)$, $\FN$ for $\FN(\bfx)$ -- the IS estimator, and $\ZN$
for $\ZN(\bfx)$ -- the normalizing constant estimator, when the dependence on $\bfx$ is clear from the context.
We make the following assumption throughout.

\begin{asm}\label{asm:abscontinuitypositivity}
  For any measurable set $A\in\mathcal{X}$, if \(q(A) = 0\), then \(\pi(A) = 0\). Furthermore, $\omega(\x)$ with $\x\sim q$ is almost surely positive, and $q(\omega) = 1$.
\end{asm}
Under Assumption~\ref{asm:abscontinuitypositivity}, if $\pi(f)$ exists then $ \FN(\x_{1}, \ldots, \x_{N}) \to \pi(f)$ as $N\to\infty$ almost surely.
The asymptotic variance of $\FN$ can be directly computed from the delta method \citep{mcbook,robert:casella:2004,liu2008monte},
assuming $q(\omega^2\cdot f^2)<\infty$ and $q(\omega^2)<\infty$,
\begin{equation}\label{eq:is_asymvar}
  \lim_{N\to\infty}\mathbb{V}\left[\sqrt{N}(\FN(\x_1,\ldots,\x_N) - \pi(f))\right] = q(\omega^2 \cdot (f-\pi(f))^2) =: \sigma^2_{\textsc{IS}}.
\end{equation}
Non-asymptotic bounds on the mean squared error and on the bias
of $\FN$ are derived in \citet{agapiou_2014}. These bounds are inversely proportional to the number $N$ of draws from $q$.
The exact form of the asymptotic bias of the IS estimator $\FN$ is well-known \citep[e.g.][]{hesterberg1988,liu2008monte}, and we provide a formal
statement in Section~\ref{sec:asbiasis}.

\subsection{Independent Metropolis--Hastings\label{subsec:intro:imh}}

Independent Metropolis-Hastings (IMH) is an instance of the Metropolis--Rosenbluth--Teller--Hastings algorithm \citep[Section 2.5]{hastings:1970}, described in Algorithm \ref{alg:IMH}. 
The IMH chain is aperiodic and $\pi$-invariant by design, whereas under  Assumption~\ref{asm:abscontinuitypositivity}, it  is also $\pi$-irreducible; thus 
for $\pi$-almost every $x$, $|P^t(x,\cdot)-\pi|_{\mathrm{TV}}\to 0$ as $t\to\infty$ \citep[Theorem 4 in][]{roberts2004general},
where $P$ denotes the transition kernel of IMH, $P^t$ denotes the $t$-step transition kernel, and $|\mu - \nu|_{\mathrm{TV}} = \sup_{A\in\mathcal{X}} \mu(A) - \nu(A)$ is the total variation distance.

The asymptotic variance of the ergodic average $t^{-1}\sum_{s=0}^{t-1} f(\x_s)$ generated by IMH, denoted by $\sigma^2_{\textsc{IMH}}$, 
is finite if and only if $\pi(f^2)<\infty$ and $q(\omega^2 \cdot f^2)<\infty$ \citep[Theorem 2 in][]{deligiannidis2018ergodic}. Furthermore, if $\sigma^2_{\textsc{IMH}}$ is finite then \citet[Proposition 2,][]{deligiannidis2018ergodic} provides a general comparison:
\begin{equation}\label{eq:imh_asymvar_greater}
  \sigma^2_\textsc{IS} \leq \sigma^2_{\textsc{IMH}},
\end{equation}
with $\sigma^2_{\textsc{IS}}$ as in \eqref{eq:is_asymvar}. Thus IS outperforms IMH in asymptotic variance. Note that, since IMH defines a Markov transition, it can be used directly as a step within an encompassing Gibbs sampler \citep{skare2003improved}, and it is commonly used within sequential Monte Carlo samplers \citep{chopin:2002,South2019}; thus IMH has its specific uses irrespective of the performance comparison with importance sampling.

\begin{algorithm}[!ht]
  \begin{enumerate}
    \item Draw $\x^{\star} \sim q$.
    \item Compute the acceptance probability:
    \[
      \alpha(\x,\x^{\star}) = \min\left\{1,\frac{\omega(\x^{\star})}{\omega(\x)}\right\}.\]
    \item Draw  $U$ from a $\text{Uniform}(0,1)$ distribution, independently of $\x^{\star}$.
    \item If $U<\alpha(\x,\x^{\star})$, return $\x^{\star}$, otherwise return $\x$.
  \end{enumerate}
  \caption{One step of the IMH algorithm starting from state $\x$.}\label{alg:IMH}
\end{algorithm}

When it comes to non-asymptotic behavior, for IMH there is an important distinction between two cases \citep{mengersen1996rates}: either the weight is bounded, in which case the chain is geometrically ergodic and exact rates are obtained in \citet{wang2022exact}, or the weight is unbounded, and the convergence cannot be geometric; in the latter case, various results are provided 
e.g. in \citet{jarner2002polynomial,douc2007computable,roberts2011quantitative,andrieu2022comparison} and \citet[][Chapter 17]{douc2018MarkovChains}; see Section~\ref{subsec:relatedworks}.
In Section~\ref{sec:imh} we provide polynomial bounds on the total variation distance to stationarity
under moment conditions on $\omega$ under $q$. Our results enable bias comparison of IS and IMH in Section~\ref{subsec:comparison}, along with practical recommendations.

In the following we consider the particle generalization of IMH (PIMH), where $N$ proposals are drawn at each iteration \citep[Section 4.2]{andrieu:doucet:holenstein:2010};
see Algorithm~\ref{alg:PIMH}. We define the algorithm on the state space $\mathbb{X}^N$,
use boldface to denote its elements, e.g. $\bfx = (\x_1,\ldots,\x_N)\in\mathbb{X}^N$,
and denote the transition kernel by $P$. If $N=1$ the algorithm corresponds to IMH, and our results apply for all $N\geq 1$. To view Algorithm~\ref{alg:PIMH} as a special case of IMH, define for any $N\geq 1$
\begin{align}
  \bar{\pi}(x_1,\ldots,x_N) &= \sum_{k=1}^N \frac{\pi(x_k)}{N}  \prod_{n \neq k} q(x_n) = \left(\frac{1}{N}\sum_{k=1}^N \omega(x_k)\right) \prod_{n=1}^N q(x_n), \label{eq:pimhtarget}\\
  \bar{q}(x_1,\ldots,x_N) &= \prod_{n=1}^{N} q(x_n), \label{eq:pimhproposal}
\end{align}
and, in the case $N=1$, $\bar{\pi}(x_1) = \pi(x_1)$. From \eqref{eq:def:Zhat} and the above definitions, we define the Radon--Nikodym derivative of $\bar{\pi}$ with respect to $\bar{q}$ as:
\begin{equation}
  \label{eq:pimhweight}
  \bar{\omega}(x_1,\ldots,x_N) = \frac{\bar{\pi}(x_1,\ldots,x_N)}{\bar{q}(x_1,\ldots,x_N)} = \frac{1}{N}\sum_{n=1}^N \omega(x_n) = \ZN(x_1,\ldots,x_N).
\end{equation}
Note that $\bE_{\bar{q}}[\bar{\omega}(\bfx)]=N^{-1}\sum_{n=1}^Nq(\omega)=1$ under Assumption~\ref{asm:abscontinuitypositivity}, thus $\bar{\pi}$ is properly normalized.
Hence, IMH as in Algorithm~\ref{alg:IMH}, with proposal $\bar{q}$ and target $\bar{\pi}$, is equivalent to Algorithm~\ref{alg:PIMH}.

\begin{algorithm}[!ht]
  \begin{enumerate}
    \item Draw $\bfxstar=(\x^{\star}_1,\ldots,\x_N^{\star}) \sim \bar{q}$.
    \item Compute the acceptance probability:
    \begin{align}\label{eq:def:pimh-acceptanceratio}
      \alpha(\bfx,\bfxstar) = \min\left\{1,\frac{\ZN(\bfxstar)}{\ZN(\bfx)}\right\}, \quad \text{where } \ZN: (x_1,\ldots,x_N) \mapsto N^{-1}\sum_{n=1}^N \omega(x_n).
    \end{align}
    \item Draw  $U$ from a $\text{Uniform}(0,1)$ distribution, independently of $\bfxstar$.
    \item If $U<\alpha(\bfx,\bfxstar)$, return $\bfxstar$, otherwise return $\bfx$.
  \end{enumerate}
  \caption{One step of the PIMH algorithm starting from state $\bfx=(\x_1,\ldots,\x_N)$.}\label{alg:PIMH}
\end{algorithm}

In order to estimate an expectation $\pi(f)$ from the PIMH output using all the generated particles, 
note that 
\begin{align}
  \bE_{\bar{\pi}}[\FN(\bfx)] &= \int \FN(x_1,\ldots,x_N) \cdot \ZN(x_1,\ldots,x_N) \cdot \bar{q}(x_1,\ldots,x_N) \diff x_1\ldots\diff x_N \nonumber \\
  &=  \int \left\{\frac{1}{N} \sum_{n=1}^N \omega(x_n) f(x_n) \right\}\cdot \bar{q}(x_1,\ldots,x_N) \diff x_1\ldots\diff x_N \nonumber\\
  &=  \int \omega(x_1) f(x_1) q(x_1) \diff x_1 = \pi(f). \label{eq:unbiasedness-pimh_limit}
\end{align}
Thus, we can evaluate $\FN$ at each state of the chain $(\bfx_t)_{t\geq 0}$ generated by the PIMH algorithm.
The ergodic average $T^{-1}\sum_{t=0}^{T-1} \FN(\bfx_t)$ converges to $\pi(f)$ under suitable conditions. 

\subsection{Moment conditions on the weight}

Our key assumption is the following.

\begin{asm}
  \label{asm:pfinitemoments}
  The weights have a finite $p$-th moment for $p > 1$: $q(\omega^{p})<\infty$.
\end{asm}

This is a weak assumption in the context of both self-normalized importance sampling and IMH, where the weight is often assumed to be bounded.
For bounded test functions, Assumption~\ref{asm:pfinitemoments} with $p\geq 2$
is necessary for the asymptotic variance of both self-normalized importance sampling and IMH to be finite.
The case where $p<2$ is also of interest, and we consider it in the study of both SNIS
(Proposition~\ref{prop:snis_moments_p_less_than_2}) and IMH (in Section~\ref{subsec:p_lessthan_2}).

\begin{exa}[Exponential distributions\label{example:expo}]
  Let $\pi$ be the Exponential(1) distribution and let $q$ be the Exponential($k$) distribution with $q(x) = ke^{-kx}$, both on $\mathbb{R}_+$. If $k\leq 1$, the weight $\omega(x)$ is upper bounded by $k^{-1}$, and Assumption~\ref{asm:pfinitemoments} holds for all $p$. If $k > 1$, then $q(\omega^p)<\infty$ holds 
  with any $p <  k/(k-1)$. The example is considered in \citet{jarneretroberts07,roberts2011quantitative,andrieu2022comparison,orenstein2022robust}.
\end{exa}

\begin{exa}[Normal distributions\label{example:normal}]
  Let $\pi$ be the Normal(0,1) distribution and let $q$ be the Normal$(0,\sigma^2)$ distribution, both on $\mathbb{R}$. If $\sigma^2\geq 1$, the weight $\omega(x)$ is upper bounded by $\sigma$, and Assumption~\ref{asm:pfinitemoments} holds for all $p$. If $\sigma^2<1$, then  $q(\omega^p)<\infty$ holds for $p<\sigma^{-2}/(\sigma^{-2}-1)$.
  The example is considered in \citet{roberts2011quantitative,mcbook}.
\end{exa}

Assumption~\ref{asm:pfinitemoments} with $p\geq 2$ implies the following well-known behavior of the average of $N$ independent and identically distributed random variables, obtained from the Marcinkiewicz--Zygmund inequalities.

\begin{prop}\label{prop:zhatconcentration}
  Let $\bfx=(\x_n)_{n=1}^N$ be $N$ independent random variables with distribution $q$. Under Assumptions~\ref{asm:abscontinuitypositivity}--\ref{asm:pfinitemoments}, with $p \geq 2$, $\ZN(\bfx) = N^{-1}\sum_{n=1}^N \omega(\x_n)$ satisfies, for all $N\geq 1$,
  and $M(p) := 2p\, q(\omega^p)^{1/p}$,
  \begin{align}
    &\mathbb{E}_{\bar{q}}[| \ZN(\bfx)-1|^p]^{1/p} \leq \frac{M(p)}{\sqrt{N}},\label{eq:prop:zhatconcentration:centred}\\
    &\mathbb{E}_{\bar{q}}[ \ZN(\bfx)^p]^{1/p} \leq 1+\frac{M(p)}{\sqrt{N}},\label{eq:prop:zhatconcentration:uncentred}\\
    &\mathbb{P}_{\bar{q}} \left(\ZN(\bfx)\ge 1+t\right) \leq \left(\frac{M(p)}{t\sqrt{N}}\right)^p, \qquad t>0.\label{eq:Z-tail-bound}
  \end{align}
\end{prop}

\begin{remark}\label{rem:smc}
  Some of our results do not require $\ZN(\bfx)$ to be an average of i.i.d.\ weights, but only that $\ZN(\bfx)$ is non-negative, unbiased for the normalizing constant of the target distribution, and satisfies the moment bound~\eqref{eq:prop:zhatconcentration:centred},
  and thus \eqref{eq:prop:zhatconcentration:uncentred} and \eqref{eq:Z-tail-bound} as well since they follow from \eqref{eq:prop:zhatconcentration:centred}.
  These properties may hold, for example, for the normalizing constant estimator generated by sequential Monte Carlo (SMC) samplers, under some conditions; see e.g.\ \citet[Section 16.5]{del2013mean} for related results. Specifically, the results on the coupling of particle IMH and the resulting meeting times in Sections~\ref{sec:crncoupling}-\ref{sec:imh} hold whenever $\ZN(\bfx)$ satisfies these properties, as can be seen from the proofs. However, the results on importance sampling in Section~\ref{sec:asbiasis} and bias removal in Section~\ref{sec:biasremoval_snis} rely on the i.i.d.\ structure of $\bfx$ in multiple places, and do not extend to the SMC setting without modification, which we leave as an open line of research.
\end{remark}

\begin{remark}\label{rem:asm1_p_less_than_2}
  If Assumption~\ref{asm:pfinitemoments} holds for some $p\in (1,2)$, the von Bahr--Esseen inequality \citep{vonbahr_esseen_1965} can be used to obtain a bound similar to \eqref{eq:prop:zhatconcentration:centred} but with $N^{-(p-1)/p}$ instead of $N^{-1/2}$
  on the right-hand side. This is stated in Proposition~\ref{prop:zhatconcentration_less2}.
\end{remark}

\section{Bias and moments of importance sampling\label{sec:asbiasis}}

Our first result is a clean statement on the asymptotic bias of self-normalized importance sampling.
Introductory material on importance sampling often notices that the basic importance sampling estimator $N^{-1}\sum_{n=1}^N  f(\x_n)\omega(\x_n)$ is unbiased, but since $\omega$ can only be evaluated up to a multiplicative constant, users may need to resort to the
self-normalized estimator $\FN(\bfx)$ in \eqref{eq:def:Fhat}, which is biased: $\mathbb{E}[\FN(\bfx)] \neq \pi(f)$. 
The form of the asymptotic bias is well known, e.g. Section 2.5. in \citet{liu2008monte}. 
Formal results with self-contained proofs are however hard to find.
Theorem 2 in \citet{skare2003improved} describes
the pointwise relative error of the density of a particle selected from the IS approximation. Their Remark 1 translates this into
a bound on the bias for bounded functions under the assumption of bounded weights.
The asymptotic bias expression appears in \citet{hesterberg1988}, equation~(2.63), and is justified by the Edgeworth expansion of the cumulative distribution function of $\sqrt{N}(\FN-\pi(f))$, using results from \citet{bhattacharya1978}.
However, as noted in \citet[][p.~40]{hesterberg1988}, the expression is obtained as a first order term in an Edgeworth expansion which does not establish that the actual bias is well-defined or finite.
Moreover, the approach requires third moments of $\omega$ and $\omega f$ under $q$, as well as Cram\'er's condition on the characteristic function of $(\omega, \omega^2, \omega f, \omega^2 f^2)$ \citep[][p.~39]{hesterberg1988}.
Theorem~\ref{thm:asbias_is} below requires fewer positive moments on $\omega$ and $\omega f$, but 
assumes the existence of an inverse moment for $\omega$.

\begin{thm}\label{thm:asbias_is}
  Let $f$ with $\pi(|f|)<\infty$.
  Let $\bfx=(\x_1,\ldots,\x_N) {\sim} \bar{q}$ as in \eqref{eq:pimhproposal} and let $\FN(\bfx)$ be as in \eqref{eq:def:Fhat}.
  Assume that $q(|\omega^2(f - \pi(f))|^{1+\epsilon}) < \infty$ for some $\epsilon > 0$ and $q(\omega^{-\eta})<\infty$ for some $\eta>0$. Then
  \begin{align}
    \lim_{N\to\infty} N \times \mathbb{E}_{\bar{q}}\left[\FN(\bfx) - \pi(f)\right] &= - \int \left(f(x) -\pi(f)\right) \omega^2(x) q(\diff x).  \label{eq:asbias_snis}
  \end{align}
\end{thm}

Theorem~\ref{thm:asbias_is} relates to the question raised in Section~\ref{subsec:intro:bias}. Indeed, 
consider a draw $\x_I$ obtained by Sampling-Importance Resampling (SIR).  
Theorem~\ref{thm:asbias_is} applied to bounded test functions shows that the total variation distance between $\x_I$ and $\pi$ is of order $N^{-1}$ as $N\to\infty$. 
Instead of exact asymptotic order, 
\citet[Theorem 2.1,][]{agapiou_2014} provides an upper bound on the bias under weaker assumptions.
\citet[Theorem 2.3,][]{agapiou_2014} provides upper bounds of order $N^{-1}$ also for unbounded test functions,
under moment conditions on $f$ and on $f\cdot \omega$. Our Theorem~\ref{thm:asbias_is}
establishes that $N^{-1}$ is the exact order of the asymptotic bias as a function of $N$, but requires additional conditions.
Theorem~\ref{thm:asbias_is} is similar to results proposed recently in \citet{daudel2024learning}, under fewer assumptions.
Finally, the bias of sequential Monte Carlo estimators is typically also of order $N^{-1}$, where $N$ is the number of particles, see \citet{del2007sharp} and  \citep[Section 15.3,][]{del2013mean}.

We next provide a result on the $s$-th moments of the error in importance sampling for unbounded test functions. Theorem~\ref{thm: snis unbounded convergence} generalizes the MSE part of Theorem 2.3 in \citet{agapiou_2014} to arbitrary orders $s\geq 2$, and its assumptions are weaker in the case $s=2$, as discussed below. These bounds are used to obtain the results of Section~\ref{sec:biasremoval_snis}, and may be of independent interest.

\begin{thm}\label{thm: snis unbounded convergence}
  Assume that $q(f^2 \cdot \omega^2)<\infty$ and that there exist $p \in [2,\infty)$ and $r \in [2,\infty]$ such that $q(\omega^p) < \infty$ and $q(|f|^r) <\infty$. Then for any $2\leq s \leq pr/(p+r+2)$ and any $N\geq 1$, we have:
  \begin{align*}
    \bE_{\bar{q}}\left[\left\lvert \FN(\bfx) - \pi(f)\right\rvert^s \right]\leq C N^{-s/2},
  \end{align*}
  where the constant $C$ depends on $r, p, s, q(|f|^r), q(\omega^p), q(f^2 \cdot \omega^2)$. 
  When $ r = \infty $, the statement holds for $f$ such that $|f|_\infty < \infty$ and all $s \leq p $.
\end{thm}

A few remarks are in order:
\begin{itemize}
  \item The condition $s \leq pr/(p+r+2)$ implies $s\leq \min\{p,r\}$.
  \item We have $q((|f|\omega)^{pr/(p+r)}) < \infty$ if $q(\omega^p) < \infty$ and $q(|f|^r) <\infty$. Indeed, when $r<\infty$,
$q((|f|\omega)^{pr/(p+r)})\leq q(|f|^r)^{p/(p+r)} q(\omega^p)^{r/(p+r)}<\infty$.
  When $r= \infty$, the claim remains correct (by understanding $pr/(p+r)$ as $p$), since $q((|f|\omega)^p) \leq |f|_\infty^p q(\omega^p)$.
  This observation leads to two facts:  1) If $pr/(p+r) \geq2$ (e.g. $p=r=4$ or $p = 2, r = \infty$), the assumption $q(f^2 \cdot \omega^2)<\infty$ in Theorem \ref{thm: snis unbounded convergence} can be derived from the assumptions $q(\omega^p) < \infty$ and $q(|f|^r) <\infty$. 2) The basic importance sampling estimator $N^{-1}\sum_{n=1}^N f(\x_n)\omega(\x_n)$ has a finite $s$-th moment under the same conditions, as it has a finite $pr/(p+r)$-th moment, and $s \leq pr/(p+r+2) \leq pr/(p+r)$.
  \item We may be particularly interested in the mean-squared error (MSE) of IS, corresponding to $s=2$. 
  Theorem~\ref{thm: snis unbounded convergence} implies that the MSE is of order $1/N$ as long as $2 \leq pr/(p+r+2)$. This condition holds, for example, if $ \min\{p, r\} \geq 2(1+\sqrt{2}) \approx 4.828 $, or if $ p \geq 3 $ and $ r \geq 10 $, or if $ p = 2 $ and $ r = \infty $.
  The case $s=2$ can be compared to the MSE part of Theorem 2.3 in \citet{agapiou_2014}. 
  In our notation, they require 
$q(|f\cdot \omega|^{2d})<\infty$, 
$q(\omega^{2e})<\infty$, 
$q(|f|^{2a})<\infty$, 
$q(\omega^{2b(1+a^{-1})})<\infty$, for $a,b,d,e>1$ such that $a^{-1}+b^{-1}=1, d^{-1}+e^{-1} = 1$. 
  Their assumption implies ours, as can be seen by setting $r = 2a$ and $p = 2b(1+a^{-1})$, since then 
  \begin{align*}
    \frac{pr}{(p+r+2)} = \frac{4b(a+1)}{2a + 2b + 2ba^{-1}+2} = \frac{2(a+1)}{a(1-a^{-1}) + 1 + a^{-1} + (1 - a^{-1})} = \frac{2(a+1)}{a+1},
  \end{align*}
  i.e. our theorem holds with $s=2$ under their assumptions.
  \item Comparable results are given in \cite{ShiCornish}. For bounded test functions \citep[Supplementary material, Lemma~5]{ShiCornish}, the two results are equivalent. For unbounded test functions \citep[Supplementary material, Lemma~6]{ShiCornish}, Theorem~\ref{thm: snis unbounded convergence} yields a bound of $\mathcal{O}(N^{-s/2})$ under strictly weaker conditions, whereas Lemma~6 of \citet{ShiCornish} gives only a bound of $\mathcal{O}(N^{-s/2+s/r})$ under $q(\omega^p)<\infty$ and $q(|f|^r)<\infty$.
  \item Similar types of results have been obtained for sequential Monte Carlo estimators, e.g. in \citet[Section 15.3]{del2013mean}, under strong conditions.
\end{itemize}

The following proposition extends Theorem~\ref{thm: snis unbounded convergence} to the heavy-tailed regime $p\in(1,2)$.

\begin{prop}\label{prop:snis_moments_p_less_than_2}
  Assume that there exist $p\in(1,2)$ and $r\in[1,\infty]$ such that $q(\omega^p)<\infty$, $q(|f|^r)<\infty$, and such that $\alpha := pr/(p+r) > 1$. Then for any $1\leq s\leq\alpha$ and any $N\geq 1$,
  \begin{align*}
    \bE_{\bar{q}}\left[\left\lvert \FN(\bfx) - \pi(f)\right\rvert^s \right]\leq C\, N^{-s(\alpha-1)/\alpha},
  \end{align*}
  where the constant $C$ depends on $r, p, s, q(|f|^r), q(\omega^p)$. When $r = \infty$, the statement holds for $f$ such that $|f|_\infty < \infty$, giving $\alpha = p$ and the rate $N^{-s(p-1)/p}$ for all $1\leq s\leq p$.
\end{prop}

\section{Optimality of coupling IMH with common draws\label{sec:crncoupling}}

With a view toward deriving upper bounds on the total variation distance 
of IMH to stationarity, we consider the common draws (or \emph{common random numbers}) coupling of a generic IMH algorithm, described in Algorithm~\ref{alg:coupledIMH} and denoted by $\bar{P}$. In this section we denote a state variable by $X$ instead of the notation $\bfx$ employed
in Algorithm~\ref{alg:PIMH}, 
to simplify the notation, given that the results apply to both IMH and PIMH and for any $N$. We also write $\hat{Z}(X)$ instead of $\ZN(X)$.
The coupling is very simple and was considered in \citet{liu1996metropolized,roberts2011quantitative}.
It was remarked around Lemma 1 in \citet{wang2021maximal} that this coupling is ``one-step maximal'': the meeting probability $\bar{P}((x,y),D)$, with $D=\{(x',y'): x'=y'\}$, is maximal over all couplings of $P(x,\cdot)$ and $P(y,\cdot)$, and hence equals $1-|P(x,\cdot) - P(y,\cdot)|_{\mathrm{TV}}$. A direct computation gives the identity
\begin{equation}
  \bar{P}((x,y),D)=1-|P(x,\cdot) - P(y,\cdot)|_{\mathrm{TV}}=\int  \min \left\{\frac{\hat{Z}(x^\star)}{\hat{Z}(x)},\; \frac{\hat{Z}(x^\star)}{\hat{Z}(y)}, \; 1 \right\}\bar{q}(\diff x^\star).
\end{equation}

\begin{algorithm}[!ht]
  \begin{enumerate}
    \item Draw $X^\star \sim \bar{q}$.
    \item Draw \(U\) from a Uniform(0,1) distribution, independently of $X^\star$.
    \item If $U < {\hat{Z}(X^\star)}/{\hat{Z}(X)}$, set \(X' = X^\star\),
    otherwise set \(X' = X\).
    \item If \(U < {\hat{Z}(X^\star)}/{\hat{Z}(Y)}\), set \(Y'= X^\star\),
    otherwise set \(Y' = Y\).
    \item Return $(X',Y')$
  \end{enumerate}
  \caption{Common draws coupling of one step of IMH starting from states $X$ and $Y$.\label{alg:coupledIMH}}
\end{algorithm}

Let $(X_t,Y_t)$ be a coupled chain started from $(x,y)$ at time zero, and evolving according to $\bar{P}$.
Denoting the meeting time by
\begin{equation}\label{eq:deftau}
  \tau = \inf\{t\geq 1: X_t = Y_t\},
\end{equation}
the coupling inequality states that, for $t\geq 1$,
\begin{equation}
  |P^t(x,\cdot) - P^t(y,\cdot)|_{\mathrm{TV}} \leq \mathbb{P}_{x,y}(\tau > t),
\end{equation}
where the probability $\mathbb{P}_{x,y}$ is under the law of $(X_t,Y_t)$ started from $(x,y)$ at time zero.
We will relate the probability $\mathbb{P}_{x,y}(\tau > t)$ to the rejection probabilities of IMH from $x$ and $y$, and we define
\begin{equation}\label{eq:pimh:reject}
  r: x \mapsto \int_{x^\star \neq x} \left(1-\alpha(x,x^{\star})\right) \bar{q}(\diff x^{\star}),
\end{equation}
where $\alpha(x,x^{\star})$ is defined in \eqref{eq:def:pimh-acceptanceratio}. 

The meeting time $\tau$ is the first time at which both chains accept the proposal simultaneously,
which corresponds to the first time at which 
the chain with the highest weight accepts the proposal. Indeed, if $\hat{Z}(x) \geq \hat{Z}(y)$,
then $\alpha(x,x^\star) \leq \alpha(y,x^\star)$ for all $x,y,x^\star$, 
and thus $u<\alpha(x,x^\star)$ implies that $u<\alpha(y,x^\star)$.
Thus, conditionally on $x_0=x,y_0=y$, the meeting time $\tau$ follows
a Geometric distribution with parameter $1-r(x)$, where $r(x)$ is defined in \eqref{eq:pimh:reject}.
Recall that the survival function of a Geometric variable $T$ with parameter $\gamma$ is given by:
$\mathbb{P}(T > t) = (1 - \gamma)^{t}$ for $t \in \mathbb{N}$.
Still assuming $\hat{Z}(x)\geq \hat{Z}(y)$, we obtain, for $t\geq 1$,
\begin{equation}\label{eq:upperbound-roberts}
  |P^t(x,\cdot) - P^t(y,\cdot)|_{\mathrm{TV}} \leq \mathbb{P}_{x,y}(\tau > t) = (r(x))^t.
\end{equation}
The above upper bound is given in \citet{roberts2011quantitative}. In their remark following Theorem 5, they state that this is also a lower bound without providing a proof. 
We do so below, for both discrete and continuous state spaces; \citet{roberts2011quantitative} focus on non-atomic spaces. First, we express
\begin{equation}
  |P^t(x,\cdot) - P^t(y,\cdot)|_{\mathrm{TV}} = \sup_{A\in \mathcal{X}} |P^t(x,A) - P^t(y,A)|,
\end{equation}
and we select the set $A = \mathbb{X} \setminus \{x\}$ to obtain a lower bound, recalling that $(\mathbb{X},\mathcal{X})$ refers to the measurable space on which the Markov chains are defined. Then $P^t(y,A) = 1$  since $x\neq y$, assuming that $q(\{x\})=0$, while $P^t(x,A) = 1 - \left(r(x)\right)^t$, i.e.\ the chain is in $A$ at step $t$ except if $t$ proposals have been rejected. The situation 
is slightly more complicated if the proposal has non-zero mass on $\{x\}$ and $\{y\}$, e.g. in discrete state spaces, but the following result still holds.

\begin{thm}\label{thm:couplingIMH}
  Let $(X_t,Y_t)$ be a Markov chain evolving according to $\bar{P}$ in Algorithm~\ref{alg:coupledIMH} and starting from $X_0=x$ and $Y_0=y$. Let $\tau=\inf\{t\geq 1: X_t = Y_t\}$, and let $r(x)$ be defined as in \eqref{eq:pimh:reject}.
  Then, under Assumption~\ref{asm:abscontinuitypositivity}, for all $t\geq 1$,
  \begin{equation}
    \label{eq:couplingIMH-maxcoupling}
    \displaystyle |P^t(x,\cdot) - P^t(y,\cdot)|_{\mathrm{TV}} = \mathbb{P}_{x,y}(\tau > t) = \max(r(x),r(y))^t.
  \end{equation}
\end{thm}

Thus, the chain $(X_t,Y_t)$ generated by the common draws coupling follows a \emph{maximal coupling}, as in \citet{pitman1976coupling}: the coupling inequality is an equality at all times $t$.
To the best of our knowledge, this is the only known case of ``all time maximal'' couplings of an MCMC algorithm. This peculiarity could prove useful when considering the problem of numerically solving for the optimal coupling between Markov chains, as in the literature on bicausal optimal transport \citep{moulos2021bicausal} and bisimulation metrics \citep{calo2024bisimulation}. The case of IMH provides a rare case of an explicit solution, which could serve as a benchmark
for algorithmic approaches to the problem of finding optimal couplings.
Note that the upper bound in \eqref{eq:couplingIMH-maxcoupling} decreases geometrically in $t$; the polynomial rates that come later come from the integration with respect to $x$ or $y$.

\section{Meeting times and polynomial convergence\label{sec:imh}}

\subsection{Meeting times of lagged chains\label{subsec:imh:meetingtimes}}

We consider coupled IMH chains with a lag, as in \citet{middleton2019unbiased}. The construction is described in Algorithm~\ref{alg:coupledPIMHlag}.
Compared to Section~\ref{sec:crncoupling}, note the redefinition of the meeting time $\tau$,
which now corresponds to $\inf\{t\geq 1: \bfx_t = \bfy_{t-1}\}$. The generated chains $(\bfx_t)_{t\geq 0}$ and $(\bfy_t)_{t\geq 0}$ have the same marginal distribution, that of an IMH chain started from $\bar{q}$. Recall that we denote by $P$ the transition kernel of the particle IMH chain in Algorithm~\ref{alg:PIMH}, and by $\bar{P}$ the transition kernel of the coupled PIMH chain in Algorithm~\ref{alg:coupledIMH}, both of which depend on the parameter $N$.
\begin{algorithm}[!ht]
  \begin{enumerate}
    \item Set \(\tau = +\infty\) and \(t = 1\).
    \item Draw $\bfx_0 \sim \bar{q}$ and $\bfy_0\sim \bar{q}$ independently.
    \item Draw $U$ from a $\text{Uniform}(0,1)$ distribution, independently of $\bfx_0$ and $\bfy_0$.
    \item If $U < \ZN(\bfy_0)/\ZN(\bfx_0)$, set $\bfx_1=\bfy_0, \tau = 1$. Otherwise, set \(\bfx_1 = \bfx_0\).
    \item While \(\tau = +\infty\):
    \begin{enumerate}
      \item Sample $(\bfx_{t+1},\bfy_t) \sim \bar{P}((\bfx_{t},\bfy_{t-1}),\cdot)$, the common draws coupling of PIMH in Algorithm~\ref{alg:coupledIMH}.
      \item If $\bfx_{t+1} = \bfy_t$, set \(\tau = t+1\).
      \item Set \(t = t + 1\).
    \end{enumerate}
    \item Return $\tau$ and the generated chains up to time $\tau$.
  \end{enumerate}
  \caption{Coupled PIMH with a lag\label{alg:coupledPIMHlag}.}
\end{algorithm}

We relate the distribution of the meeting times generated by Algorithm~\ref{alg:coupledPIMHlag} to the expected rejection probability in the following result.

\begin{prop}\label{prop:meetingtimes_rejproba}
  Consider $\tau$ generated by Algorithm~\ref{alg:coupledPIMHlag}. Under Assumption~\ref{asm:abscontinuitypositivity}, for all $t\geq 1$, with $r$ defined in \eqref{eq:pimh:reject}, we have
  \[
    \mathbb{P}(\tau>t) \leq \mathbb{E}_{\bar{q}}\left[\left(r(\bfx)\right)^t\right].
    \]
\end{prop}

This connection between meeting times and expected rejection probability motivates our next result that bounds the latter. That rejection probability is known to be central in the study of IMH, e.g. Theorem 6 in \citet{roberts2011quantitative}. Our bounds are explicit functions of $t$ and $N$.
We focus on the case $p\geq 2$ until Section~\ref{subsec:p_lessthan_2},
where we consider the case $p\in (1,2)$.

\begin{prop}\label{prop:expected_rejection_probability_bound}
  Fix $p \geq 2$ and let 
  \begin{equation}\beta_p:= 1-\left({
    2^{\frac{3p-2}{p-1}}
    q(\omega^p)^{\frac{1}{p-1}}
    }\right)^{-1}.\label{eq:betap}
  \end{equation}
  Under Assumptions~\ref{asm:abscontinuitypositivity}-\ref{asm:pfinitemoments}, $\beta_p\in(0,1)$ and there exist finite constants $A_p, C_p>0$, depending only on $p$ and $q(\omega^p)$, such that for all $N\geq 1$, for all $t\geq 1$, the following holds:
  \[
    \mathbb{E}_{\bar{q}}\left[r(\bfx)^t\right] \leq  \frac{A_p}{N^{(t\wedge p)/2}}\beta_p^t + \frac{C_p}{t^p N^{p/2}}.
    \]
\end{prop}

Proposition~\ref{prop:expected_rejection_probability_bound} holds
for all $t\geq 1$ and all $N\geq 1$. The bounds decay to $0$ as either 
$N$ or $t$ approaches infinity, polynomially with rate at most $N^{-1/2}$ with respect to $N$,  
and, for fixed $N$, polynomially with rate $t^{-p}$.
A direct consequence of the previous two propositions is a bound on the tails of the meeting times.

\begin{prop}\label{prop:upb_meetingproba}
  Consider $\tau$ generated by Algorithm~\ref{alg:coupledPIMHlag}. Under Assumptions~\ref{asm:abscontinuitypositivity}-\ref{asm:pfinitemoments}, there exists a finite $C>0$ such that for all $N\geq 1$ and all $t\geq 1$, if $p\geq 2$ in Assumption~\ref{asm:pfinitemoments},
  \[
  \displaystyle \mathbb{P}(\tau>t) \leq  \frac{C}{\sqrt{N}t^{p}}.
  \]
  As a consequence, we have $\mathbb{E}[\tau] \leq 1 + C'/\sqrt{N}$ with $C'=C\sum_{t\geq 1} t^{-p}$.
\end{prop}

That bound retains the slowest rates in $N$ and $t$ from the previous result. Proposition~\ref{prop:upb_meetingproba} is consistent with Proposition~8 in \citet{middleton2019unbiased}, which showed that $\mathbb{P}(\tau = 1)$ approaches one as $N\to\infty$ under the assumption of bounded weights. However, our present assumptions are considerably weaker, and we provide explicit dependencies on both $N$ and $t$.

\begin{remark}
  We comment on the sharpness of the dependency on $N$ in Proposition~\ref{prop:upb_meetingproba}. For $t=1$, the result reads $\mathbb{P}(\tau >1) \leq  C/\sqrt{N}$. The event $\{\tau > 1\}$ corresponds to the rejection of $\bfxstar$ from a state $\bfx$, both $\bfx,\bfxstar$ being independent draws from $\bar{q}$. Here we show that we cannot improve upon the rate $N^{-1/2}$ as a function of $N$. 
  The central limit theorem implies $\sqrt N(\ZN(\bfx) - 1)\to\text{Normal}(0, q(\omega^2) - 1)$ in distribution. Therefore,
  $\bP(\ZN(\bfx) \geq 1 + N^{-1/2})\rightarrow  p_0$ as $N\to\infty$, with $p_0$ depending on $q(\omega^2)$.
  The same argument shows $\bP(\ZN(\bfxstar) \leq 1 - N^{-1/2}) \rightarrow  p_1$ as $N\to\infty$, with $p_1$ depending on $q(\omega^2)$.
  Therefore, we can choose a large enough $N$ that depends  on $q(\omega^2)$ such that $\bP(\ZN(\bfx) \geq 1 + N^{-1/2})\geq p_0/2$ and $	\bP(\ZN(\bfxstar) \leq 1 - N^{-1/2}) \geq p_1/2$. Thus, with a constant probability $c$, $\ZN(\bfxstar)  \leq 1 - N^{-1/2}$ and  $\ZN(\bfx) \geq 1 + N^{-1/2}$ occur simultaneously, and thus the  acceptance probability is at most $(1 - N^{-1/2})/ (1 + N^{-1/2}) \leq 1 - N^{-1/2}$. In turn this means that the rejection probability is at least $cN^{-1/2}$.
\end{remark}

\subsection{Polynomial convergence rates\label{subsec:imh:convrates}}

As discussed in Section 6 of \citet{joa2020} and in \citet{biswas2019estimating}, lagged chains such as those generated by Algorithm~\ref{alg:coupledPIMHlag} can be employed to bound the total variation distance between the chain at time $t$ and its stationary distribution.
We state upper bounds on the distance of IMH chains to stationarity that are explicit in their dependency on $t$ and $N$.

\begin{thm}\label{thm:convergence_rate_from_q}
  Consider $\tau$ generated by Algorithm~\ref{alg:coupledPIMHlag}. Let $P$ be the transition kernel of the PIMH chain as in Algorithm~\ref{alg:PIMH}. Under Assumption~\ref{asm:pfinitemoments} with $p\geq 2$, we have that for all  $t\geq 0$,
  \begin{equation}\label{eq:tvupperbound_from_laggedchains}
    \left|\bar{q}P^t - \bar{\pi}\right|_{\mathrm{TV}} \leq \mathbb{E}\left[\max\left(0, \tau - 1 - t \right)\right].
  \end{equation}
  Furthermore, still under Assumption~\ref{asm:pfinitemoments} with $p\geq 2$, there exists a constant $C$, independent of $t$ and $N$, such that for all $N\geq 1$ and $t\geq 0$,
  \begin{equation}\label{eq:tvupperbound_from_laggedchains_Nandt}
    \left|\bar{q} P^t - \bar{\pi}\right|_{\mathrm{TV}} \leq \frac{C}{\sqrt{N} (1+t)^{p-1}}.
  \end{equation}
\end{thm}

\begin{remark}
  The case $t=0$ states that $|\bar{q}-\bar{\pi}|_{\mathrm{TV}} \leq C N^{-1/2}$, despite both  $\bar{\pi}$ and $\bar{q}$ being defined in \eqref{eq:pimhtarget}-\eqref{eq:pimhproposal} on spaces growing with $N$.
  With the density representation of the total variation distance, we can directly compute
  \begin{align*}
    |\bar{q}-\bar{\pi}|_{\mathrm{TV}} &= 
    \frac{1}{2} \int |1-\ZN(\mathbf{x}^N)|\bar{q}(\diff{} \mathbf{x}^N)\\
    &\leq \frac{1}{2}\mathbb{E}_{\bar{q}}\left[|1-\ZN(\bfx)|^2\right]^{1/2}
    \leq \frac{1}{2} M(2) N^{-1/2},
  \end{align*}
  where the first inequality is Cauchy--Schwarz and the second uses Proposition~\ref{prop:zhatconcentration} under Assumption~\ref{asm:pfinitemoments} with $p\geq 2$. Furthermore, 
  in the large $N$ asymptotics we expect $1-\ZN(\bfx)$ to behave as a Normal random variable with mean zero and standard deviation $\sqrt{q(\omega^2)-1}/\sqrt{N}$,
  so that the expectation of its absolute value should indeed behave as $\sqrt{(2/\pi)(q(\omega^2)-1)}/\sqrt{N}$. 
\end{remark}

We also state a bound for a PIMH chain started at any initial point $\bfxsmall \in \mathbb{X}^N$. 

\begin{cor}\label{cor:convergence_rate_from_x}
  Under Assumptions~\ref{asm:abscontinuitypositivity}-\ref{asm:pfinitemoments}, with $p\geq 2$, there exists a constant $\tilde{C}$, independent of $t$ and $N$, such that for all $N\geq 1$, $t\geq 1$, and any starting point $\bfxsmall\in \mathbb{X}^N$,
    \[
    \left|P^t(\bfxsmall,\cdot) - \bar{\pi}\right|_{\mathrm{TV}} \leq (r(\bfxsmall))^t + \frac{\tilde{C}}{\sqrt{N} \; t^{p-1}}.
    \]
\end{cor}

Theorem~\ref{thm:convergence_rate_from_q} and Corollary~\ref{cor:convergence_rate_from_x} provide explicit bounds on the convergence rate of the PIMH algorithm. Both results are interpretable in terms of the number of iterations $t$ and the number of particles $N$, and apply to IMH as a special case when $N=1$. The difference between these results lies in the starting distribution. Practitioners would typically start the algorithm from the proposal distribution, as it is the best available approximation of the target. Corollary~\ref{cor:convergence_rate_from_x} reveals two phases in the  convergence from a fixed starting point: an initial phase where the distance decays exponentially in $t$ but not arbitrarily with $N$, followed by a polynomial decay in both $t$ and $N$. 

\begin{remark} The weight $\omega$ can be unbounded while
  having infinitely many moments under $q$, i.e. $q(\omega^p)<\infty$ for all $p\geq 1$. For example, this happens when $\pi$ is Gamma$(2,1)$ and $q$ is Exponential$(1)$, leading to $\omega(x) = x$. In that case \citet{mengersen1996rates} prove that the IMH chain cannot be geometrically ergodic, while Corollary~\ref{prop:convimh_from_x} holds with any $p > 1$. Indeed, the actual decay of 
  the distance to stationarity could be between geometric and polynomial in $t$, for example of the form $\exp(-t^{1/2})$.
\end{remark}

\subsection{Case when the weight has less than two moments\label{subsec:p_lessthan_2}}

We consider the heavy-tail regime where $q(\omega^p)<\infty$ for $p\in(1,2)$.
The next results show that the polynomial rate in $t$ remains $t^{-(p-1)}$, but the dependence on $N$ deteriorates from $N^{-1/2}$ to $N^{-(p-1)/p}$.

\begin{thm}\label{thm:convergence_rate_from_q_less2}
  Consider $\tau$ generated by Algorithm~\ref{alg:coupledPIMHlag}. Let $P$ be the transition kernel of the PIMH chain as in Algorithm~\ref{alg:PIMH}. Under Assumption~\ref{asm:abscontinuitypositivity}, and assuming that $q(\omega^p)<\infty$ for some $p\in(1,2)$, there exists a constant $C$, independent of $t$ and $N$, such that for all $N\geq 1$ and $t\geq 0$,
    \[
    \left|\bar{q} P^t - \bar{\pi}\right|_{\mathrm{TV}} \leq \frac{C}{N^{(p-1)/p}(1+t)^{p-1}}.
  \]
\end{thm}

\begin{cor}\label{cor:convergence_rate_from_x_less2}
  Under Assumption~\ref{asm:abscontinuitypositivity}, and assuming that $q(\omega^p)<\infty$ for some $p\in(1,2)$, there exists a constant $\tilde{C}$, independent of $t$ and $N$, such that for all $N\geq 1$, $t\geq 0$, and any starting point $\bfxsmall\in \mathbb{X}^N$,
  \[
    \left|P^t(\bfxsmall,\cdot) - \bar{\pi}\right|_{\mathrm{TV}} \leq (r(\bfxsmall))^t + \frac{\tilde{C}}{N^{(p-1)/p}(1+t)^{p-1}}.
  \]
\end{cor}
The case $N=1$ corresponds to IMH. Combining Corollaries~\ref{cor:convergence_rate_from_x} and \ref{cor:convergence_rate_from_x_less2}, the rate $t^{-(p-1)}$ holds for every $p>1$, as stated in the following result.

\begin{cor}\label{prop:convimh_from_x}
  Consider IMH under Assumption~\ref{asm:abscontinuitypositivity}, and Assumption~\ref{asm:pfinitemoments} for any $p>1$. There exists $D<\infty$ independent of $t$ such that for all $t\geq 1$, and any starting point $x\in \mathbb{X}$,
  \[
    \left|P^t(x,\cdot) - \pi\right|_{\mathrm{TV}} \leq (r(x))^t + \frac{D}{t^{p-1}}.
  \]
\end{cor}

We provide a direct proof in Appendix~\ref{appx:proofs:PIMH_less2moments}, which may be of independent interest. That proof is closer in spirit to existing analyses of IMH mentioned in Section~\ref{subsec:relatedworks}.

\subsection{Lower bounds on the convergence rate of IMH\label{subsec:lowerbounds_cvgrates}}

The purpose of the following example is to demonstrate that the rate $t^{-(p-1)}$ in Corollary~\ref{prop:convimh_from_x} cannot be improved beyond polylogarithmic factors, without further assumptions.

\begin{exa}\label{example:lowerbound-generalp}
  Fix any $p > 1$. Let $k\geq 1$ be such that $k>1/(p-1)$. Consider the IMH algorithm targeting
  $\pi(x) := Z_\pi x^{-p}$ on $[2,\infty)$, with proposal distribution
  $q(x) := Z_q \log^k(x) x^{-(p+1)}$ on $[2,\infty)$, started from $x_0 = 3$.
  Then $q(\omega^p)<\infty$, and there exist $C<\infty$ and $t_0\in\mathbb{N}$ such that, for all $t\geq t_0$,
  \begin{align}\label{eq:lowerbound_generalp}
    \left|P^t(x_0, \cdot )-\pi\right|_{\mathrm{TV}} \geq \frac{C}{t^{p-1}(\log t)^{(k+1)(p-1)}}.
  \end{align}
\end{exa}

\subsection{Related results on IMH\label{subsec:relatedworks}}

The convergence of IMH has garnered significant interest over decades, and in particular the sub-geometric rates have been studied in several works including \citet{jarner2002polynomial,douc2007computable,roberts2011quantitative,andrieu2022comparison}. 
One approach utilizes drift and minorization techniques \citep{jarner2002polynomial}.

\begin{thm}[Theorem 5.3 in \citet{jarner2002polynomial}] \label{thm:jarner-imh}
  Let $P$ be the transition kernel of the IMH chain as in Algorithm~\ref{alg:IMH}. Assume that for some $r > 0$,
  \begin{align}\label{eq:condition-jarner}
    \displaystyle   \pi(A_{\epsilon}) = \mathcal{O}(\epsilon^{1/r}) \quad \text{for} \quad \epsilon \to 0,
  \end{align}
  where $A_{\epsilon}=\left\{x\in\mathbb{X}: \omega(x)>1/\epsilon \right\}$, for any $\epsilon>0$. Then, for any $x\in \mathbb{X}$, and any $t\geq 1$, we have that
  \[
    \displaystyle  \lim_{t\to\infty} (1+t)^{\beta}|P^t(x,\cdot)-\pi |_{\mathrm{TV}} = 0,
    \]
  for any $0\le \beta \le(s-r)/r$, with $r<s<r+1$.
\end{thm}

The $\mathcal{O}$ notation here is such that if $f(x) = \mathcal{O}(g(x))$ then there exists a constant $M$ such that $|f(x)| \leq M|g(x)|$ for all $x$ in the domain of $f$.
Theorem~\ref{thm:jarner-imh} provides a polynomial rate of convergence for the IMH chain in total variation of order $o\left(t^{-1/r+\kappa}\right)$ for any $\kappa>0$ under the assumption that the tail weights satisfy the condition specified in equation~(\ref{eq:condition-jarner}).    
Notably, under the assumption $q(\omega^p)<\infty$, the condition in \eqref{eq:condition-jarner} is satisfied for $r=1/(p-1)$, using Markov's inequality. 
Our Corollary~\ref{prop:convimh_from_x} differs slightly as our bounds are in $t^{-(p-1)}$ instead of $t^{-(p-1)+\kappa}$ for some arbitrarily small $\kappa>0$.
Similar results can be obtained using weak Poincar\'e inequalities 
as described in Remark 29 of \citet{andrieu2022comparison}, under $\pi(\omega^{p})<\infty$ with $p>1$ which amounts to our Assumption~\ref{asm:pfinitemoments} with $p> 2$.

Our bounds in Theorem~\ref{thm:convergence_rate_from_q} and Corollary~\ref{cor:convergence_rate_from_x} have the advantage of providing an explicit dependency on $N$ in the case of PIMH, which is critical for the results on bias removal in Section~\ref{sec:biasremoval_snis}.

\subsection{Bias comparison and practical recommendations\label{subsec:comparison}}

Results in Sections~\ref{sec:asbiasis} and \ref{sec:imh} enable a comparison of the bias of IS and IMH under Assumptions~\ref{asm:abscontinuitypositivity}-\ref{asm:pfinitemoments}, addressing the question raised
in Section~\ref{subsec:intro:bias}.
Performing $N$ steps of IMH to
obtain a sample $X_N$, and considering a test function $f$ bounded by one,
Corollary~\ref{prop:convimh_from_x}
shows that the bias of $f(X_N)$ is at most $C\,N^{-(p-1)}$.
In terms of variance, however, it is preferable to average over successive states
of the chain. We recommend setting a burn-in period $T = \lceil c N^\eta \rceil$ with
$\eta\in(0,1]$ and $c>0$, and averaging over $X_{T+1},\dots,X_{N}$. We can then bound the bias of the average by 
\begin{equation*}
  \frac{1}{N-T}\sum_{t=T+1}^{N}\left|qP^t f-\pi(f)\right|
  \le \frac{C}{N-T}\sum_{t=T+1}^{\infty} t^{-(p-1)}
  = \mathcal{O}\left(N^{-1-\eta(p-2)}\right),
\end{equation*}
using $\sum_{t>T} t^{-(p-1)}=\mathcal{O}(T^{-(p-2)})$ as in (18) in the proof of
Theorem~\ref{thm:convergence_rate_from_q} and $N-T\sim N$. 
For $\eta<1$ the burn-in phase is negligible in the limit, 
and thus the ergodic average has the same asymptotic variance as
if no burn-in were used.
A value $\eta = 1$ makes the bias decrease to $N^{-(p-1)}$, but the
asymptotic variance is inflated by $1/(1-c)$.

On the other hand, Theorem~\ref{thm:asbias_is} shows that the bias of the SNIS estimator with $N$ samples is of the order of $N^{-1}$, under an additional inverse weight assumption. Therefore, as soon as $q(\omega^p)<\infty$ with $p>2$, the bias of IMH
becomes smaller than that of SNIS for a similar budget of weight function evaluations. Our results also show that, in terms of bias and for a fixed total budget, the best choice for the number of proposals per iteration in PIMH is one.
If the weights are bounded, MCMC methods such as PIMH or particle Gibbs \citep{andrieu:doucet:holenstein:2010} are geometrically ergodic \citep[e.g.][]{wang2022exact,lee2019coupled,cardoso2022br}, and thus the bias decays geometrically. The bias comparison is then clearly in favor of MCMC algorithms.

We therefore recommend considering IMH as a potential alternative to SNIS when low-bias estimators are desired, see e.g., Section~\ref{subsec:advantages_low_bias} for examples of such settings such as gradient estimation or nested expectations.
Importantly,
the user typically does not know the largest value of $p$ such that $q(\omega^p)<\infty$, and thus the exact rate at which the bias of IMH decreases is unknown. This can be problematic: for example, for nested expectations,
optimizing budget allocation requires knowledge of the bias of the inner expectation estimator \citep{andradottir2016computing}. To address this, one approach would be to estimate the largest value of $p$ such that $q(\omega^p)<\infty$. This is a tail estimation problem, which has been considered in the context of importance sampling in \citet[e.g. ][]{monahan1993testing,koopman2009testing,vehtari2024pareto}.
In Section~\ref{sec:biasremoval_snis}, we pursue
a different approach, which is to remove the bias, without having to estimate the tail behavior of the weight. We will consider a bias removal technique, described in Algorithm~\ref{alg:suis}, that is applicable whenever SNIS is, 
and incurs a moderate and controlled increase in variance.

For completeness, we recall that various bias and variance reduction techniques are available for both IS and IMH. Bias reduction strategies
for IS include the leave-one-out modification proposed in \citet{skare2003improved},
or the jackknife \citep{shao2012jackknife,nowozin2018debiasing}, both of which can reduce the bias to $N^{-2}$; we implement the former in Section~\ref{subsec:applications:nestedexpectations}. Variance reduction methods for IS, beyond optimizing the choice of proposal distribution, include \citep{owen2025zero,henmi2007importance}, and variance reduction methods for IMH include Rao-Blackwellization of the acceptance step \citep{casella:robert:1996}, 
averaging over different orderings of the proposed samples \citep{atchade2005improving,jacob2011using},
and, depending on the tractability of the proposal distribution, constructing control variates as in \citet{neal2001improving,alexopoulos2023variance}.

\section{Bias removal for self-normalized importance sampling\label{sec:biasremoval_snis}}

\subsection{Motivation\label{subsec:motivation_biasremoval}}

We consider techniques to remove the bias of IS entirely. There exist at least three distinct methods to achieve this. We focus on the approach of \citet{middleton2019unbiased}, that follows the pioneering work of \citet{glynn2014exact}, and relies on couplings of PIMH chains as in Algorithm~\ref{alg:coupledPIMHlag}. We refer to this approach as Coupling-based Unbiased Importance Sampling (Coupling-UIS).
Another approach is based on randomly-truncated multilevel Monte Carlo (MLMC) 
\citep{blanchet2015unbiasedmultilevel,ShiCornish}. We refer to it as MLMC-UIS.
Finally, any approach that provides unbiased estimators of $1/q(\omega)$ can be used to remove the bias of IS, since multiplying an unbiased estimator of $1/q(\omega)$ by an independent unbiased estimator of $q(\omega f)$ yields an unbiased estimator of $\pi(f)$; here we understand $\omega$ as being possibly unnormalized. An approach using a Taylor expansion of the inverse function was
proposed in  \citet{blanchet2015unbiasedtaylor,chopin_crucinio_singh_2025}. We refer to that approach as Taylor-UIS. Both MLMC-UIS and Taylor-UIS are described in Appendix~\ref{app:other_uis}. Conversely, using the test function $f:x\mapsto \omega^{-1}(x)$, 
the Coupling-UIS estimator of $\pi(f)$ provides an unbiased estimator of $1/q(\omega)$; see Section~\ref{subsec:inverse_Z} and Section~\ref{sec:applications}.

We propose conditions under which Coupling-UIS  has finite moments, 
we propose a Rao--Blackwellization of the scheme for increased efficiency,
and we provide conditions under which its asymptotic efficiency matches that of IS as $N\to\infty$ (see Proposition~\ref{prop:suisinefficiency}). 
The two other approaches, MLMC-UIS and Taylor-UIS, do not seem to match this asymptotic efficiency, as we observe in Section~\ref{subsec:applications:inverse_Z}.

\subsection{Construction\label{subsec:dsnis:construction}}

For bias removal, \citet{middleton2019unbiased} employ common random numbers couplings of PIMH within the approach of \citet{glynn2014exact}.
Upon running Algorithm~\ref{alg:coupledPIMHlag} with $N\geq 1$, with $\tau=\inf\{t\geq 1: \bfx_t = \bfy_{t-1}\}$, one can compute the following unbiased estimator of $\pi(f)$:
\begin{equation}\label{eq:def:unbiasedsnis}
  \FNu = \FN(\bfx_0) + \sum_{t=1}^{\tau - 1} \{\FN(\bfx_t) - \FN(\bfy_{t-1})\},
\end{equation}
where $\FN$ is as in \eqref{eq:def:Fhat}.
By convention the sum in \eqref{eq:def:unbiasedsnis} is zero in the event $\{\tau = 1\}$, and it is also equal to the infinite sum $\sum_{t=1}^{\infty} \{\FN(\bfx_t) - \FN(\bfy_{t-1})\}$ since $\FN(\bfx_t) = \FN(\bfy_{t-1})$ from time $\tau$ onward.
The lack of bias can be seen via a telescopic sum argument, since $\bfx_t$ and $\bfy_t$ have the same marginal distribution for all $t$, provided that limit and expectation can be swapped, and using \eqref{eq:unbiasedness-pimh_limit}.
Since $\bfx_0\sim \bar{q}$, $\hat{F}(\bfx_0)$ is the (biased) SNIS estimator,
so that the sum in \eqref{eq:def:unbiasedsnis} can be seen as a bias cancellation term, under some conditions. 
\citet{middleton2019unbiased} consider the case where $\omega$ is uniformly upper bounded, and show that \eqref{eq:def:unbiasedsnis} can have a finite variance. Below we work under much weaker assumptions and bound the moments of these unbiased estimators,
highlighting the effect of the tuning parameter $N$, and comparing its efficiency with that of SNIS.

\begin{remark}\label{rem:uis}
  The estimator $\FNu$ in \eqref{eq:def:unbiasedsnis} is an unbiased MCMC estimator whose efficiency is usually compared to that of the underlying MCMC algorithm \citep{joa2020,atchade2024unbiased}, i.e. PIMH in the present case. However, it can also be compared to that of SNIS as $N\to\infty$, 
  which is more compelling, as the asymptotic variance of SNIS is smaller than that of IMH \eqref{eq:imh_asymvar_greater}.
\end{remark}

\subsection{Moments of Coupling-UIS}

We subtract $\pi(f)$ from both sides in \eqref{eq:def:unbiasedsnis} to obtain
\begin{equation}\label{eq:usnis:errordecomp}
  \FNu - \pi(f) = \FN(\bfx_0)  - \pi(f) + \sum_{t=1}^{\infty} \{\FN(\bfx_t) - \FN(\bfy_{t-1})\}\mathds{1}(\tau > t).
\end{equation}
We introduce the notation 
\begin{align}\label{eq:usnis:errordecomp2}
  \Delta_t &= \FN(\bfx_t) - \FN(\bfy_{t-1}), \quad \text{BC} = \sum_{t=1}^{\infty}  \Delta_t \mathds{1}(\tau > t),
\end{align}
where BC stands for the bias cancellation term. Using Minkowski's inequality, the moments of the error of $\FNu$ can be 
bounded by the moments of the error of the IS estimator $\FN(\bfx_0)$, as in Theorem~\ref{thm: snis unbounded convergence}, 
and the moments of BC.

A first result is that, for bounded test functions $f$, $\FNu$ has as many moments as the meeting time $\tau$, 
which is up to $p$ (non-included) under Assumption~\ref{asm:pfinitemoments}.

\begin{prop}\label{prop:unbiasedsnis_has_finitemoments}
  Assume that $|f|_{\infty}\le1$ and let $s\geq 1$. If Assumptions~\ref{asm:abscontinuitypositivity}-\ref{asm:pfinitemoments} hold with $p> 1$ and $p>s$, then $\tau$ has $s$ finite moments, and the unbiased self-normalized importance sampling (Coupling-UIS) estimator $\FNu$ in \eqref{eq:def:unbiasedsnis} has $s$ finite moments for any $N\geq 1$.
\end{prop}

To deal with unbounded test functions, we need to control the moments of the terms $\Delta_t$.
For this we derive the following result about PIMH at any iteration $t$, under the same conditions as Theorem~\ref{thm: snis unbounded convergence}. Here and in the rest of this section we focus on the case $p\geq 2$ for simplicity, but the result could be extended to $p>1$ with a different dependency on $N$.

\begin{prop}\label{prop:pimh}
  Assume that there exist $p \in [2,\infty)$ and $r \in [2,\infty]$ such that $q(\omega^p) < \infty$ and $q(|f|^r)<\infty$, and $q(f^2\cdot \omega^2)<\infty$. 
  Let $(\bfx_t)$ be the PIMH chain started from $\bfx_0\sim \bar{q}$.
  Then, for any $2\leq s \leq pr/(p+r+2)$ and any $N\geq 1$, there exists $C$ such that for all $t\geq 0$:
  \begin{align*}
    \bE_{\bfx_0 \sim \bar{q}}\left[\left\lvert \FN(\bfx_t) - \pi(f)\right\rvert^s \right]\leq C N^{-s/2},
  \end{align*}
  where the constant $C$ depends on $r, p, s, q(f^r), q(\omega^p), q(f^2\cdot \omega^2)$.
  When $ r = \infty $, the statement holds for $f$ such that $|f|_\infty < \infty$ and all $s \leq p $.
\end{prop}

By Minkowski's inequality, under the conditions of Proposition~\ref{prop:pimh}, the moments of 
$\Delta_t$ have similar bounds. This can be used to obtain the following result.

\begin{prop}\label{prop:unbiasedsnis_has_finitemoments_unbounded}
  Assume that there exist $p \in (2,\infty)$ and $r \in [2,\infty]$ such that $q(\omega^p) < \infty$ and $q(|f|^r) <\infty$, and $q(f^2 \cdot \omega^2)<\infty$, 
  then for any $2 \leq s < p$ such that $pr/(p+r+2) > ps / (p-s)$, and for any $N\geq 1$, 
  the unbiased importance sampling (Coupling-UIS) estimator satisfies:
  \begin{align*}
    \bE\left[\left\lvert \FNu - \pi(f)\right\rvert^s \right]\leq C N^{-s/2},
  \end{align*}
  where the constant $C$ depends on $r, p, s, q(|f|^r), q(\omega^p), q(f^2 \cdot \omega^2)$. 
  When $r = \infty $, the statement holds for $f$ such that $|f|_\infty < \infty$ and all $s < p$ such that $p > sp/(p-s)$.
\end{prop}

\begin{remark} The construction of Coupling-UIS in \eqref{eq:def:unbiasedsnis}
  can be generalized to the case of i.i.d. pairs of variables $(\Omega_n,F_n)$ 
  that jointly
  follow a distribution $p$ on $\mathbb{R}_+\times \mathbb{R}$. In that case a similar
  construction allows the unbiased estimation of $\mathbb{E}[\Omega \cdot F]/\mathbb{E}[\Omega]$.
  As noted in \citet[][Section 4]{blanchet2015unbiasedmultilevel}, this includes steady-state regenerative simulation.
\end{remark}

\subsection{The asymptotic price of bias removal\label{subsec:dsnis:mse}}

We consider the price of debiasing self-normalized importance sampling using Coupling-UIS in terms of inefficiency, 
defined by the mean squared error multiplied by the average cost \citep[see e.g.][]{glynn1992asymptotic}. We start by comparing the mean squared errors
of unbiased and regular IS.
From \eqref{eq:usnis:errordecomp}, squaring and using Cauchy--Schwarz, we obtain
\begin{align}\label{eq:MSE_unbiasedSNIS_decompCS}
  &\left|\mathbb{E}\left[(\FNu - \pi(f) )^2\right] - \mathbb{E}\left[(\FN(\bfx_0) - \pi(f) )^2\right]\right|\nonumber\\
  &\leq   2
  \sqrt{ \mathbb{E}\left[(\FN(\bfx_0) - \pi(f) )^2\right]\cdot \mathbb{E}[\text{BC}^2]} + \mathbb{E}[\text{BC}^2],
\end{align}
with $\text{BC} =\sum_{t=1}^{\tau - 1} \{\FN(\bfx_t) - \FN(\bfy_{t-1})\}$.

The mean squared error (MSE) of IS, which is the term $\mathbb{E}[(\FN(\bfx_0) - \pi(f) )^2]$, is of order $N^{-1}$ under conditions stated in Theorem~\ref{thm: snis unbounded convergence}.
If we can bound $ \mathbb{E}[\text{BC}^2]$ by a term that decreases faster than $N^{-1}$, then the MSE of Coupling-UIS would be asymptotically equivalent to that of IS.
Intuitively, the bias cancellation term BC in \eqref{eq:usnis:errordecomp2} goes to zero for two reasons: first because $\tau$ goes to one as $N\to\infty$, since two independent draws of $\ZN(\bfx)$ are more likely to take similar values as $N$ increases.
Furthermore, each term $\Delta_t=\FN(\bfx_t) - \FN(\bfy_{t-1})$ goes to $0$ as $N\to\infty$, as described in Proposition~\ref{prop:pimh}.
We obtain the following result.

\begin{prop}\label{prop:dsnis:mseequiv}
  Let $\FNu$ be the Coupling-UIS estimator \eqref{eq:def:unbiasedsnis} and $\FN(\bfx)$ with $\bfx\sim \bar{q}$ be the IS estimator. Suppose that the assumptions of Proposition~\ref{prop:unbiasedsnis_has_finitemoments_unbounded} are satisfied with $s=2$, that is: $p>2$ and $r> 2$ such that $q(\omega^p) < \infty$, $q(|f|^r) <\infty$, and $q(f^2 \cdot \omega^2)<\infty$, with $2p+4r+4 < r\cdot p$.
  Then the mean squared errors of $\FNu$ and $\FN(\bfx)$ are asymptotically equivalent:
  \[\lim_{N\to\infty} N\cdot \mathbb{E}\left[(\FNu - \pi(f))^2\right] = \lim_{N\to\infty} N\cdot \mathbb{E}_{\bar{q}}\left[(\FN(\bfx) - \pi(f))^2\right].\] 
\end{prop}

The assumption $2p+4r+4 <r\cdot p$ is for example satisfied if $r=\infty$ and $p=4+\epsilon$ with an arbitrary $\epsilon>0$,
or if $p=5$ and $r = 15$. However it cannot be satisfied with $p\leq 4$.

The cost of Coupling-UIS is that of running Algorithm~\ref{alg:coupledPIMHlag}. It starts with two draws from $\bar{q}$, 
i.e.\ $2N$ draws from $q$, and as many evaluations of the weight function $\omega$. Then either $\tau = 1$ or the algorithm enters its while loop up to the meeting time $\tau$, drawing $N$ new particles at each iterate of the loop. Counting the cost in units of number of evaluations of $\pi$, 
Coupling-UIS has an overall cost of $\mathcal{C} = 2N + N (\tau - 1)$. 
If Assumption~\ref{asm:pfinitemoments} holds with $p \geq 2$, using Proposition~\ref{prop:upb_meetingproba} then $\mathbb{E}[\tau - 1] \leq C N^{-1/2}$ for a finite constant $C$.
Thus, as $N\to\infty$, $\mathbb{E}[\mathcal{C}]$ is equivalent to $2N$. We can then compare the asymptotic inefficiencies of Coupling-UIS and IS as in the next statement. 

\begin{prop}\label{prop:dsnis:meantau}
  Denote the cost of the Coupling-UIS estimator $\FNu$ in \eqref{eq:def:unbiasedsnis} by $\mathcal{C} = 2N + N (\tau - 1)$. Under the conditions of Proposition~\ref{prop:dsnis:mseequiv}, the inefficiency, defined as expected cost multiplied by mean squared error, of Coupling-UIS is twice that of IS as $N\to\infty$.
\end{prop}

The reason for the efficiency loss in Coupling-UIS is that, in Algorithm~\ref{alg:coupledPIMHlag} the $N$ particles in $\bfy_0$ 
are required to determine $\tau$ but in the event $\{\tau = 1\}$, which is increasingly likely as $N\to\infty$, these $N$ particles do not participate directly in the estimator $\FNu$.

\subsection{An improved coupling-based unbiased importance sampling estimator\label{subsec:suis}}

A simple trick provides a remedy, and cuts the asymptotic inefficiency by a half. We can view $\FNu$ as a deterministic function of 
initial states $\bfx_0$ and $\bfy_0$ drawn from $\bar{q}$ independently,
as well as additional variables: an arbitrary long sequence of proposals from $\bar{q}$, and a sequence of uniform random variables used to accept or reject these proposals. Denote these additional variables by $\zeta$. Then $\FNu$ can be
written as $\mathcal{A}(\bfx_0,\bfy_0,\zeta)$, where $\mathcal{A}$ is a deterministic function, described by Algorithm~\ref{alg:coupledPIMHlag} and \eqref{eq:def:unbiasedsnis}.
Then we define the estimator
that averages over swapping $\bfx_0$ and $\bfy_0$:
\begin{equation}\label{eq:def:unbiasedsnis:symmetrized}
  \FNusym = \frac{1}{2}\left(\mathcal{A}(\bfx_0,\bfy_0,\zeta) + \mathcal{A}(\bfy_0,\bfx_0,\zeta)\right).
\end{equation}
Computing \eqref{eq:def:unbiasedsnis:symmetrized} only requires simple modifications of Algorithm~\ref{alg:coupledPIMHlag}.
Indeed, either $\ZN(\bfx_0)\geq \ZN(\bfy_0)$ or $\ZN(\bfx_0)<\ZN(\bfy_0)$.
In the first case, we always have $\mathcal{A}(\bfy_0,\bfx_0,\zeta) = \FN(\bfy_0)$, and $\mathcal{A}(\bfx_0,\bfy_0,\zeta)$ can be computed
following Algorithm~\ref{alg:coupledPIMHlag} and \eqref{eq:def:unbiasedsnis}. 
In the second case, we always have $\mathcal{A}(\bfx_0,\bfy_0,\zeta)= \FN(\bfx_0)$, 
and $\mathcal{A}(\bfy_0,\bfx_0,\zeta)$ can be computed following Algorithm~\ref{alg:coupledPIMHlag} and \eqref{eq:def:unbiasedsnis} with the role of $\bfx_0$ and $\bfy_0$ swapped.
The estimator \eqref{eq:def:unbiasedsnis:symmetrized} amounts to a \emph{Rao--Blackwellized} version of $\FNu$ in \eqref{eq:def:unbiasedsnis}, averaging over the arbitrary specification of which of the two draws from $\bar{q}$ is used as $\bfx_0$ or as $\bfy_0$.
The following statement is a mild variation of the previous results and is stated without a proof. 

\begin{prop}\label{prop:suisinefficiency}
  Consider the Coupling-based UIS estimator $\FNusym$ in \eqref{eq:def:unbiasedsnis:symmetrized}, with cost $\tilde{\mathcal{C}}$.
  Suppose that the conditions of Proposition~\ref{prop:dsnis:mseequiv} are satisfied.
  Then $\FNusym$ is an unbiased estimator of $\pi(f)$ for any $N\geq 1$, with finite expected cost and
  finite variance, and its inefficiency is equivalent to that of IS as $N\to\infty$.
\end{prop}

The estimator $\FNusym$ should be preferred to $\FNu$, as it is asymptotically twice as efficient,
but note that its cost is always larger.
Algorithm~\ref{alg:suis} provides pseudo-code for the Coupling-UIS estimator, which can be implemented under the same conditions as Algorithm~\ref{alg:snis} (SNIS) or Algorithm~\ref{alg:PIMH} (PIMH). 
We recall that $\ZN$ is defined in \eqref{eq:def:Zhat}, $\FN$ in \eqref{eq:def:Fhat},
and
$\bar{q}$ in \eqref{eq:pimhproposal}.
The only tuning parameter is $N$, which affects
the cost per estimate and its variance.
The pseudo-code also implements a Rao--Blackwellization argument \citep{casella:robert:1996}
that replaces each $\FN(\bfx_t)$ by 
\begin{align*}
  &\left(1 - \alpha(\bfx_{t-1},\bfxstar)\right) \FN(\bfx_{t-1}) + \alpha(\bfx_{t-1},\bfxstar) \FN(\bfxstar),
\end{align*}
where $\alpha(\bfx_{t-1},\bfxstar) = \min(1, {\ZN(\bfxstar)}/{\ZN(\bfx_{t-1})})$ is the acceptance probability of the PIMH step, as in \eqref{eq:def:pimh-acceptanceratio}.

\begin{algorithm}[ht]
  \caption{Coupling-based unbiased importance sampling (Coupling-UIS).}
  \label{alg:suis}
  \begin{enumerate}
    \item Sample \(\bfx = (\x_{1}, \ldots, \x_{N})\sim \bar{q}\)
    and \(\bfy = (\y_{1}, \ldots, \y_{N})\sim \bar{q}\) independently. 
    \item 
    If $\ZN(\bfx) < \ZN(\bfy)$, swap $\bfx$ and $\bfy$, so that
    $\ZN(\bfx)\geq\ZN(\bfy)$ henceforth.    
    \item Set $\tau\leftarrow 1$.
    Initialize
    \[
    \FNusym\leftarrow
    \frac{1}{2}\FN(\bfx)+\frac{1}{2}\FN(\bfy).
    \]
    \item 
    Update estimate using Rao--Blackwellization over the acceptance step:
    \[
    \FNusym\leftarrow \FNusym + \frac{1}{2} \left(1-\alpha(\bfx,\bfy)\right)
    \left(
      \FN(\bfx) 
      - \FN(\bfy)\right).
    \]
    Draw $U\sim\mathrm{Uniform}(0,1)$. If $U < \ZN(\bfy)/\ZN(\bfx)$, go to step~6.
    \item Iterate until both conditions in \textup{e)} and \textup{f)} are satisfied:
    \begin{enumerate}
      \item $\tau\leftarrow\tau+1$.
      \item Sample $\bfxstar = (\x^\star_{1}, \ldots, \x^\star_{N})\sim \bar{q}$.
      \item Update estimate using Rao--Blackwellization over the acceptance step:
      \begin{align*}\FNusym\leftarrow \FNusym
      +\frac{1}{2}&\left((1-\alpha(\bfx,\bfxstar))\hat{F}_N(\bfx) 
        -(1-\alpha(\bfy,\bfxstar))\hat{F}_N(\bfy) \right.\\
        &+ \left. (\alpha(\bfx,\bfxstar)- \alpha(\bfy,\bfxstar))\hat{F}_N(\bfxstar)
        \right).
      \end{align*}
      \item Draw $U\sim\mathrm{Uniform}(0,1)$.
      \item If $U < \ZN(\bfxstar)/\ZN(\bfx)$,
      set $\bfx\leftarrow\bfxstar$.
      \item If $U < \ZN(\bfxstar)/\ZN(\bfy)$,
      set $\bfy\leftarrow\bfxstar$.
    \end{enumerate}
    \item Set
    $\mathrm{cost}\leftarrow N(\tau+1)$.
    \item Return $\FNusym$.
  \end{enumerate}
\end{algorithm}

\subsection{Unbiased estimation of the inverse of the normalizing constant\label{subsec:inverse_Z}}

A particular case is the unbiased estimation of $1/Z$, where $Z = q(\omega_u)$ for an unnormalized weight function $\omega_u$ proportional to $\omega$. This problem was part of the motivation for \citet{wang2023unbiased,chopin_crucinio_singh_2025}, where such unbiased estimators are used for inference in doubly-intractable models.
Coupling-UIS can be directly used for this purpose by setting the test function as $f = 1/\omega_u$. In this case, IS gives $\FN(\bfx) = 1/\ZN(\bfx)$, which is biased; here we redefine $\ZN(\bfx)$ as $N^{-1} \sum_{n=1}^N \omega_u(x_n)$. Our previous results
on Coupling-UIS apply
with the choice of test function $f = 1/\omega_u$, but we obtain weaker conditions by leveraging, for example, the identity $\omega_u \cdot f = 1$. Thus, we specialize Propositions~\ref{prop:unbiasedsnis_has_finitemoments},
\ref{prop:unbiasedsnis_has_finitemoments_unbounded} and \ref{prop:dsnis:mseequiv},
for that test function in Proposition~\ref{prop:inverse-CUIS}.

\begin{prop}[Properties of Coupling-UIS for $1/Z$]\label{prop:inverse-CUIS}
  Denote by $\hat{I}_N$ the Coupling-UIS estimator~\eqref{eq:def:unbiasedsnis} with $f = 1/\omega_u$.
  Assume that $q(\omega^p) < \infty$ for some $p \geq 2$ and $q(\omega^{-\eta}) < \infty$ for some $\eta > 0$.
  Then for any $1 \leq s < p$ and $N \geq sp/((p-s)\eta)$, $\hat{I}_N$ has $s$ finite moments:
  \begin{equation}\label{eq:uis-inverse-finitemoments}
    \bE\left[|\hat{I}_N - Z^{-1}|^s\right] < \infty.
  \end{equation}
  If furthermore $p > 2s$, then for $N \geq sp/((p-s)\eta)$ there exists $C > 0$ such that
  \begin{equation}\label{eq:uis-inverse-rate}
    \bE\left[|\hat{I}_N - Z^{-1}|^s\right] \leq C\, N^{-s/2}.
  \end{equation}
  For $s = 2$: $\hat{I}_N$ has finite variance for $p > 2$ and $N > 2p/((p-2)\eta)$; if $p > 4$, then
  \[\lim_{N\to\infty} N \cdot \bE\left[(\hat{I}_N - Z^{-1})^2\right] = \lim_{N\to\infty} N \cdot \bE_{\bar{q}}\left[\left(\ZN(\bfx)^{-1} - Z^{-1}\right)^{\!2}\right].\]
\end{prop}

The distinctive feature of Coupling-UIS
is the last item of the proposition, i.e. the asymptotic equivalence of the MSE of $\hat{I}_N$ and that of the SNIS estimator $1/\ZN(\bfx)$, which is the natural estimator of $1/Z$. We will see in experiments that the condition $p>4$ is not necessary.

\begin{remark}\label{rmk:inverse-comparison}
  The general Coupling-UIS moment bound (Proposition~\ref{prop:unbiasedsnis_has_finitemoments_unbounded}) specialized to $f = 1/\omega_u$ requires $q(\omega^{-r}) < \infty$ for some $r \geq 2$ together with the compatibility condition $pr/(p+r+2) > ps/(p-s)$. Proposition~\ref{prop:inverse-CUIS} replaces this by $q(\omega^{-\eta}) < \infty$ for any $\eta > 0$ and $p > 2s$, with no compatibility condition between $\eta$ and $p$.
  When $s = 2$, the general compatibility condition reduces to $2p + 4r + 4 < rp$, i.e.\ $r > (2p+4)/(p-4)$, which forces $p > 4$. Both results thus require $p > 4$, but they differ on the assumption required on $\omega^{-1}$. The general bound demands $q(\omega^{-r}) < \infty$ with $r > (2p+4)/(p-4)$, which is very restrictive for moderate $p$: for instance, $p = 5$ requires $r > 14$, $p = 6$ requires $r > 8$. Proposition~\ref{prop:inverse-CUIS}, on the other hand, requires only $p > 4$ and $q(\omega^{-\eta}) < \infty$ for any $\eta > 0$.
\end{remark}

\section{Applications\label{sec:applications}}

\subsection{Unbiased estimation of \texorpdfstring{$1/Z$}{1/Z}}\label{subsec:applications:inverse_Z}
We consider the unbiased estimation of $1/Z$ as in Section~\ref{subsec:inverse_Z}. As described in \citet[Section 4.3,][]{chopin_crucinio_singh_2025}
the problem arises for example in likelihood-based inference for exponential random graphs.
We compare Coupling-UIS with MLMC-UIS \citep{blanchet2015unbiasedmultilevel,ShiCornish} and Taylor-UIS
\citep{blanchet2015unbiasedtaylor,chopin_crucinio_singh_2025},
fully described in Appendix~\ref{app:other_uis}.

We consider Example~\ref{example:expo},
with $p = 3$. The importance weight has strictly less than $3$ finite moments. The value of $Z$ is $1$.
For SNIS and Coupling-UIS, the unique tuning parameter is the number of samples $N$, which we vary from $4$ to $128$. 
For MLMC-UIS, we select a Geometric parameter based on the number of moments of $\omega$ under $q$. We vary the initial level of the multi-level construction from $1$ to $5$, and we report both the ``single sample'' (SS) and the ``Russian roulette'' (RR) estimates \citep{ShiCornish}. For Taylor-UIS, the input estimators of $Z$ are based on $N$ samples, with $N$ ranging from $4$ to $128$. For each value of $N$, we tune Taylor-UIS following the guidelines of \citet{chopin_crucinio_singh_2025} and we do not count the cost of the tuning step, which is small.
The results are shown in Figure~\ref{fig:inverseZ_inefficiency}. 
The inefficiency is defined as the product of the variance of the estimator and its expected cost, where the cost is measured in number of evaluations of $\omega$ per estimator. The results are based on $250,000$ independent repeats of each method.

Firstly, we note that the inefficiencies of Coupling-UIS and SNIS converge to the same limit as $N$ increases, as predicted by Propositions~\ref{prop:suisinefficiency} and \ref{prop:inverse-CUIS}, 
even though the setting violates the assumption $p > 4$ in these propositions.
Secondly, Coupling-UIS provides a smaller inefficiency than 
the other unbiased estimators of $1/Z$, for all tuning parameters. This, combined with the fact that Coupling-UIS only requires a choice of $N$ as tuning parameter, suggests that Coupling-UIS is a competitive method for unbiased estimation of $1/Z$.

\begin{figure}
  \centering
  \includegraphics[width=.85\textwidth]{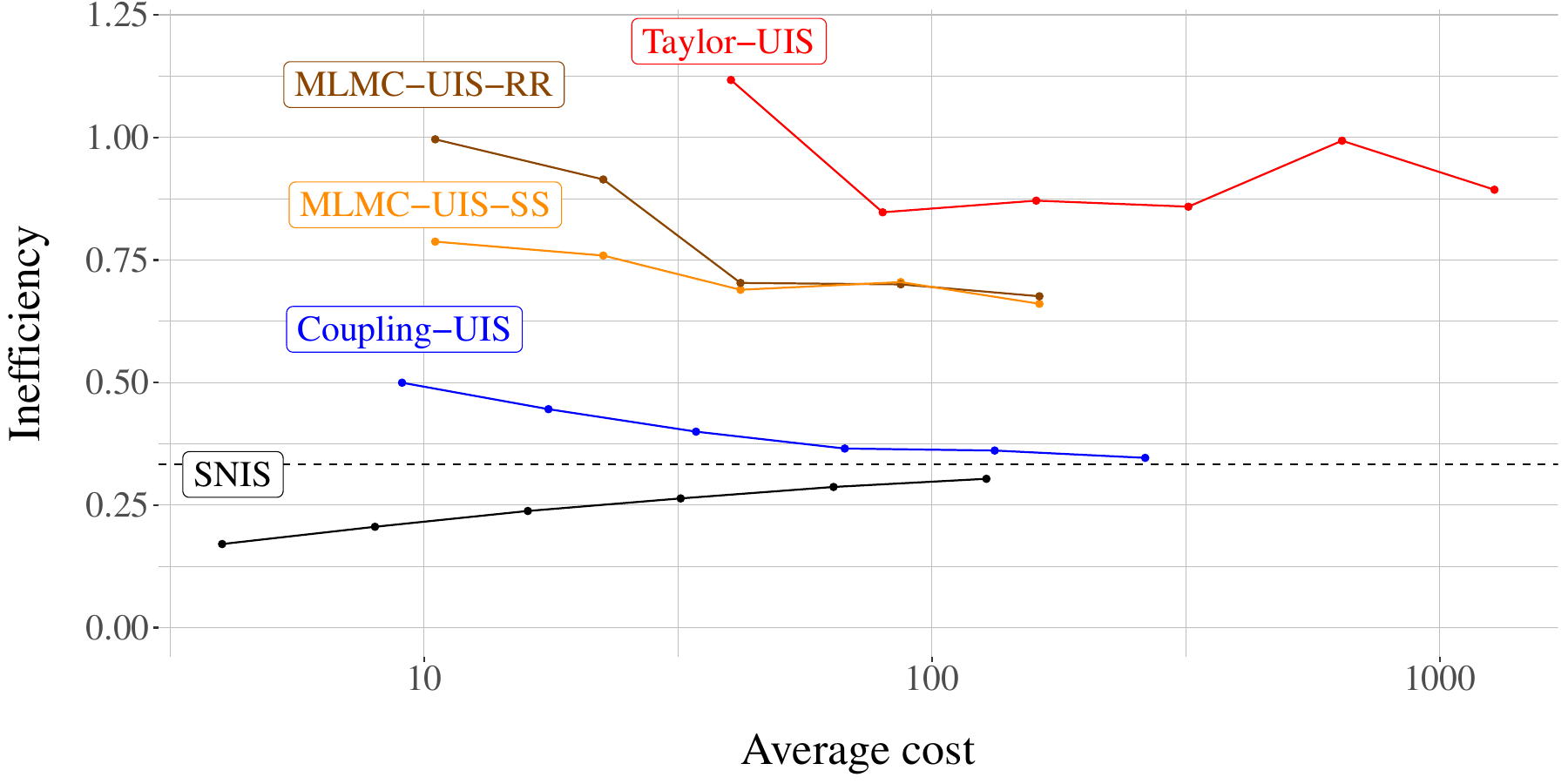}
  \caption{\label{fig:inverseZ_inefficiency} Inefficiency  against average cost, for different estimators of $1/Z$.
  For each method, we vary a tuning parameter that impacts both the inefficiency and the cost. For SNIS and Coupling-UIS the tuning parameter is the number of samples $N$; see Section~\ref{subsec:applications:inverse_Z}. The dashed line represents the inefficiency of SNIS as $N\to\infty$.}
\end{figure}

\subsection{Nested expectations\label{subsec:applications:nestedexpectations}}

We consider a nested expectation problem, as defined in Subsection~\ref{subsec:nested_expectations_bias}. The setting is that of Bayesian inference after multiple imputation of missing data \citep[e.g.][]{zhou2010note}.
We consider a logistic regression 
where the covariates  $X$ are continuous and contain missing entries. We generate the data by sampling $X\in \mathbb{R}^{n \times d}$ independently from a Normal$(0,1)$ distribution,
with $n=100$ and $d=10$. The response $Y$ is generated  given $X$ from a logistic regression with parameters $\beta^\star$ defined as $\beta^\star_i = i / 10$ for $i=1,\ldots,5$, and $\beta^\star_i = 0$ for $i=6,\ldots,10$. We obtain $\sum_{i=1}^n Y_i = 52$. We then remove some entries from $X$, with $X_{ij}$ becoming missing with probability $0.4$ if $Y_i = 1$ and $0.2$ if $Y_i = 0$, independently across $i$ and $j$.
As a result, $285$ entries are missing out of $1000$. We use the correct Normal$(0,1)$ distribution as the imputation model for each missing entry.

For each completed data set, we approximate the posterior distribution in the logistic regression model, using the Normal$(0, 10 \cdot I_d)$ distribution
as prior on the coefficients. We use importance sampling using as proposal 
a multivariate Student distribution with $10$ degrees of freedom, centered at the MLE, and with covariance matrix equal to the inverse of minus the Hessian of the log-posterior density at the MLE. We consider three variants of importance sampling: SNIS, with bias of order $N^{-1}$, the variant of \citet{skare2003improved}, referred to as SNIS-LOO,
where each importance weight is divided by the sum of the other $N-1$ importance weights, with bias of order $N^{-2}$, 
and finally the Coupling-UIS estimator from Section~\ref{subsec:suis}, with no bias.
We consider the test function that maps the parameter to its first component, 
so that we are estimating the posterior mean of the first coefficient, 
averaged over the distribution of the imputed data. For a total budget $C$, 
we need to allocate the budget between the number of imputed data sets $M$ and the number of samples $N$ per estimator for the inner expectation, so that $C = M N$.
We follow the optimal allocation from \citet{andradottir2016computing}:
for SNIS we use $M = C^{2/3}$ and $N = C^{1/3}$. For SNIS-LOO, we use $M = C^{4/5}$ and $N = C^{1/5}$. For Coupling-UIS, we use a fixed $N = 5$, and $M = C / 5$. In the latter case the cost is random and we report its average.

Figure~\ref{fig:nested_expectations} displays the MSE of the three methods as a function of the computing budget $C$. The results are based on $200$ independent repeats. The MSE is computed using an approximate ground truth set to the average of the $200$ estimators obtained with Coupling-UIS at the highest budget of $M = 32000$ and $N = 5$. From those $200$ estimators we compute a 95\% confidence interval for the quantity of interest and obtain $[0.2286, 0.2289]$. The figure shows regression lines corresponding to log-mean squared error over log-average cost. We obtain a slope of $-0.65$ for SNIS, $-0.83$ for SNIS-LOO and $-1.06$ for Coupling-UIS, matching closely the theoretical rates of $-2/3$, $-4/5$ and $-1$, respectively. The figure also shows that, for a total budget between $10^4$ and $3\times 10^5$, nested expectation estimators using Coupling-UIS outperform those using SNIS 
throughout the range, and start outperforming those using SNIS-LOO for the larger budgets.

\begin{figure}
  \centering
  \includegraphics[width=.85\textwidth]{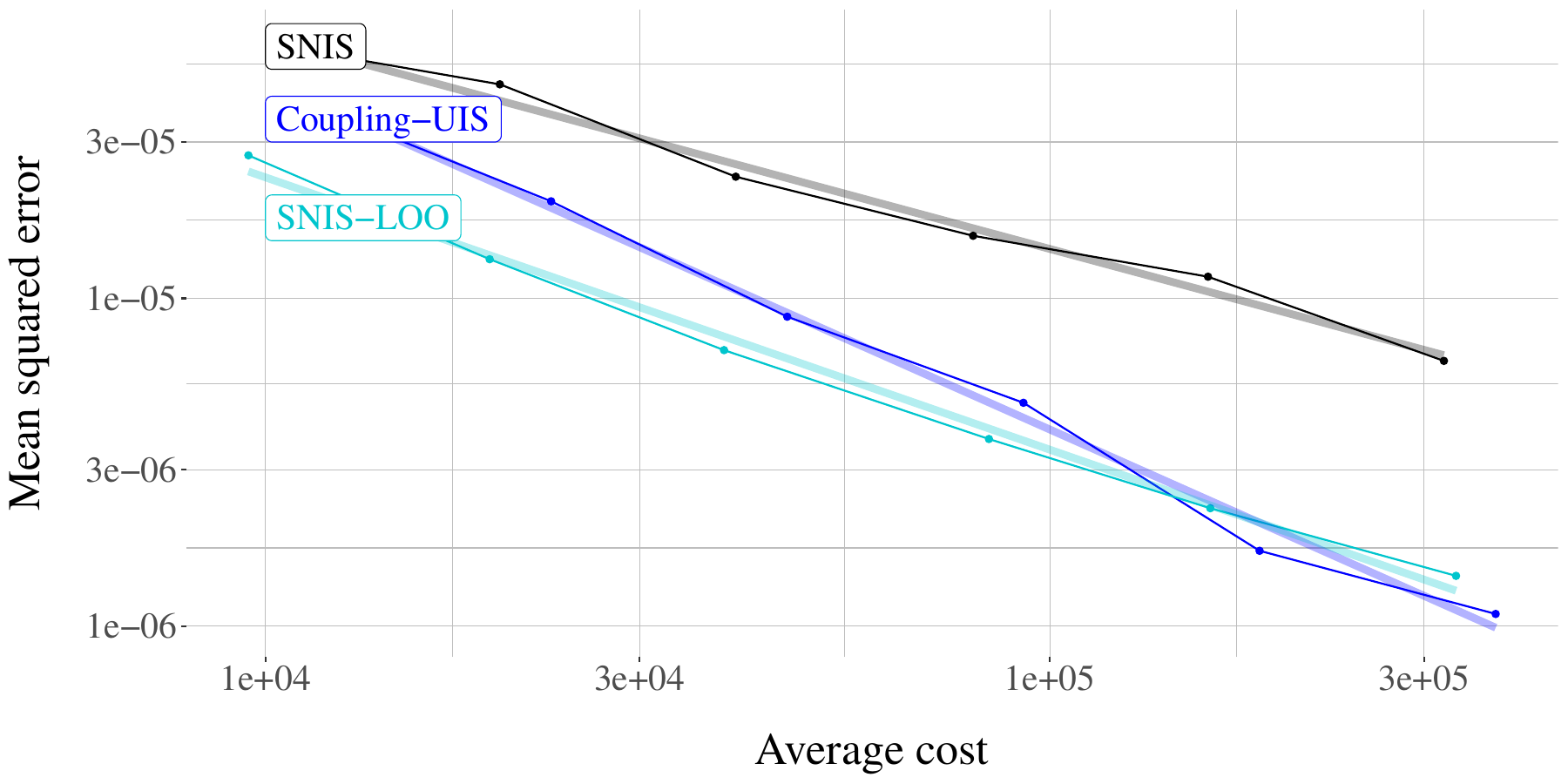}
  \caption{\label{fig:nested_expectations} Mean-squared error of nested expectation estimators, in the context of Bayesian inference after multiple imputation, using either SNIS, SNIS-LOO or Coupling-UIS for the inner expectation, as a function of the computing budget $C$. The dashed lines are regression lines fitted on $\log$-MSE $\sim$ $\log C$.}
\end{figure}

\subsection{Robust mean estimation\label{subsec:robust_mean_estimation}}

A motivation for eliminating the bias is that
one can then employ any statistical technique that requires unbiased inputs,
without requiring additional assumptions or modifications.
An example is robust mean estimation \citep{lugosi2019mean}. 
The goal is to estimate the mean $\mu$ of a random variable with finite variance $\sigma^2$ from $n$ i.i.d.\ copies $X_1,\ldots,X_n$: for a confidence level $\delta\in(0,1)$, we seek $\hat{\mu}_n = \hat{\mu}_n(X_1,\ldots,X_n,\delta)$ such that
the inequality
\begin{equation}\label{eq:robust_goal}
  \bP(|\hat{\mu}_n - \mu| > \epsilon) \leq \delta
\end{equation}
holds for the smallest possible $\epsilon = \epsilon(n,\delta)$. The empirical average satisfies~\eqref{eq:robust_goal} asymptotically with $\epsilon = \sigma\sqrt{2\log(2/\delta)/n}$ via the CLT. On the other hand, Chebyshev's inequality only yields  $\epsilon = \sigma\sqrt{1/(n\delta)}$, with a $\delta$-dependence that cannot be improved under finite variance alone \citep[Section~2]{lugosi2019mean}. Remarkably, implementable alternatives recover the sub-Gaussian rate. 
The Median-of-Means (MoM, \citet{nemirovskij1983problem}) estimator provides a prime example, where the draws $X_1,\ldots,X_n$ are 
partitioned into $K$ blocks of size $m$ with $n=mK$, and the estimator is obtained as the median of the blockwise means. The following result 
shows that the MoM estimator achieves sub-Gaussian performance.

\begin{thm}[Theorem $2$ in~\cite{lugosi2019mean}]\label{thm:momlugosimendelson}
Let $X_1, \ldots, X_n$ be i.i.d. random variables with mean $\mu\in\mathbb{R}$ and variance $\sigma^2\in (0,\infty)$. Let $\delta \in (0,1)$ and 
$K= \lceil 8 \log(1/\delta) \rceil$ be the number of blocks, and assume that $n = mK$ for some positive integer $m$.
Denote by $\hat{\mu}_{\textsc{MoM},n}$ the MoM estimator computed with $K$ blocks of size $m$. Then, with probability at most $\delta$,
for all $n\geq \lceil 8\log(1/\delta)\rceil$,
\[
    |\hat{\mu}_{\textsc{MoM},n} - \mu| > \sigma\sqrt{\frac{32 \log(1/\delta)}{n}}.
\]
\end{thm}

The combination of MoM with self-normalized importance sampling (SNIS) was explored in \citet{Dau2022}. The proposed estimator, termed MoM-SNIS, is the median of $K$ SNIS estimators obtained with $M$ draws each, so that the cost is equivalent to SNIS with $N=M\times K$ draws. The above result on MoM does not directly apply because SNIS is biased for the quantity of interest. \citet[Proposition 2,][]{Dau2022} shows that
MoM-SNIS achieves sub-Gaussian performance under minimal assumptions on the weight and test function, but only once $N$ exceeds a problem-dependent threshold involving the variance of the weights. 
Furthermore, as shown in \citet[][Proposition 3]{Dau2022}, it is not possible to obtain a sub-Gaussian concentration inequality for MoM-SNIS that would hold for all $N$ larger than a threshold that would depend on $\delta$ only.

In contrast, since Coupling-UIS estimators are unbiased,
we can directly use MoM with them, and obtain sub-Gaussian performance for all $n$ larger than $\lceil 8\log(1/\delta)\rceil$ as in Theorem~\ref{thm:momlugosimendelson},
where $n$ is the number of independent copies of Coupling-UIS.
We could also use any other robust mean estimator, such as those proposed in  \citet{minsker2021robust} or \citet{lee2022optimal}. Drawbacks relative to MoM-SNIS \citep{Dau2022} include the higher variance of Coupling-UIS relative to SNIS for matching costs, and the random nature of the computing cost.

We revisit Example~\ref{example:expo}, with $p = 2.1$, and the test function $f(x) = x$, 
so that $\pi(f) = 1$.
We compare the distribution of the errors of self-normalized importance sampling (SNIS)
with $N=4000$, with the MoM-SNIS method of \citet{Dau2022}, tuned with $\delta = 0.05$, leading to $K = \lceil 8 \log(1/\delta)\rceil = 24$ and $M=192$ so that $K\times M = 4608$ is comparable to $N$. 
For MoM-Coupling-UIS, we use $24$ samples for each unbiased estimator, 
and then we compute the median of $K = 24$ averages, each over $M = 3$
unbiased estimators per block. This leads to an average cost of $3856$
target density evaluations per MoM-Coupling-UIS estimator. We generate $10^6$ independent runs of SNIS with $4000$ samples, 
and then for MoM, of SNIS with $192$ samples and Coupling-UIS with $24$ samples. This allows us to collect 41,660 MoM-SNIS and 13,880 MoM-Coupling-UIS estimates.

Figure~\ref{fig:qtles_robust} displays the quantiles of the absolute value of the difference between the estimates and $\pi(f)$, 
for SNIS, MoM-SNIS, and MoM-Coupling-UIS. The median absolute error is smaller for SNIS than for the robust estimates. However, for quantile of very high order, 
the error of SNIS increases faster than the error of the robust methods. From the experiments we observe a small advantage of MoM-SNIS over MoM-Coupling-UIS at comparable costs, so the appeal of the latter appears mostly theoretical.

\begin{figure}
  \centering
  \includegraphics[width=.85\textwidth]{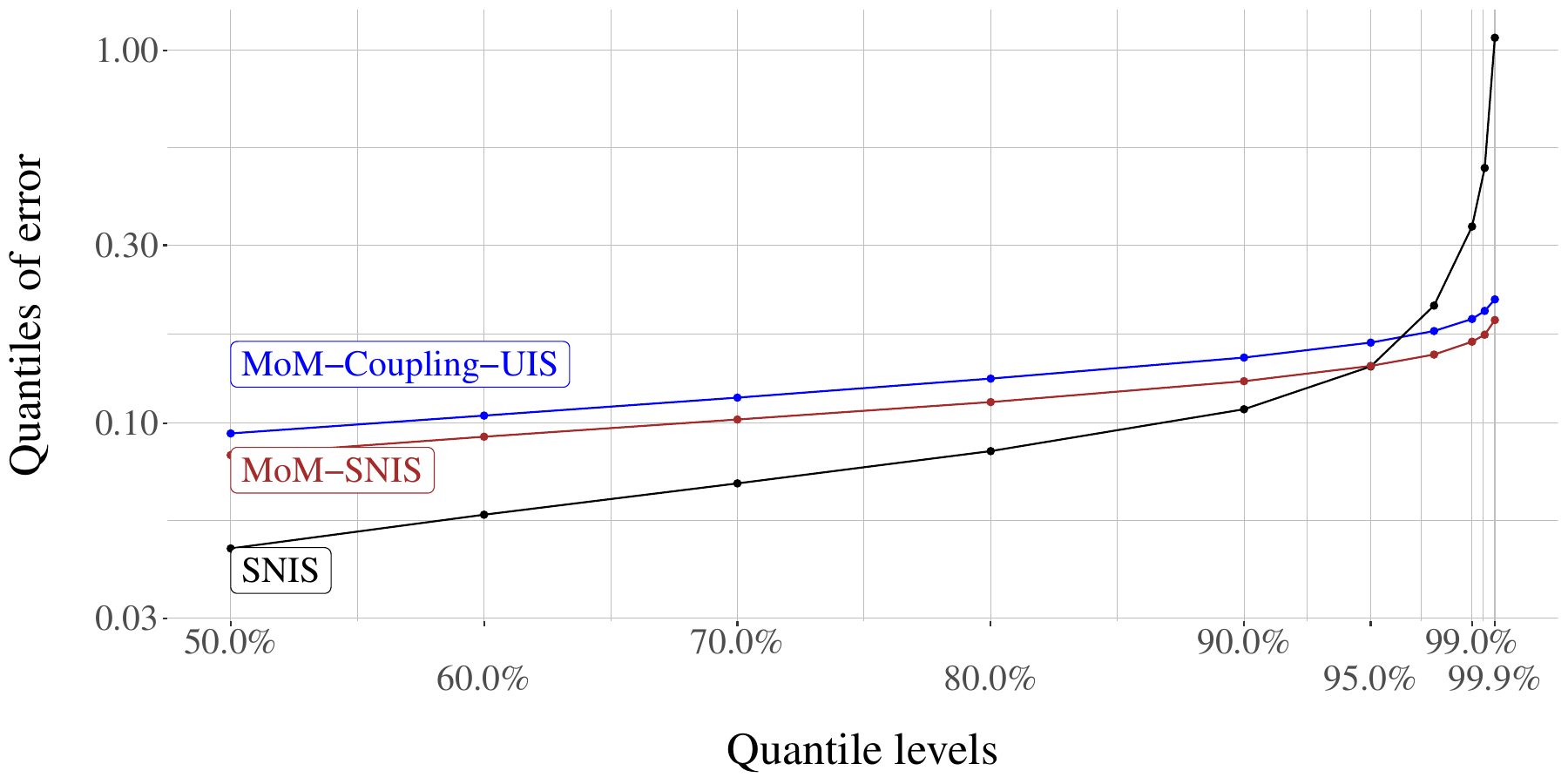}
  \caption{\label{fig:qtles_robust} Quantiles of the absolute error for SNIS, MoM-SNIS, MoM-Coupling-UIS for the Exponential example with $p=2.1$.}
\end{figure}

\paragraph*{Code.}
The code to reproduce the experiments is available at \url{https://github.com/pierrejacob/unbiasedsnis}.

\paragraph*{Acknowledgements.}
  The first author was supported by the Engineering and Physical Sciences Research Council through
  grant EP/Y018273/1.
  The second author was supported by ERC Consolidator Grant UMCMC 101171415
  and France 2030 via the CY Initiative Emergence. The fourth author acknowledges support from the National Science Foundation through grant DMS-2210849.
  The authors are thankful to Hai Dang Dau, Andrés Mugnier, Mohamed Ndaoud, 
  Christian P. Robert, Yanbo Tang, Kamélia Daudel and Nicolas Chopin for helpful comments, and to the reviewers for constructive feedback that improved the manuscript.

\newpage
\bibliographystyle{agsm}
\bibliography{refs}

\appendix

  \section{Proofs\label{appx:proofs}}
  
  \subsection{Proofs of Section~\ref{sec:intro}\label{appx:proofs:intro}}
  
  \begin{proof}[Proof of Proposition~\ref{prop:zhatconcentration}]
    It suffices to establish~\eqref{eq:prop:zhatconcentration:centred}, since~\eqref{eq:prop:zhatconcentration:uncentred} and~\eqref{eq:Z-tail-bound} then follow by the triangle inequality in $L^p$ and Markov's inequality, respectively.
    By the Marcinkiewicz--Zygmund-type inequality of~\citet{petrov1975sums} (Section III.5, item 16, p.~60), for $U_1,\ldots,U_N$ independent, mean-zero random variables with finite $p$-th moment, $p\geq 2$,
    \[\mathbb{E}\!\left[\left|\sum_{n=1}^N U_n\right|^p\right] \leq m(p)\, N^{p/2-1}\sum_{n=1}^N  \mathbb{E}[|U_n|^p],\]
    where $m(p)$ depends only on $p$ and satisfies $(m(p))^{1/p}\leq p-1$ \citep{RenLiang2001}.
    Setting $U_n = \omega(\x_n) - 1$, which has zero mean since $q(\omega)=1$, and considering the average $\ZN(\bfx)-1 = N^{-1}\sum_{n=1}^N U_n$, we obtain
    \begin{equation}\label{eq:petrov-applied}
      \mathbb{E}_{\bar{q}}\!\left[| \ZN(\bfx)-1|^p\right]\leq (p-1)^p\, N^{-p/2}\, q(|\omega-1|^p).
    \end{equation}
    The $C_p$-inequality gives $q(|\omega-1|^p)\leq 2^{p-1}(1+q(\omega^p))$, and since $q(\omega^p)\geq q(\omega)^p = 1$ by Jensen's inequality, we have $1+q(\omega^p)\leq 2\,q(\omega^p)$, so $q(|\omega-1|^p) \leq 2^p\, q(\omega^p)$.
    Substituting into \eqref{eq:petrov-applied} and using $p-1\leq p$ yields
    \[
    \mathbb{E}_{\bar{q}}\!\left[| \ZN(\bfx)-1|^p\right]
    \;\leq\; (2p)^p\, q(\omega^p)\, N^{-p/2}
    \;=\; M(p)^p\, N^{-p/2},
    \]
    and taking $p$-th roots gives \eqref{eq:prop:zhatconcentration:centred}.
  \end{proof}  
  
  We recall the following result on the $p$-th moment of sums of independent random variables with $1\leq p\leq 2$.

  \begin{lem}[von Bahr--Esseen inequality \citep{vonbahr_esseen_1965}]\label{lem:vbe_less2}
    Let $\zeta_1,\ldots,\zeta_N$ be independent random variables with $\bE[\zeta_i]=0$, and let $p\in[1,2]$. Then
    \[
    \bE\left[\left|\sum_{n=1}^N \zeta_n\right|^p\right] \leq 2 \sum_{n=1}^N \bE\left[|\zeta_n|^p\right].
    \]
  \end{lem}
  
The following proposition is a version of Proposition~\ref{prop:zhatconcentration} for $p\in(1,2)$.

  \begin{prop}\label{prop:zhatconcentration_less2}
    Under Assumption~\ref{asm:abscontinuitypositivity}, and assuming that $q(\omega^p)<\infty$ for some $p\in(1,2)$, then, for all $N\geq 1$,
    \begin{align}
      \mathbb{E}_{\bar{q}}\left[\left|\ZN(\bfx)-1\right|^p\right]
      &\leq \frac{2 q\left(|\omega-1|^p\right)}{N^{p-1}},\label{eq:prop:zhatconcentration:centred_less2}\\
      \mathbb{P}_{\bar{q}}\left(\ZN(\bfx)\geq 1+t\right)
      &\leq \frac{2 q\left(|\omega-1|^p\right)}{N^{p-1}t^p}, \qquad t>0.\label{eq:Z-tail-bound_less2}
    \end{align}
    If $\ZN(\bfxprime)$ is an independent copy of $\ZN(\bfx)$ under $\bar{q}$, then
    \begin{align}
      \mathbb{E}_{\bar{q}\otimes\bar{q}}\left[\left|\ZN(\bfx)-\ZN(\bfxprime)\right|^p\right]
      &\leq \frac{2^{p+1} q(\omega^p)}{N^{p-1}},\label{eq:Z-diff-bound_less2}\\
      \mathbb{E}_{\bar{q}\otimes\bar{q}}\left[\left|\ZN(\bfx)-\ZN(\bfxprime)\right|\right]
      &\leq 2^{(p+1)/p} q(\omega^p)^{1/p} N^{-(p-1)/p}.\label{eq:Z-diff-bound_one_less2}
    \end{align}
  \end{prop}

  \begin{proof}
    Let $\Omega_1,\ldots,\Omega_N$ be i.i.d. with the law of $\omega(X)$ where $X\sim q$, and write $\ZN=N^{-1}\sum_{n=1}^N \Omega_n$. Lemma~\ref{lem:vbe_less2} gives
    \[
    \bE\left[\left|\sum_{n=1}^N (\Omega_n-1)\right|^p\right]
    \leq 2 \sum_{n=1}^N \bE\left[|\Omega_n-1|^p\right]
    = 2 N q\left(|\omega-1|^p\right).
    \]
    Dividing by $N^p$ yields \eqref{eq:prop:zhatconcentration:centred_less2}. Markov's inequality then gives \eqref{eq:Z-tail-bound_less2}.
    
    For \eqref{eq:Z-diff-bound_less2}, let $\Omega_1',\ldots,\Omega_N'$ be an independent copy and set $\ZN'=N^{-1}\sum_{n=1}^N \Omega_n'$. Then
    \[
    \ZN-\ZN' = \frac{1}{N}\sum_{n=1}^N (\Omega_n-\Omega_n'),
    \]
    and the summands are independent and mean-zero. Applying Lemma~\ref{lem:vbe_less2} again,
    \[
    \bE\left[|\ZN-\ZN'|^p\right]
    \leq \frac{2}{N^{p-1}}\bE\left[|\Omega-\Omega'|^p\right].
    \]
    Since $|a-b|^p\leq 2^{p-1}(a^p+b^p)$ for $a,b\geq 0$, we obtain $\bE[|\Omega-\Omega'|^p]\leq 2^p q(\omega^p)$, which gives \eqref{eq:Z-diff-bound_less2}. Finally, \eqref{eq:Z-diff-bound_one_less2} follows from Lyapunov's inequality.
  \end{proof}

  \subsection{Proofs of Section~\ref{sec:asbiasis}\label{appx:proofs:asbiasis}}
  
  \subsubsection{Proof of Theorem~\ref{thm:asbias_is}\label{appx:proofs:thm:asbias_is}}
  
  We start with a technical result on the inverse moments of averages, which may be well-known.
  
  \begin{prop}\label{prop:inversemoments}
    Let $r\geq 1$ and let $\bfx=(\x_1,\ldots,\x_N) \sim \bar{q}$.
    Suppose $\omega:\mathbb{X}\to(0,\infty)$  satisfies $q(\omega^{-\eta})<\infty$ for some $\eta>0$.
    Then, for all integers $N\geq r/\eta$,
    \begin{equation}\label{eq:inversemoment_bound}
      \mathbb{E}_{\bar{q}}\!\left[\ZN(\bfx)^{-r}\right]
      \;\leq\;
      q(\omega^{-\eta})^{r/\eta}
      \;<\;\infty.
    \end{equation}
    The bound on the right-hand side is independent of $N$.
  \end{prop}
  
  \begin{proof}
    Write $\Omega_n=\omega(\x_n)$ and recall $\ZN(\bfx)=N^{-1}\sum_{n=1}^N\Omega_n$.  
    By the AM--GM inequality applied to $\Omega_1,\ldots,\Omega_N$,
    \[
    \ZN(\bfx) = \frac{\Omega_1+\ldots+\Omega_N}{N} \;\geq\; (\Omega_1\ldots\Omega_N)^{1/N},
    \]
    so, taking reciprocals and raising to the power $r$,
    \[
    \ZN(\bfx)^{-r} \;\leq\; (\Omega_1\ldots\Omega_N)^{-r/N}
    = \prod_{n=1}^N \Omega_n^{-r/N}.
    \]
    Taking expectations and using independence of $\x_1,\ldots,\x_N$ under $\bar{q}$,
    \[
    \mathbb{E}_{\bar{q}}\!\left[\ZN(\bfx)^{-r}\right]
    \;\leq\;
    \prod_{n=1}^N q(\omega^{-r/N})
    \;=\;
    q(\omega^{-r/N})^N.
    \]
    Since $N\geq r/\eta$, we have $r/N\leq\eta$, so the map $t\mapsto t^{r/(N\eta)}$ is concave on $(0,\infty)$.
    Applying Jensen's inequality to the random variable $\omega(x)^{-\eta}$ with $x\sim q$,
    \[
    q(\omega^{-r/N})
    \;=\;
    q\!\left[(\omega^{-\eta})^{r/(N\eta)}\right]
    \;\leq\;
    q(\omega^{-\eta})^{r/(N\eta)}.
    \]
    Substituting back,
    \[
    q(\omega^{-r/N})^N
    \;\leq\;
    q(\omega^{-\eta})^{r/\eta},
    \]
    which gives the claimed bound.
  \end{proof}
  
  \begin{proof}[Proof of Theorem~\ref{thm:asbias_is}]
    We first write the rescaled bias of the SNIS estimator as
    \begin{align*}
      N \times \mathbb{E}_{\bar{q}}\left[\FN(\bfx) - \pi(f)\right] &=
      \mathbb{E}\left[\frac{\sum_{n=1}^N \omega(\x_n)(f(\x_n) - \pi(f))}{\sum_{n=1}^N \omega(\x_n) / N}\right]\nonumber\\
      &= N \mathbb{E}\left[\frac{ \omega(\x_1)(f(\x_1) - \pi(f))}{\sum_{n=1}^N \omega(\x_n) / N}\right] \quad \text{by identity in distribution}\\
      &= N \mathbb{E}\left[\frac{ \omega(\x_1)(f(\x_1) - \pi(f))}{\sum_{n=2}^N \omega(\x_n) / N}\right]\\
      &\hspace*{-2cm}+ N \mathbb{E}\left[\omega(\x_1)(f(\x_1) - \pi(f)) \left\{\frac{1}{\sum_{n=1}^N \omega(\x_n) / N} - \frac{1}{\sum_{n=2}^N \omega(\x_n) / N}\right\}\right].
    \end{align*}
    Since $\x_1$ is independent of $(\x_2,\ldots,\x_N)$ and $q(\omega(f-\pi(f))) = q(\omega f) - \pi(f)q(\omega) = 0$, the first term vanishes. For the second term,
    \begin{equation}\label{eq:diffratioNandNm1}
      \frac{1}{\sum_{n=1}^N \omega(\x_n) / N} - \frac{1}{\sum_{n=2}^N \omega(\x_n) / N} = \frac{-\omega(\x_1)/N}{(\sum_{n=1}^N \omega(\x_n) / N)(\sum_{n=2}^N \omega(\x_n) / N)}.
    \end{equation}
    Substituting, we obtain
    \begin{align*}
      N \times \mathbb{E}_{\bar{q}}\left[\FN(\bfx) - \pi(f)\right] &=
      - \mathbb{E}\left[\frac{\omega(\x_1)^2(f(\x_1) - \pi(f))}{(\sum_{n=1}^N \omega(\x_n) / N)(\sum_{n=2}^N \omega(\x_n) / N)} \right].
    \end{align*}
    
    Define, for $N\geq 2$,
    \begin{align*}
      A_N := \frac{\omega(\x_1)^2(f(\x_1) - \pi(f))}{(\sum_{n=1}^N \omega(\x_n) / N)(\sum_{n=2}^N \omega(\x_n) / N)}.
    \end{align*}
    
    By the strong law of large numbers, $\sum_{n=1}^N \omega(\x_n) / N \xrightarrow{a.s.} 1$ and $\sum_{n=2}^N \omega(\x_n) / N \xrightarrow{a.s.} 1$ as $N \to \infty$, so
    \begin{align*}
      A_N \xrightarrow{a.s.} \omega(\x_1)^2(f(\x_1) - \pi(f)) \quad \text{as } N \to \infty.
    \end{align*}
    If $\{A_N\}_{N \geq 2}$ is uniformly integrable, we may exchange expectation and limit to get
    \begin{align*}
      \lim_{N \to \infty} N\times\mathbb{E}_{\bar{q}}\left[\FN(\bfx) - \pi(f)\right] &= -\mathbb{E}[\lim_{N \to \infty} A_N] = -q\left(\omega^2(f - \pi(f))\right),
    \end{align*}
    which is the claimed result.
    
    It remains to establish uniform integrability of $\{A_N\}_{N \geq 1}$. By \citet{billingsleyconvergence}, Equation (3.18), it suffices to find $\epsilon > 0$ such that
    \begin{align*}
      \sup_N \mathbb{E}[|A_N|^{1+\epsilon}] < \infty.
    \end{align*}
    Since $\omega > 0$ and $\sum_{n=1}^N \omega(\x_n)/N \geq \sum_{n=2}^N \omega(\x_n)/N$ a.s., we can a.s. bound
    \begin{align*}
      |A_N| \leq \frac{|\omega(\x_1)^2(f(\x_1) - \pi(f))|}{(\sum_{n=2}^N \omega(\x_n) / N)^2}.
    \end{align*}
    Using independence of $\x_1$ and $(\x_2, \ldots, \x_N)$,
    \begin{align*}
      \mathbb{E}[|A_N|^{1+\epsilon}] &\leq \mathbb{E}\left[\left|\omega(\x_1)^2(f(\x_1) - \pi(f))\right|^{1+\epsilon}\right] \mathbb{E}\left[\left(\frac{N}{\sum_{n=2}^N \omega(\x_n)}\right)^{2(1+\epsilon)}\right].
    \end{align*}
    
    Write $\hat{Z}_{N\setminus 1} = (N-1)^{-1}\sum_{n=2}^N\omega(\x_n)$ for the sample average of the last $N-1$ weights. Since $N\leq 2(N-1)$ for $N\geq 2$, we have $N^{-1}\sum_{n=2}^N\omega(\x_n) = \frac{N-1}{N}\hat{Z}_{N\setminus 1} \geq \frac{1}{2}\hat{Z}_{N\setminus 1}$, so
    \[\left(\frac{N}{\sum_{n=2}^N\omega(\x_n)}\right)^{2(1+\epsilon)} \leq 2^{2(1+\epsilon)}\,\hat{Z}_{N\setminus 1}^{-2(1+\epsilon)}.\]
    Applying Proposition~\ref{prop:inversemoments} with $r=2(1+\epsilon)$ to the $N-1$ i.i.d.\ variables $\x_2,\ldots,\x_N$, we get $\mathbb{E}[\hat{Z}_{N\setminus 1}^{-2(1+\epsilon)}]\leq q(\omega^{-\eta})^{2(1+\epsilon)/\eta}$ for $N-1\geq 2(1+\epsilon)/\eta$. The first factor in the bound on $\mathbb{E}[|A_N|^{1+\epsilon}]$ is finite by assumption. Hence $\sup_N \mathbb{E}[|A_N|^{1+\epsilon}] < \infty$, and uniform integrability follows by \citet[][(3.18)]{billingsleyconvergence}.
  \end{proof}

  \subsubsection{Proof of Theorem~\ref{thm: snis unbounded convergence}\label{appx:proofs:thm: snis unbounded convergence}}
  
  \begin{proof}[Proof of Theorem~\ref{thm: snis unbounded convergence}]
    
For a sample $\x_1, \ldots, \x_N$ i.i.d. from $q$,
we write their empirical distribution $q^N$ so that, for example, $q^N(\omega):= \sum_{n=1}^N \omega(\x_n)/N$.
    The IS estimator is thus $q^N(f\omega)/q^N(\omega)$.
    First, it is enough to consider the case where the test function
    $f$ is non-negative. Indeed, for a general function $f$ we write $f = f_{+} - f_{-}$ where $f_{+}(x) := \max \{f(x), 0\}$ and $f_{-}(x) := -\min\{f(x), 0\}$. Then 
    \begin{align*}
      \left\lvert \frac{q^N(f\omega)}{q^N(\omega)}  - \pi(f)\right\rvert 
      &= \left\lvert \frac{q^N(f_{+}\omega) - q^N(f_{-}\omega) }{q^N(\omega)}  - (\pi(f_{+}) - \pi(f_{-}))\right\rvert \\
      &\leq \left\lvert \frac{q^N(f_{+}\omega)}{q^N(\omega)}  - \pi(f_+)\right\rvert + \left\lvert \frac{q^N(f_-\omega)}{q^N(\omega)}  - \pi(f_{-})\right\rvert.
    \end{align*}
    Using $(a+b)^s \leq 2^{s-1} (a^s + b^s)$, and applying the result for non-negative functions $f_+$ and $f_-$ separately, we obtain the result for general $f$. We thus assume that $f$ is non-negative.
    
    We write the absolute error between the IS estimator with the target $\pi(f) = q(f\omega)$ in two different ways. The first is: 
    \begin{align}\label{eq:upb_max_plus_qwf}
      \left\lvert \frac{q^N(f\omega)}{q^N(\omega)}  - q(f\omega)\right\rvert \leq \max_{1\leq i \leq N} f(x_i)+q(f\omega).
    \end{align}
    The second is: 
    \begin{align*}
      \left\lvert \frac{q^N(f\omega)}{q^N(\omega)}  - q(f\omega)\right\rvert \leq \left\lvert \frac{q^N(f\omega)}{q^N(\omega)} - \frac{q(f\omega)}{q^N(\omega)}\right\rvert + 
      q(f\omega)\left\lvert \frac{1}{q^N(\omega)} - 1 \right\rvert.
    \end{align*}
    Now we consider two cases: 1) $\lvert q^N(\omega) - 1\rvert > \tfrac{1}{2}$, 2) $\lvert q^N(\omega) - 1\rvert \leq \tfrac{1}{2}$. We will separately bound the expected error under the two cases using the two inequalities above.
    
    We start with the first case, and we assume $r < \infty$. First, we use $(a+b)^s \leq 2^{s-1}(a^s + b^s)$ to write
    \begin{align*}
      \bE\left[ 	\left\lvert \frac{q^N(f\omega)}{q^N(\omega)}  - q(f\omega)\right\rvert^s \mathds{1}(\lvert q^N(\omega) - 1\rvert > \tfrac{1}{2})\right] &\leq 
      \bE\left[ \left(\max_{1\leq i \leq N} f(\x_i)+q(f\omega)\right)^s	\mathds{1}(\lvert q^N(\omega) - 1\rvert > \tfrac{1}{2})\right] \\
      &\leq 2^{s-1}\bE\left[ \left(\max_{1\leq i \leq N} f(\x_i)\right)^s	\mathds{1}(\lvert q^N(\omega) - 1\rvert > \tfrac{1}{2})\right] \break\\
      &+2^{s-1}\left(q(f\omega)\right)^s	\bP(\lvert q^N(\omega) - 1\rvert > \tfrac{1}{2}).
    \end{align*}
    The second term leads to a bound in $N^{-s/2}$ using Markov's inequality as in Proposition~\ref{prop:zhatconcentration}, since $q(\omega^s)<\infty$ under the assumptions.
    The first term is dealt with first using H\"older's inequality with exponents $r/s$ and $(1-s/r)^{-1}$,
    \begin{align*}
      &	\bE\left[ \left(\max_{1\leq i \leq N} f(\x_i)\right)^s	\mathds{1}(\lvert q^N(\omega) - 1\rvert > \tfrac{1}{2})\right]\\
      &\leq \bE\left[ \left(\max_{1\leq i \leq N} f(\x_i)\right)^r \right]^{s/r} \times \bP(\lvert q^N(\omega) - 1\rvert > \tfrac{1}{2})^{1- s/r}\\
      & \leq q(f^r)^{s/r} N^{s/r} \cdot C \cdot N^{-\tfrac{1}{2}p(1- s/r)},
    \end{align*}
    for a constant $C$.
    The last inequality uses the fact that  $\bE[(\max_{1\leq i \leq N} f(\x_i))^r] \leq N \bE[f(\x_1)^r]$, and Markov's inequality using $q(\omega^p)<\infty$.
    Given $s \leq pr/(p+r+2)$, the exponent of $N$ satisfies
    \begin{align*}
      \frac{s}{r} - \frac{p(r-s)}{2r} = \frac{2s + ps - pr}{2r} \leq \frac{-s}{2},
    \end{align*}
    using $s \leq pr/(p+r+2) \Leftrightarrow -pr \leq -s(p+r+2)$. 
    Altogether we arrive at 
    \begin{align*}
      \bE\left[ 	\left\lvert \frac{q^N(f\omega)}{q^N(\omega)}  - q(f\omega)\right\rvert^s \mathds{1}(\lvert q^N(\omega) - 1\rvert > \tfrac{1}{2})\right] \leq C N^{-s/2},
    \end{align*}
    for another constant $C$.
    
    In the case $r = \infty$, we can directly write
    \begin{align*}
      \bE\left[ 	\left\lvert \frac{q^N(f\omega)}{q^N(\omega)}  - q(f\omega)\right\rvert^s \mathds{1}(\lvert q^N(\omega) - 1\rvert > \tfrac{1}{2})\right] &\leq 
      \bE\left[ \left(\max_{1\leq i \leq N} f(\x_i)+q(f\omega)\right)^s	\mathds{1}(\lvert q^N(\omega) - 1\rvert > \tfrac{1}{2})\right] \\
      &\leq 2^s \lvert f \rvert_\infty^s \bP(\lvert q^N(\omega) - 1\rvert > \tfrac{1}{2})\\
      & \leq  2^s C \lvert f \rvert_\infty^s N^{-\tfrac{1}{2}p} \leq C N^{-\tfrac{1}{2}s},
    \end{align*}
    using $s \leq \min\{p, r\} \leq p$ in the last line, and changing the value of $C$ between inequalities.
    
    For the case $|q^N(\omega)  - 1| \leq \tfrac{1}{2}$,
    \begin{align*}
      \left\lvert \frac{q^N(f\omega)}{q^N(\omega)}  - \pi(f)\right\rvert \mathds{1}(\lvert q^N(\omega) - 1\rvert \leq  0.5)
      &\leq 2 \lvert q^N(f\omega) - \pi(f)\rvert + 2\pi(f) \left \lvert q^N(\omega) - 1 \right \rvert.
    \end{align*}
    Therefore, 
    \begin{align*}
      \bE\left[ 	\left\lvert \frac{q^N(f\omega)}{q^N(\omega)}  - \pi(f)\right\rvert^s \mathds{1}(\lvert q^N(\omega) - 1\rvert \leq 0.5)\right] &\leq C\left(\bE[\lvert q^N(f\omega) - \pi(f)\rvert^s] + \bE[\left \lvert q^N(\omega) - 1 \right \rvert^s]\right)\\
      &\leq C N^{-s/2},
    \end{align*}
    for some constant $C$ that changes at each line.
    The first term is $O(N^{-s/2})$ with a reasoning similar to
    that in the proof of Proposition~\ref{prop:zhatconcentration},
    since $q^N(f\omega)$ is the average of  $N$ i.i.d. random variables with mean $q(f\omega)$ and $s$ finite moments, since $s\leq pr/(p+r+2) \leq pr/p+r$.
    Putting everything together gives
    \begin{align*}
      \bE\left[ 	\left\lvert \frac{q^N(f\omega)}{q^N(\omega)}  - \pi(f)\right\rvert^s\right] \leq C N^{-s/2}.
    \end{align*}
    
  \end{proof}

  \subsubsection{\label{appx:proof_sketch:snis_moments_p_less_than_2}Proof of Proposition~\ref{prop:snis_moments_p_less_than_2}}

\begin{proof}[Proof of Proposition~\ref{prop:snis_moments_p_less_than_2}]
  Set $\alpha := pr/(p+r)$, so that $1/\alpha = 1/p+1/r$ and
  $1<\alpha\leq 2$ under the assumptions of
  Proposition~\ref{prop:snis_moments_p_less_than_2}. By H\"older's
  inequality with conjugate exponents $p/\alpha$ and $r/\alpha$,
  \begin{align}\label{eq:hv:fomega-bound}
    q(|f\omega|^{\alpha}) \leq \left[q(\omega^{p})\right]^{\alpha/p}\,\left[q(|f|^{r})\right]^{\alpha/r} <\infty,
  \end{align}
  so the summands $\omega(\x_n)f(\x_n)$ have a finite absolute moment of
  order $\alpha$. The assumption $\alpha>1$, equivalent to $1/p+1/r<1$,
  is exactly the integrability threshold beyond which the von
  Bahr--Esseen inequality of Lemma~\ref{lem:vbe_less2} sharpens the
  trivial $\mathcal{O}(1)$ bound on centred sums into the decaying rate
  $N^{-(\alpha-1)/\alpha}$.

  Proposition~\ref{prop:zhatconcentration_less2} controls the centred
  sum of weights:
  \begin{align}\label{eq:hv:Z-bound}
    \bE\!\left[\,\left|q^N(\omega) - 1\right|^{p}\,\right]
    \leq \frac{C}{N^{p-1}},
    \qquad
    \bP\!\left(\,\left|q^N(\omega) - 1\right| > \tfrac{1}{2}\,\right)
    \leq \frac{C}{N^{p-1}}.
  \end{align}
  For the centred numerator, set $Y_n := \omega(\x_n)f(\x_n) -
  q(f\omega)$ and apply Lemma~\ref{lem:vbe_less2} at order $\alpha$ to
  $\sum_n Y_n$, then Lyapunov's inequality at $1\leq s\leq \alpha$:
  \begin{align}\label{eq:hv:num-bound-s}
    \bE\!\left[\,\left|q^N(f\omega) - q(f\omega)\right|^{s}\,\right]
    \leq C\, N^{-s(\alpha-1)/\alpha}.
  \end{align}
  Specialising to $f\equiv 1$ gives the analogous bound
  \begin{align}\label{eq:hv:Z-bound-s}
    \bE\!\left[\,\left|q^N(\omega) - 1\right|^{s}\,\right]
    \leq C\,N^{-s(p-1)/p},
    \qquad 1\leq s\leq p.
  \end{align}

  From here the argument follows that of
  the proof of Theorem~\ref{thm: snis unbounded convergence}, the only change being the $L^s$-bounds in $N^{-1/2}$ being replaced by the above, slower-in-$N$ rates.

  \end{proof}

  \subsection{Proofs of Section~\ref{sec:crncoupling}\label{appx:proofs:crncoupling}}
  
  We assume that both target and proposal distributions admit densities with respect to a measure $\lambda$. Although we will express all subsequent notations using integration, this should be interpreted as summation when the space is discrete and $\lambda$ represents the counting measure.
  The rejection probability at $x$ is defined as
\[r(x) =  \int_{x^\star \neq x} \left(1- \min\left(1,\frac{\hat{Z}(x^\star)}{\hat{Z}(x)}\right)\right)\bar{q}(x^\star)\lambda(\diff x^\star).\]
  That definition only considers the probability of moves to states different than $x$ that are rejected. We will use the following fact: at every iteration, for each chain one of the following three events occurs: 1) a proposal to a different state is accepted, 2) a proposal to a different state is rejected, 3) a proposal is made to the current state (and systematically accepted). In a continuous state space with an atomless measure $\lambda$, the last event occurs with probability zero.
  We assume Assumption~\ref{asm:abscontinuitypositivity} throughout so that $\bP(\hat{Z}(x) = 0) = 0$
  under $q$.
  
  We first prove a lemma that describes the coupling time $\tau$.
  
  \begin{lem}\label{lem:couplingtime-facts}
    Assuming $\omega(x) \geq \omega(y)$, we have the following facts: 
    \begin{itemize}
      \item Let $\tau_0$ be the first time when the $X$-chain moves to a different state. Then  $\tau \leq \tau_0$, i.e. the chains meet at $\tau_0$ or earlier.
      \item	Let $\tau_1$ be the first time when a common proposal is $x$. Then  $\tau \leq \tau_1$, i.e. the chains meet at $\tau_1$ or earlier.
      \item The meeting time satisfies $\tau = \min\{\tau_0,\tau_1\}$.
    \end{itemize}
  \end{lem}
  \begin{proof}[Proof of Lemma \ref{lem:couplingtime-facts}]
    The first two observations can be proven by induction, once we recognize that the common draws coupling of Algorithm~\ref{alg:coupledIMH} implies $\omega(X_t) \geq \omega(Y_t)$ for all  $t\geq 0$. Regarding the last observation, for every $t < \min\{\tau_0,\tau_1\}$, the $X$-chain must have rejected moves to a different state than $x$ at each iteration up to $t$. In that situation, the $X$-chain is still at $x$, while the $Y$-chain never proposed a move to $x$ and thus $X_t = x \neq Y_t$, as claimed.
  \end{proof}
  
  Now we calculate the tail probability of $\tau$.
  \begin{lem}\label{lem:coupling-upperbound}
    For all $t\geq 1$, $\displaystyle |P^t(x,\cdot) - P^t(y,\cdot)|_{\mathrm{TV}}  \leq \mathbb{P}_{x,y}(\tau > t) = \max(r(x),r(y))^t.$
  \end{lem}
  \begin{proof}[Proof of Lemma \ref{lem:coupling-upperbound}]
    The inequality in the statement is the celebrated coupling inequality. For the equality,
    we assume $\omega(x) \geq \omega(y)$ without loss of generality, which implies $r(x) \geq r(y)$. By Lemma \ref{lem:couplingtime-facts}, the event $\{\tau > t\}$ is equivalent to $\{\min\{\tau_0,\tau_1\} > t\}$. The latter event corresponds to the event: ``the $X$-chain proposes to move to a different state but gets rejected at each of the first $t$ iterations''.
    Then its probability is $r(x)^t$, since $r(x)$ is the probability of a failed attempt to move to a different state.
  \end{proof}
  
  It remains to show the following lower bound. 
  \begin{lem}\label{lem:coupling-lowerbound}
    For all $t\geq 1$, $\displaystyle |P^t(x,\cdot) - P^t(y,\cdot)|_{\mathrm{TV}} \geq \mathbb{P}_{x,y}(\tau > t)$.
  \end{lem}
  \begin{proof}[Proof of Lemma \ref{lem:coupling-lowerbound}]
    Again, we assume $\omega(x) \geq \omega(y)$ without loss of generality. The definition of total variation distance as a supremum over measurable sets implies $\displaystyle |P^t(x,\cdot) - P^t(y,\cdot)|_{\mathrm{TV}} \geq P^t(x,\{x\}) - P^t(y,\{x\})$, considering the set $\{x\}$.
    
    Under the distribution of the coupled chains, we can write $P^t(x,\{x\}) - P^t(y,\{x\})$ as $\bP(X_t = x) - \bP(Y_t = x)$. Now we decompose each probability according to $\tau$ being greater or less than $t$, for any $t\geq 1$:
    \begin{align*}
      \bP\left(X_t = x\right) - \bP\left(Y_t =  x\right) = &\; \bP\left(X_t = x; \tau >  t\right) +   \bP\left(X_t = x; \tau \leq t\right) \\
      &- \bP\left(Y_t = x; \tau > t\right) -  \bP\left(Y_t = x; \tau \leq t\right).
    \end{align*}
    We simplify with the following observations.
    \begin{itemize}
      \item Under the event $\tau > t$: we have $X_t = x$; otherwise, the $X$-chain would have successfully moved to a new state jointly with the $Y$-chain implying $\tau \leq t$ by Lemma \ref{lem:couplingtime-facts}. Therefore, 
      \begin{align*}
        \bP\left(X_t = x; \tau >  t\right) = \bP\left(\tau > t\right) \bP\left(X_t = x \mid \tau>t\right) =  \bP\left(\tau > t\right).
      \end{align*} 
      Meanwhile, under that event we have $Y_t \neq x$; otherwise, the $Y$-chain must have proposed a move to $x$ at or before time $t$, and that would have resulted in a meeting by Lemma \ref{lem:couplingtime-facts}. Therefore, 
      \begin{align*}
        \bP\left(Y_t = x; \tau >  t\right) = 0.
      \end{align*}
      \item 	Under the event  $\tau \leq t$: we have $X_t = Y_t$, therefore  $ \bP(X_t = x; \tau \leq t) =  \bP(Y_t = x; \tau \leq t)$. 
    \end{itemize}
    Putting these together, we conclude that $\bP\left(X_t = x\right) - \bP\left(Y_t = x\right) = \bP\left(\tau > t\right)$. 
  \end{proof}
  Theorem \ref{thm:couplingIMH} is obtained by combining Lemmas \ref{lem:coupling-upperbound} and \ref{lem:coupling-lowerbound}.

  \subsection{Proofs of Section~\ref{sec:imh}\label{appx:proofs:imh}}
  
  \subsubsection{Proof of Proposition~\ref{prop:meetingtimes_rejproba}\label{appx:proofs:prop:meetingtimes_rejproba}}
  
  Let $t\geq 1$. The event $\{\tau > t\}$ only occurs when Algorithm~\ref{alg:coupledPIMHlag} enters its while loop, in which case we must have that 1) $\bfx_1=\bfx_0$, 2) $\hat{Z}(\bfx_0)>\hat{Z}(\bfy_0)$, and 3) the first generated Uniform variable was greater than $\hat{Z}(\bfy_0)/\hat{Z}(\bfx_0)$. Thus,
  \begin{align*}
    \mathbb{P}(\tau > t) = \iint &\mathbb{P}_{\bfxsmall_0,\bfysmall_0}\left(\tau > t\right) \left(1-\min\left\{1,\frac{\hat{Z}(\bfysmall_0)}{\hat{Z}(\bfxsmall_0)}\right\}\right) \\
    & \times \mathds{1}\left(\hat{Z}(\bfysmall_0)<\hat{Z}(\bfxsmall_0)\right) \mathds{1}(\bfysmall_0 \neq \bfxsmall_0)\bar{q}(\diff \bfxsmall_0)\bar{q}(\diff\bfysmall_0).
  \end{align*}
  The quantity $\mathbb{P}_{\bfxsmall_0,\bfysmall_0}(\tau > t)$ in the event $\hat{Z}(\bfysmall_0)<\hat{Z}(\bfxsmall_0)$ is equal to $r(\bfxsmall_0)^t$, as in Theorem~\ref{thm:couplingIMH}.
  By upper-bounding the other terms by one and integrating with respect to $\bar{q}(\diff\bfysmall_0)$, we obtain the upper bound
  \begin{equation}\label{eq:taugreaterthant:step2}
    \mathbb{P}(\tau > t) \leq \int (r(\bfxsmall_0))^t \bar{q}(\diff\bfxsmall_0) = \mathbb{E}_{\bar{q}}\left[(r(\bfx_0))^t\right].
  \end{equation}

  \subsubsection{Proof of Proposition~\ref{prop:expected_rejection_probability_bound}\label{appx:proofs:prop:expected_rejection_probability_bound}}
  We prove Proposition~\ref{prop:expected_rejection_probability_bound} by first splitting the expectation according to whether $\ZN(\bfx)$ is less than or greater than 2:
  \[
    \mathbb{E}_{\bar{q}}\left[r(\bfx)^t\right] = \mathbb{E}_{\bar{q}}\left[r(\bfx)^t \mathds{1}(\ZN(\bfx) \leq 2) \right] + 
    \mathbb{E}_{\bar{q}}\left[r(\bfx)^t\mathds{1}(\ZN(\bfx) > 2)\right].\]
  
  We then proceed through a series of lemmas to bound each term. The following lemmas are used to handle the case where $\ZN(\bfx) > 2$. 
 
  \begin{lem}\label{lem:rupb}
    Under Assumption~\ref{asm:abscontinuitypositivity} and $q(\omega^p)<\infty$ for any $p>1$, the rejection probability~\eqref{eq:pimh:reject} is upper bounded as follows, for any $\theta\in[0,1]$ and any $\bfxsmall\in\mathbb{X}^N$:
    \begin{align}\label{eqn:rejection_upper_bound}
      r(\bfxsmall) \leq 1 - \min\left\{1,\frac{\theta}{\ZN(\bfxsmall)}\right\} c_p(\theta),\quad \text{with}\quad c_p(\theta) = \frac{(1-\theta)^{p/(p-1)}}{q(\omega^p)^{1/(p-1)}}\in[0,1].
    \end{align}
  \end{lem}
  
  \begin{proof}
    Let $\theta\in[0,1]$. We start with an $L^p$-version of the Paley--Zygmund inequality, as on page 2705, equation (12) of \cite{petrov2007lower} with $r = 1$.
    If $W$ is a non-negative random variable and $p > 1$, then 
    \begin{align}\label{eqn:P-Z inequality}
      \bP\left(W > \theta \bE[W]\right) \geq \frac{(1-\theta)^{p/(p-1)}\left(\bE[W]\right)^{p/(p-1)}}{\left(\bE[W^p]\right)^{1/(p-1)}}.
    \end{align}
    Indeed, for any $b > 0$,  H{\"o}lder's inequality implies
    \begin{align*}
      \bE[W] &= \bE[W \mathds{1}(W> b)] + \bE[W \mathds{1}(W\leq b)]\\
      &\leq \bP(W> b)^{(1 - 1/p)} \bE[W^p]^{1/p} + b.
    \end{align*}
    Rearranging with $b = \theta \bE[W]$ implies \eqref{eqn:P-Z inequality}. We apply this to $\ZN(\bfx)$, under Assumption~\ref{asm:abscontinuitypositivity}:
    \begin{align}\label{eqn:P-Z inequality for weight}
      \bP_{\bar{q}}\left(\ZN(\bfx) > \theta\right) \geq \frac{(1-\theta)^{p/(p-1)}}{\left(\bE_{\bar{q}}\left[\left(\ZN(\bfx)\right)^p\right]\right)^{1/(p-1)}} \geq \frac{(1-\theta)^{p/(p-1)}}{q(\omega^p)^{1/(p-1)}}.
    \end{align}
    The latter inequality comes from Jensen's inequality, since $z\mapsto z^p$ with $p>1$ is convex:
      \[
      \mathbb{E}_{\bar{q}}\left[\left(\ZN(\bfx)\right)^p\right]\leq \mathbb{E}_{\bar{q}}\left[\frac{1}{N}\sum_{n=1}^N \omega(\x_n)^p\right] = q(\omega^p).
      \]
    Inequality \eqref{eqn:P-Z inequality for weight} implies that
    \begin{align*}
      \displaystyle
      \int \min\left\{1,\frac{\ZN(\bfxsmallstar)}{\ZN(\bfxsmall)}\right\} \bar{q}(\diff\bfxsmallstar) 
      &= \int_{\{\bfxsmallstar:\ZN(\bfxsmallstar)\leq\theta\}} \min\left\{1,\frac{\ZN(\bfxsmallstar)}{\ZN(\bfxsmall)}\right\} \bar{q}(\diff\bfxsmallstar) \\
      &\quad+ \int_{\{\bfxsmallstar:\ZN(\bfxsmallstar)>\theta\}} \min\left\{1,\frac{\ZN(\bfxsmallstar)}{\ZN(\bfxsmall)}\right\} \bar{q}(\diff\bfxsmallstar)\\
      &\geq 0 + \min\left\{1,\frac{\theta}{\ZN(\bfxsmall)}\right\}\mathbb{P}_{\bar{q}}\left(\ZN(\bfxsmallstar)>\theta\right)\\
      &\geq \min\left\{1,\frac{\theta}{\ZN(\bfxsmall)}\right\} \frac{(1-\theta)^{p/(p-1)}}{q(\omega^p)^{1/(p-1)}}.
    \end{align*}
    This yields the desired result.
  \end{proof}
  
  \begin{lem}\label{lem:large-Z-bound}
    Under Assumptions~\ref{asm:abscontinuitypositivity}-\ref{asm:pfinitemoments} with $p\geq 2$, there exists a constant $C > 0$ such that for all $t\geq 1$, $N\geq 1$,
    \begin{equation*}
      \mathbb{E}_{\bar{q}}\left[r(\bfx)^t\mathds{1}(\ZN(\bfx)> 2)\right] \leq \frac{C}{N^{p/2}t^{p}}.
    \end{equation*}
  \end{lem}
  
  \begin{proof}
    We split the expectation into two parts:
    \begin{align*}
      \mathbb{E}_{\bar{q}}\left[r(\bfx)^t\mathds{1}\left(\ZN(\bfx)> 2\right)\right] 
      &= \mathbb{E}_{\bar{q}}\left[r(\bfx)^t\mathds{1}\left(\ZN(\bfx)\in (2,1+t)\right)\right] \\
      &+ \mathbb{E}_{\bar{q}}\left[r(\bfx)^t\mathds{1}\left(\ZN(\bfx)\geq 1+t\right)\right].
    \end{align*}
    
    For $\{\ZN(\bfx)\geq 1+t\}$, we directly apply Proposition~\ref{prop:zhatconcentration}:
    \begin{align*}
      \mathbb{E}_{\bar{q}}\left[r(\bfx)^t\mathds{1}(\ZN(\bfx)\geq 1+t)\right] 
      &\leq \mathbb{P}_{\bar{q}}\left(\ZN(\bfx)\geq 1+t\right) \\
      &\leq \frac{M(p)^p}{N^{p/2}t^{p}}.
    \end{align*}
    
    For $\{\ZN(\bfx)\in (2,1+t)\}$, we use Lemma~\ref{lem:rupb} with $\theta = 1/2$:
    \begin{align*}
      r(\bfxsmall) &\leq 1 - \frac{c_p(1/2)}{2\ZN(\bfxsmall)} \leq \exp\left(-\frac{c_p(1/2)}{2\ZN(\bfxsmall)}\right).
    \end{align*}
    
    Let $c = c_p(1/2)/4$. Then using the fact that $\ZN(\bfxsmall)>2$ implies that $\ZN(\bfxsmall)\leq 2(\ZN(\bfxsmall)-1)$, we have:
    \begin{align*}
      r(\bfxsmall)^t &\leq \exp\left(-\frac{2ct}{\ZN(\bfxsmall)}\right) \leq \exp\left(-\frac{ct}{\ZN(\bfxsmall)-1}\right).
    \end{align*}
    
    We introduce the sets $A_k = \left[t/(k+1), t/k\right]$ for $k \geq 1$, so that $\cup_{k=1}^{\infty} A_k = (0,t]$ which contains $(1,t)$. Using the result of Proposition~\ref{prop:zhatconcentration}, we obtain the bound:
    \begin{align*}
      &\mathbb{E}_{\bar{q}}\left[r(\bfx)^t\mathds{1}(\ZN(\bfx)\in (2,1+t))\right] \\
      &\leq \mathbb{E}_{\bar{q}}\left[\exp\left(-\frac{ct}{\ZN(\bfx)-1}\right)\mathds{1}\left(\ZN(\bfx)-1\in (1,t)\right)\right] \\
      &\leq \sum_{k=1}^{\infty} \mathbb{E}_{\bar{q}}\left[\exp\left(-\frac{ct}{\ZN(\bfx)-1}\right)\mathds{1}\left(\ZN(\bfx)-1\in A_k\right)\right] \\
      &\leq \sum_{k=1}^{\infty} \exp(-ck) \mathbb{P}_{\bar{q}}\left(\ZN(\bfx)\geq 1+t/(k+1)\right) \\
      &\leq \sum_{k=1}^{\infty} \exp(-ck) \frac{M(p)^p}{N^{p/2}}\left(\frac{k+1}{t}\right)^{p}.
    \end{align*}
    
    Let $S_p = \sum_{k=1}^{\infty} \exp(-ck) (k+1)^{p}$, which is finite. Then:
    \begin{align*}
      \mathbb{E}_{\bar{q}}\left[r(\bfx)^t\mathds{1}\left(\ZN(\bfx)\in (2,1+t)\right)\right] 
      &\leq \frac{M(p)^p S_p}{N^{p/2}t^{p}}.
    \end{align*}    
    Combining the bounds for both parts, we get:
    \begin{align*}
      \mathbb{E}_{\bar{q}}\left[r(\bfx)^t\mathds{1}\left(\ZN(\bfx)> 2\right)\right] 
      &\leq \frac{M(p)^p}{N^{p/2}t^{p}} + \frac{M(p)^p S_p}{N^{p/2}t^{p}} \\
      &\leq \frac{M(p)^p(1 + S_p)}{N^{p/2}t^{p}}.
    \end{align*}
    Setting the new constant $C := M(p)^p\left(1 + S_p\right)$ completes the proof.
  \end{proof}
  Now, we turn our attention to controlling the expectation when $\ZN(\bfx) \leq 2$. 
  
  \begin{lem}\label{lem:rejection_expectation_bound}
    Fix $p \geq 2$ and let $\beta_p$ be defined as in \eqref{eq:betap}.
    There exist constants $A_p, B_p>0$, depending only on $p$ and $q(\omega^p)$, such that for all $N\geq 1$, for all $t\geq 1$, the following holds:
    \begin{equation*}
      \mathbb{E}_{\bar{q}}\left[r(\bfx)^t\mathds{1}\left(\ZN(\bfx)\leq 2\right)\right] \leq 
      \left[\frac{A_p}{N^\frac{t\wedge p}{2}}+ \frac{B_p }{N^{p/2}}\right]\beta_p^t.        
    \end{equation*}
    
  \end{lem}
  
  \begin{proof}
    We abuse notation to write $r$ as a function of the value $z$ taken by $\ZN(\bfx)$, 
    instead of a function of $\bfx$, in various places in this proof. First notice that $r(z)$ is increasing in $z$. 
    We thus have that for $t\geq 1$ that
    \begin{align*}
      &\mathbb{E}_{\bar{q}}\left[r(\ZN(\bfx))^t\mathds{1}\left(\ZN(\bfx)\leq 2\right)\right]\\
      &\leq r(2)^{t-t\wedge p}\mathbb{E}_{\bar{q}}\left[r(\ZN(\bfx))^{t\wedge p} \mathds{1}\left(\ZN(\bfx)\leq 2\right)\right].
    \end{align*}
    We first consider the second factor. We have for any $\alpha \in (0,1)$,
    \begin{align}
      &\mathbb{E}_{\bar{q}}\left[r(\ZN(\bfx))^{t\wedge p} \mathds{1}\left(\ZN(\bfx)\leq 2\right)\right]\nonumber \\
      &\leq \mathbb{E}_{\bar{q}}\left[r(\ZN(\bfx))^{t\wedge p} \mathds{1}\left(1-\alpha \leq \ZN(\bfx)\leq 2\right)\right] + r(2)^{t\wedge p}\bar{q}\left\{|\ZN(\bfx)-1|\geq \alpha \right\}.\label{eq:lemmaA.7:intermed}
    \end{align}
    That is because $\{\ZN(\bfx)\leq 1-\alpha\}\subset \{|\ZN(\bfx)-1|\geq \alpha\}$, and $r(z)\leq r(2)$ for $z\leq 2$.
    
    At this point, notice that by Lemma~\ref{lem:rupb} with $\theta =1/2$ we have that 
$r(2)\leq 1- c_p(1/2)/4 = \beta_p$,
    where $\beta_p$ is defined in \eqref{eq:betap}. Writing $\nu$ for the law of $\ZN$ under $\bar{q}$, we have that
    \begin{align*}
      r(z) &= 1 - \int \min\left\{ 1, \frac{z^\ast}{z}\right\} \nu(\diff z^\ast)= \int_{z^\ast =0}^\infty \nu(\diff z^\ast) - \int_{z^\ast=0}^z \frac{z^\ast}{z} \nu(\diff z^\ast) - \int_{z^\ast=z}^\infty \nu(\diff z^\ast)\\
      &= \int_{z^\ast =0}^{z} \nu(\diff z^\ast) - \int_{z^\ast = 0}^{z} \frac{z^\ast}{z} \nu(\diff z^\ast) = \int_{z^\ast =0}^{z} \left(\frac{z - z^\ast}{z}\right) \nu(\diff z^\ast).
    \end{align*}
    Returning to our calculation regarding the first term in \eqref{eq:lemmaA.7:intermed},
    \begin{align*}
      &\mathbb{E}_{\bar{q}}\left[r(\ZN(\bfx))^{t\wedge p} \mathds{1}\left(1-\alpha \leq \ZN(\bfx)\leq 2\right)\right]\\
      &=\mathbb{E}_{\bar{q}}\left[ \left(\frac{1}{\ZN(\bfx)} \int_{z^\ast=0}^{\ZN(\bfx)} (\ZN(\bfx) - z^\ast)\nu(\diff z^\ast) \right)^{t\wedge p} \mathds{1}\left(1-\alpha \leq \ZN(\bfx)\leq 2\right)\right]\\
      &\leq 
      \frac{ \nu\{ [0, \ZN(\bfx)]\}^{t\wedge p} }{(1-\alpha)^{t\wedge p}} \\
      & \times 
      \mathbb{E}_{\bar{q}}\left[ \left( \int_{z^\ast=0}^{\ZN(\bfx)} (\ZN(\bfx) - z^\ast)\frac{\nu(\diff z^\ast)}{\nu\{ [0, \ZN(\bfx)]\}} \right)^{t\wedge p} \mathds{1}\left(1-\alpha \leq \ZN(\bfx)\leq 2\right)\right]\\
      &\leq 
      \frac{ \nu\{ [0, \ZN(\bfx)]\}^{t\wedge p-1} }{(1-\alpha)^{t\wedge p}}\\
      &\times
      \mathbb{E}_{\bar{q}}\left[ \int_{z^\ast=0}^{\ZN(\bfx)} (\ZN(\bfx) - z^\ast)^{t\wedge p} \nu(\diff z^\ast)\;\cdot \;\mathds{1}\left(1-\alpha \leq \ZN(\bfx)\leq 2\right)\right]\\
      &\leq 
      \frac{ 1 }{(1-\alpha)^{t\wedge p}}
      \mathbb{E}_{\bar{q}}\left[ \int_{z^\ast=0}^{\ZN(\bfx)} |\ZN(\bfx) - z^\ast|^{t\wedge p} \nu(\diff z^\ast)\right]\\
      &\leq 
      \frac{ 1 }{(1-\alpha)^{t\wedge p}}
      \mathbb{E}_{(\bfx,\bfxprime)\sim \bar{q}\otimes \bar{q}}\left[ |\ZN(\bfx) - \ZN(\bfxprime)|^{t\wedge p} \right] 
      \leq \frac{ 1 }{(1-\alpha)^{t\wedge p}} \frac{A_p q(\omega^p)}{N^\frac{t\wedge p}{2}},
    \end{align*}
    for a constant $A_p$ depending only on $p$. The first inequality comes
    from $\ZN(\bfx)^{-1}\leq (1-\alpha)^{-1}$ on the event of interest, 
    the second inequality is from Jensen's since the function $u\mapsto u^{t\wedge p}$ is convex,
    the third inequality is from $\bar{q}(A)\leq 1$ and the indicator being smaller than one,
    the fourth is obtained by completing the integral over all $z^\ast \in (0,\infty)$,
    and the last is from a reasoning similar to the proof of Proposition~\ref{prop:zhatconcentration},
    or by direct application of Minkowski's inequality and Proposition~\ref{prop:zhatconcentration}.
    
    Overall, choosing $\alpha=1/2$ 
    we have that 
    \begin{align*}
      \mathbb{E}_{\bar{q}}\left[r(\bfx)^{t\wedge p} \mathds{1}\left(\ZN(\bfx)\leq 2\right)\right]&\leq 2^{t\wedge p} A_p q(\omega^p) N^{-t\wedge p / 2} + r(2)^{t\wedge p} \nu\left\{|\ZN(\bfx) - 1| \geq \alpha \right\}\\
      &\leq 2^{t\wedge p} A_p q(\omega^p) N^{-t\wedge p / 2} + \beta_p^{t\wedge p} C_p N^{-p/2} 2^{p},
    \end{align*}
    using Markov's inequality as in Proposition~\ref{prop:zhatconcentration}.
    Finally, multiplying by $r(2)^{t-t\wedge p}$ we obtain
    \begin{align*}
      &\mathbb{E}_{\bar{q}}\left[r(\bfx)^{t} \mathds{1}\left(\ZN(\bfx)\leq 2\right)\right]\leq \frac{\beta_p^{t} \beta_p^{-t\wedge p} 2^{t \wedge p} A_p q(\omega^p)}{N^\frac{t\wedge p}{2}} + \frac{2^p C_p\beta_p^{t}}{N^{p/2}},
    \end{align*}
    and we note that, since $\beta_p\leq 1$, we have $\beta_p^{-t\wedge p} 2^{t\wedge p} \leq \beta_p^{-p} 2^p$,
    and thus we can define $A_p$ and $B_p$ to obtain Lemma~\ref{lem:rejection_expectation_bound}.
  \end{proof}

  \begin{proof}[Proof of Proposition~\ref{prop:expected_rejection_probability_bound}]
    We combine the bounds from Lemmas \ref{lem:large-Z-bound} and \ref{lem:rejection_expectation_bound},
    and note that the two terms in the bound of Lemma~\ref{lem:rejection_expectation_bound}
    can be bounded by $A_p \beta_p^t N^{-(t\wedge p)/2}$ for some constant $A_p$, which is not the same $A_p$
    as in the statement of Lemma~\ref{lem:rejection_expectation_bound}.
  \end{proof}

  \subsubsection{Proofs of Theorem~\ref{thm:convergence_rate_from_q} and Corollary~\ref{cor:convergence_rate_from_x}\label{appx:proof:thm:convergence_rate_from_q}}
  
  \begin{proof}[Proof of Theorem~\ref{thm:convergence_rate_from_q}]
    Under Assumption~\ref{asm:abscontinuitypositivity}, the PIMH chain is $\bar{\pi}$-irreducible, and by construction it is aperiodic and $\bar{\pi}$-invariant, therefore 
$|\bar{q}P^t-\bar{\pi}|_{\mathrm{TV}}\to 0$ as $t\to\infty$ \citep[Theorem 4 in][]{roberts2004general}.
    Thus, for any $t\geq 0$, by the triangle inequality,
    \begin{equation*}
      |\bar{q}P^t - \bar{\pi}|_{\mathrm{TV}} \leq \sum_{j=1}^\infty |\bar{q}P^{t+j} - \bar{q}P^{t+j-1}|_{\mathrm{TV}}.
    \end{equation*}
    By the coupling representation of the TV distance, for any $t\geq 0$, $j\geq 1$,
    \begin{equation*}
      |\bar{q}P^{t+j} - \bar{q}P^{t+j-1}|_{\mathrm{TV}} \leq \mathbb{E}[\mathds{1}(\bfx_{t+j}\neq \bfy_{t+j-1})] = \mathbb{P}(\tau > t +j),
    \end{equation*}
    where $(\bfx_t)$ and $(\bfy_t)$ are jointly generated by Algorithm~\ref{alg:coupledPIMHlag}. Under Assumption~\ref{asm:pfinitemoments} with $p\geq 2$, Proposition~\ref{prop:upb_meetingproba} applies and thus the series $\sum_{j=1}^\infty \mathbb{P}(\tau > t + j)$ converges.
    Thus, by the dominated convergence theorem we may swap expectation and limit to write 
    \begin{equation*}
      |\bar{q}P^t - \bar{\pi}|_{\mathrm{TV}} \leq \mathbb{E}\left[\sum_{j=1}^\infty \mathds{1}(\bfx_{t+j}\neq \bfy_{t+j-1})\right]=\mathbb{E}\left[\max(0,\tau-t-1)\right],
    \end{equation*}
    for all $t\geq 0$. This is \eqref{eq:tvupperbound_from_laggedchains}.
    
    We may express the expectation of a non-negative variable as a series of survival probabilities:
\[
    \mathbb{E}\left[\max\left(0,  \tau - t - 1\right)\right] = \sum_{s=1}^{\infty} \mathbb{P}\left(\max\left(0,  \tau - t - 1 \right) \geq s\right).
\]
    For any $t\geq 0,s\geq 1$, $\max(0,\tau-1-t)\geq s$ if and only if $\tau> s+t$.
    Under Assumption~\ref{asm:pfinitemoments}, Proposition~\ref{prop:upb_meetingproba} obtains
\[
    \mathbb{P}\left(\tau > s + t\right) \leq CN^{-1/2} (s + t)^{-p}.
\]
    The series $\sum_{s=1}^\infty (s+t)^{-p}$ can be bounded as follows:
    \begin{align}
      \sum_{s=1}^{\infty} (s+t)^{-p} &\leq (1+t)^{-p} + \int_{1+t}^{\infty} x^{-p} \diff x  \nonumber\\
      &= (1+t)^{-p+1} \left(\frac{(1+t/p)p}{(1+t)(p-1)} \right)\nonumber\\
      &\leq \frac{p}{(p-1)(1+t)^{p-1}},\label{eq:series_bound_in_proof_of_thm4.1}
    \end{align}
    using the fact that $f(k) \leq \int_{k-1}^k f(x)dx$ for any decreasing function $f$.
    Thus, for $t\geq 0$,
\[
    |\bar{q}P^t - \bar{\pi}|_{\mathrm{TV}} \leq \frac{Cp}{\sqrt{N}(p-1)(1+t)^{p-1}},
\]
    which completes the proof.
  \end{proof}
  
  \begin{proof}[Proof of Corollary~\ref{cor:convergence_rate_from_x}]
    The proof starts with multiple applications of the triangle inequality, Theorem~\ref{thm:couplingIMH}, $\max(a,b)\leq a + b$ for $a,b\geq 0$:
    \begin{align*}
      \left| P^t(\bfxsmall,\cdot)-\bar{\pi}\right|_{\mathrm{TV}} &\le \left|P^t(\bfxsmall,\cdot)-\bar{q}P^t\right|_{\mathrm{TV}}+\left|\bar{q}P^t-\bar{\pi}\right|_{\mathrm{TV}}\\
      &\le \displaystyle\int \left|P^t(\bfxsmall,\cdot)-P^t(\bfysmall,\cdot)\right|_{\mathrm{TV}}\bar{q}(\diff \bfysmall)+\left|\bar{q}P^t-\bar{\pi}\right|_{\mathrm{TV}}\\
      &= \displaystyle\int \max\left(r(\bfxsmall),r(\bfysmall)\right)^t\bar{q}(\diff \bfysmall)+\left|\bar{q}P^t-\bar{\pi}\right|_{\mathrm{TV}}\\
      & \le \left(r(\bfxsmall)\right)^t+ \mathbb{E}_{\bar{q}}[\left(r(\bfy)\right)^t]+\left|\bar{q}P^t-\bar{\pi}\right|_{\mathrm{TV}}.
    \end{align*}
    The result then follows from  Proposition~\ref{prop:expected_rejection_probability_bound} and Theorem~\ref{thm:convergence_rate_from_q}.
  \end{proof}

  \subsubsection{Proofs of Theorem~\ref{thm:convergence_rate_from_q_less2}, Corollary~\ref{cor:convergence_rate_from_x_less2} and Corollary~\ref{prop:convimh_from_x}}\label{appx:proofs:PIMH_less2moments}

  We first treat PIMH in the regime $p\in(1,2)$, and then prove the IMH specialization in Corollary~\ref{prop:convimh_from_x}.

  The following is an adaptation of Lemma~\ref{lem:large-Z-bound}
  in the case $p\in(1,2)$.

  \begin{lem}\label{lem:large-Z-bound_less2}
    Under Assumptions~\ref{asm:abscontinuitypositivity}-\ref{asm:pfinitemoments} with $p\in(1,2)$, there exists a constant $B>0$ such that for all $t\geq 1$ and $N\geq 1$,
    \[
    \mathbb{E}_{\bar{q}}\left[r(\bfx)^t\mathds{1}(\ZN(\bfx)>2)\right]
    \leq \frac{B}{N^{p-1}t^p}.
    \]
  \end{lem}
  
  \begin{proof}
    We split the expectation into two parts:
    \begin{align*}
      \mathbb{E}_{\bar{q}}\left[r(\bfx)^t\mathds{1}(\ZN(\bfx)>2)\right]
      &= \mathbb{E}_{\bar{q}}\left[r(\bfx)^t\mathds{1}(\ZN(\bfx)\in(2,1+t))\right]\\
      &+ \mathbb{E}_{\bar{q}}\left[r(\bfx)^t\mathds{1}(\ZN(\bfx)\geq 1+t)\right].
    \end{align*}
    For $\{\ZN(\bfx)\geq 1+t\}$, we directly apply Proposition~\ref{prop:zhatconcentration_less2}:
    \begin{align*}
      \mathbb{E}_{\bar{q}}\left[r(\bfx)^t\mathds{1}(\ZN(\bfx)\geq 1+t)\right]
      &\leq \mathbb{P}_{\bar{q}}\left(\ZN(\bfx)\geq 1+t\right)
      \leq \frac{2 q(|\omega-1|^p)}{N^{p-1}t^p}.
    \end{align*}
    For $\{\ZN(\bfx)\in(2,1+t)\}$, we use Lemma~\ref{lem:rupb} with $\theta=1/2$:
    \[
    r(\bfxsmall) \leq 1-\frac{c_p(1/2)}{2\ZN(\bfxsmall)} \leq \exp\left(-\frac{c_p(1/2)}{2\ZN(\bfxsmall)}\right).
    \]
    Let $c=c_p(1/2)/4$. Since $\ZN(\bfxsmall)>2$ implies $\ZN(\bfxsmall)\leq 2(\ZN(\bfxsmall)-1)$, we have
    \[
    r(\bfxsmall)^t \leq \exp\left(-\frac{ct}{\ZN(\bfxsmall)-1}\right).
    \]
    Introduce the sets $A_k=[t/(k+1),t/k]$ for $k\geq 1$, so that $\cup_{k=1}^{\infty}A_k=(0,t]$ and in particular covers $(1,t)$. Then
    \begin{align*}
      \mathbb{E}_{\bar{q}}\left[r(\bfx)^t\mathds{1}(\ZN(\bfx)\in(2,1+t))\right]
      &\leq \mathbb{E}_{\bar{q}}\left[\exp\left(-\frac{ct}{\ZN(\bfx)-1}\right)\mathds{1}(\ZN(\bfx)-1\in(1,t))\right]\\
      &\leq \sum_{k=1}^{\infty}\exp(-ck)\,\mathbb{P}_{\bar{q}}\left(\ZN(\bfx)\geq 1+\frac{t}{k+1}\right)\\
      &\leq \sum_{k=1}^{\infty}\exp(-ck)\frac{2q(|\omega-1|^p)}{N^{p-1}}\left(\frac{k+1}{t}\right)^p,
    \end{align*}
    where the last inequality uses Proposition~\ref{prop:zhatconcentration_less2}. Since $\sum_{k=1}^{\infty} e^{-ck}(k+1)^p<\infty$, this term is also bounded by a constant multiple of $(N^{p-1}t^p)^{-1}$. Combining both parts completes the proof.
  \end{proof}

The next lemma is similar to Lemma~\ref{lem:rejection_expectation_bound} but adapted for the case $p\in(1,2)$.
  \begin{lem}\label{lem:small-Z-bound_less2}
    Under Assumptions~\ref{asm:abscontinuitypositivity}-\ref{asm:pfinitemoments} with $p\in(1,2)$, define
    \begin{equation}\label{eq:betap_less2}
      \beta_p:=1-\frac{1}{2^{\frac{3p-2}{p-1}}q(\omega^p)^{\frac{1}{p-1}}}.
    \end{equation}
    Then $\beta_p\in(0,1)$, and there exists a constant $A>0$ such that for all $t\geq 1$ and $N\geq 1$,
    \[
    \mathbb{E}_{\bar{q}}\left[r(\bfx)^t\mathds{1}(\ZN(\bfx)\leq 2)\right]
    \leq \frac{A}{N^{(p-1)/p}}\beta_p^t.
    \]
  \end{lem}
  
  \begin{proof}
    We abuse notation slightly and write $r(z)$ for the rejection probability as a function of the value $z=\ZN(\bfx)$. Since $r(z)$ is non-decreasing in $z$, on $\{\ZN(\bfx)\leq 2\}$ we have
    \[
    r(\bfx)^t\mathds{1}(\ZN(\bfx)\leq 2)
    \leq r(2)^{t-1} r(\bfx)\mathds{1}(\ZN(\bfx)\leq 2).
    \]
    By Lemma~\ref{lem:rupb} with $\theta=1/2$, we have
    \[
    r(2)\leq 1-\frac{c_p(1/2)}{4}=\beta_p.
    \]
    It remains to bound $\mathbb{E}_{\bar{q}}\left[r(\bfx)\mathds{1}(\ZN(\bfx)\leq 2)\right]$.
    
    Let $\ZN(\bfxprime)$ be an independent copy of $\ZN(\bfx)$ under $\bar{q}$. Writing $\nu$ for the law of $\ZN(\bfxprime)$, for any $z>0$,
    \[
    r(z)=1-\int \min\left\{1,\frac{z^\ast}{z}\right\}\nu(\diff z^\ast)
    = \int_0^z \frac{z-z^\ast}{z}\nu(\diff z^\ast).
    \]
    Therefore,
    \[
    r(\ZN(\bfx))
    = \frac{1}{\ZN(\bfx)}\mathbb{E}_{\bar{q}}\left[\left(\ZN(\bfx)-\ZN(\bfxprime)\right)_+\,\middle|\,\ZN(\bfx)\right].
    \]
    We split the event $\{\ZN(\bfx)\leq 2\}$ into $\{\ZN(\bfx)<1/2\}\cup\{1/2\leq \ZN(\bfx)\leq 2\}$.
    
    On $\{1/2\leq \ZN(\bfx)\leq 2\}$ we have $1/\ZN(\bfx)\leq 2$, hence
    \[
    r(\ZN(\bfx))\mathds{1}(1/2\leq \ZN(\bfx)\leq 2)
    \leq 2\mathbb{E}_{\bar{q}}\left[\left|\ZN(\bfx)-\ZN(\bfxprime)\right|\,\middle|\,\ZN(\bfx)\right].
    \]
    Taking expectations and using Proposition~\ref{prop:zhatconcentration_less2},
    \begin{align*}
      \mathbb{E}_{\bar{q}}\left[r(\bfx)\mathds{1}(1/2\leq \ZN(\bfx)\leq 2)\right]
      &\leq 2\mathbb{E}_{\bar{q}\otimes\bar{q}}\left[\left|\ZN(\bfx)-\ZN(\bfxprime)\right|\right]\\
      &\leq 4\cdot 2^{1/p} \cdot q(\omega^p)^{1/p}N^{-(p-1)/p}.
    \end{align*}
    On $\{\ZN(\bfx)<1/2\}$, we use the crude bound $r\leq 1$ and Proposition~\ref{prop:zhatconcentration_less2}:
    \begin{align*}
      \mathbb{E}_{\bar{q}}\left[r(\bfx)\mathds{1}(\ZN(\bfx)<1/2)\right]
      &\leq \mathbb{P}_{\bar{q}}\left(\ZN(\bfx)<1/2\right)
      \leq \mathbb{P}_{\bar{q}}\left(|\ZN(\bfx)-1|\geq 1/2\right)\\
      &\leq \frac{2^{p+1} q(|\omega-1|^p)}{N^{p-1}}.
    \end{align*}
    Since $p-1>(p-1)/p$ and $N\geq 1$, the last term is $O(N^{-(p-1)/p})$ and can be absorbed into the previous bound. Thus there exists $K>0$, depending only on $p$ and $q(\omega^p)$, such that
    \[
    \mathbb{E}_{\bar{q}}\left[r(\bfx)\mathds{1}(\ZN(\bfx)\leq 2)\right]
    \leq \frac{K}{N^{(p-1)/p}}.
    \]
    Combining this with $r(2)\leq \beta_p$ yields
    \[
    \mathbb{E}_{\bar{q}}\left[r(\bfx)^t\mathds{1}(\ZN(\bfx)\leq 2)\right]
    \leq \frac{K}{\beta_p}N^{-(p-1)/p}\beta_p^t,
    \]
    which proves the claim.
  \end{proof}
  
The following result is analogous to Proposition~\ref{prop:expected_rejection_probability_bound} but adapted for the case $p\in(1,2)$.

  \begin{prop}\label{prop:expected_rejection_probability_bound_less2}
    Under Assumptions~\ref{asm:abscontinuitypositivity}-\ref{asm:pfinitemoments} with $p\in(1,2)$, there exists a constant $C>0$ such that for all $N\geq 1$ and all $t\geq 1$,
    \[
    \mathbb{E}_{\bar{q}}\left[r(\bfx)^t\right] \leq 
    \frac{C}{N^{(p-1)/p}t^p}.
    \]
  \end{prop}
  
  \begin{proof}
    We combine the bounds from Lemmas~\ref{lem:large-Z-bound_less2} and \ref{lem:small-Z-bound_less2} and get 
    \[
    \mathbb{E}_{\bar{q}}\left[r(\bfx)^t\right]
    \leq \frac{A}{N^{(p-1)/p}}\beta_p^t + \frac{B}{N^{p-1}t^p}.
    \]
    Since $\beta_p$ defined in \eqref{eq:betap_less2} is in $(0,1)$, the quantity $\sup_{t\geq 1} t^p\beta_p^t$ is finite; call it $K_{\beta,p}$. Hence $\beta_p^t\leq K_{\beta,p}/t^p$ for all $t\geq 1$. Also $N^{-(p-1)}\leq N^{-(p-1)/p}$ for $N\geq 1$ because $p-1>(p-1)/p$. Therefore, we can pick $C = AK_{\beta,p}  + B$ and conclude that
    \[\mathbb{E}_{\bar{q}}\left[r(\bfx)^t\right] \leq 
    \frac{C}{N^{(p-1)/p}t^p}.\]
  \end{proof}

  The next result is analogous to Proposition~\ref{prop:upb_meetingproba} but adapted for the case $p\in(1,2)$.

  \begin{prop}\label{prop:upb_meetingproba_less2}
    Consider $\tau$ generated by Algorithm~\ref{alg:coupledPIMHlag}. Under Assumptions~\ref{asm:abscontinuitypositivity}-\ref{asm:pfinitemoments} with $p\in(1,2)$, there exists a finite $C>0$ such that for all $N\geq 1$ and all $t\geq 1$,
    \[
    \mathbb{P}(\tau>t) \leq \frac{C}{N^{(p-1)/p}t^p}.
    \]
  \end{prop}
  
  \begin{proof}
    Combining Proposition~\ref{prop:meetingtimes_rejproba} and Proposition~\ref{prop:expected_rejection_probability_bound_less2} immediately yields the claim.
  \end{proof}
  
  \begin{proof}[Proof of Theorem~\ref{thm:convergence_rate_from_q_less2}]
    We start as in the proof of Theorem~\ref{thm:convergence_rate_from_q},
    up to 
    the following bound, for any $t\geq 0$ and $j\geq 1$,
    \[
    |\bar{q}P^{t+j}-\bar{q}P^{t+j-1}|_{\mathrm{TV}}
    \leq \mathbb{P}(\tau>t+j).
    \]
    Under the present assumptions, Proposition~\ref{prop:upb_meetingproba_less2} applies, and therefore
    \[
    |\bar{q}P^t-\bar{\pi}|_{\mathrm{TV}} \leq \sum_{j=1}^{\infty} \mathbb{P}(\tau>t+j)
    \leq \frac{C}{N^{(p-1)/p}}\sum_{j=1}^{\infty} (t+j)^{-p}.
    \]
    Finally, as in \eqref{eq:series_bound_in_proof_of_thm4.1} in the proof of Theorem~\ref{thm:convergence_rate_from_q}, we can bound
    \begin{align*}
      \sum_{j=1}^{\infty}(t+j)^{-p}
      &
      \leq \frac{p}{(p-1)(1+t)^{p-1}},
    \end{align*}
    which gives the result.
  \end{proof}
  
  \begin{proof}[Proof of Corollary~\ref{cor:convergence_rate_from_x_less2}]
    Similarly to the proof of Corollary~\ref{cor:convergence_rate_from_x}, we start from
    \begin{align*}
      \left|P^t(\bfxsmall,\cdot)-\bar{\pi}\right|_{\mathrm{TV}}
      &\leq \left|P^t(\bfxsmall,\cdot)-\bar{q}P^t\right|_{\mathrm{TV}} + \left|\bar{q}P^t-\bar{\pi}\right|_{\mathrm{TV}}\\
      &\leq \int \left|P^t(\bfxsmall,\cdot)-P^t(\bfysmall,\cdot)\right|\bar{q}(\diff \bfysmall) + \left|\bar{q}P^t-\bar{\pi}\right|_{\mathrm{TV}}\\
      &= \int \max\left(r(\bfxsmall),r(\bfysmall)\right)^t\bar{q}(\diff \bfysmall)+\left|\bar{q}P^t-\bar{\pi}\right|_{\mathrm{TV}}\\
      &\leq \left(r(\bfxsmall)\right)^t + \mathbb{E}_{\bar{q}}\left[\left(r(\bfy)\right)^t\right] + \left|\bar{q}P^t-\bar{\pi}\right|_{\mathrm{TV}}.
    \end{align*}
    The result then follows from Proposition~\ref{prop:expected_rejection_probability_bound_less2} and Theorem~\ref{thm:convergence_rate_from_q_less2}.
  \end{proof}

  For $N=1$ the rejection probability can be controlled directly, without the von Bahr--Esseen inequality, which gives the shorter argument below.

  \begin{proof}[Proof of Corollary~\ref{prop:convimh_from_x}]
    Similarly to the proof of Corollary~\ref{cor:convergence_rate_from_x},
    we start from
    \begin{equation*}
      | P^t(x,\cdot)-\pi|_{\mathrm{TV}} = | P^t(x,\cdot)-\pi P^t|_{\mathrm{TV}}  \leq r(x)^t + \mathbb{E}_{\y\sim \pi}[r(\y)^t].
    \end{equation*}
    Let $p>1$. We can apply Lemma~\ref{lem:rupb} to obtain
    \begin{equation*}
      r(y) \leq 1 - \min\left\{1,\frac{\theta}{\omega(y)}\right\} \frac{(1-\theta)^{p/(p-1)}}{q(\omega^p)^{1/(p-1)}},
    \end{equation*}
    and we set $\theta = 1/2$,
    and $c = (1-\theta)^{p/(p-1)}/q(\omega^p)^{1/(p-1)}$. Note that $c \leq 1$ as $q(\omega^p) \geq q(\omega)^p = 1$.
    We next bound the expected rejection probability as follows
    \begin{align*}
      \mathbb{E}_{\y \sim \pi}\left[r(\y)^t\right]
      &\leq \mathbb{E}_{\y\sim \pi}\left[\left(1 - \min\left\{\frac{1}{2\omega(\y)},1 \right\}c\right)^t \right] \\
      &\leq \mathbb{E}_{\y \sim \pi}\left[\left(1 - \frac{c}{2}\right)^t \mathds{1}(\omega(\y) < 1)\right] \\
      &\quad + \mathbb{E}_{\y\sim \pi}\left[\left(1 - \frac{c}{2\omega(\y)}\right)^t \mathds{1}(\omega(\y)\in [1,t])\right] + \mathbb{P}\left(\omega(\y) > t\right) \\
      & \leq \left(1-\frac{c}{2}\right)^t + \mathbb{E}_{\y\sim \pi}\left[\left(1 - \frac{c}{2\omega(\y)}\right)^t \mathds{1}(\omega(\y)\in [1,t])\right] + \frac{\tilde C}{t^{p-1}}\\
      & \leq \left(1-\frac{c}{2}\right)^t + \mathbb{E}_{\y\sim \pi}\left[\exp\{- Ct/\omega(\y)\}\mathds{1}(\omega(\y)\in [1,t])\right] + \frac{\tilde C}{t^{p-1}}.
    \end{align*}
    The second inequality follows by splitting the weight into $\omega < 1, \omega \in [1, t]$ and $\omega > t$. The third inequality employs Markov's inequality and the assumption that $p>1$. The last inequality uses $\log(1 + x)\leq x$ with $x = -c/(2\omega(\y))$ and $C = c/2$.
    Consider the three terms on the last line. The first term decays exponentially fast with $t$, the third term decays at the rate of $t^{-(p-1)}$. It remains to bound the second term.

    Define $A_k := [t/(k+1), t/k]$, then clearly $\cup_{k=1}^\infty A_k = (0,t]$. We bound the second term as follows:
    \begin{align*}
      \mathbb{E}_{\y\sim \pi}\left[\exp\{-Ct/\omega(\y)\}\mathds{1}(\omega(\y)\in [1,t])\right] & \leq  \mathbb{E}_{\y\sim \pi}\left[\exp\{-Ct/\omega(\y)\}\mathds{1}(\omega(\y)\in (0,t])\right]\\
      & = \sum_{k=1}^{\infty}\mathbb{E}_{\y\sim \pi}\left[\exp\{-Ct/\omega(\y)\}\mathds{1}(\omega(\y)\in A_k)\right]\\
      & \leq\sum_{k=1}^\infty  \exp\{-Ct/(t/k)\}\mathbb{P}(\omega(\y) \geq t/(k+1))\\
      & \leq \sum_{k=1}^\infty \exp\{-Ck\} \frac{C' (k+1)^{p-1}}{t^{p-1}}\\
      &  = \frac{C''}{t^{p-1}} \sum_{k=1}^\infty \exp\{-Ck\} (k+1)^{p-1}\\
      &\leq \frac{C'''}{t^{p-1}}.
    \end{align*}

    The last inequality holds as $ \sum_{k=1}^\infty \exp\{-Ck\} (k+1)^{p-1} < \infty$ (the terms inside the summation decay exponentially fast). This concludes the proof.
  \end{proof}

   \subsubsection{Proof of the result in Example \ref{example:lowerbound-generalp}}\label{appx:example:lowerbound-generalp}
  
  Given
  \[
  \pi(x) := Z_\pi x^{-p}, \qquad q(x) := Z_q \log^k(x) x^{-(p+1)}, \qquad x\in[2,\infty),
  \]
  with $p> 1$ and $k>1/(p-1)$. We have
  \[
  \omega(x) = \frac{\pi(x)}{q(x)} = \frac{Z_\pi}{Z_q}\,\frac{x}{\log^k(x)}.
  \]
  We first verify that $q(\omega^p)<\infty$. Since
  \[
  q(\omega^p)=Z_q\left(\frac{Z_\pi}{Z_q}\right)^p \int_2^\infty \frac{x^p}{\log^{kp}(x)}\frac{\log^k(x)}{x^{p+1}}\diff x
  = Z_q\left(\frac{Z_\pi}{Z_q}\right)^p \int_2^\infty \frac{1}{x\log^{k(p-1)}(x)}\diff x,
  \]
  the change of variable $u=\log x$ gives
  \[
  q(\omega^p)= Z_q\left(\frac{Z_\pi}{Z_q}\right)^p \int_{\log 2}^\infty u^{-k(p-1)}\diff u<\infty,
  \]
  since $k(p-1)>1$.
  
  Now we estimate $\bP_{X\sim \pi}(X>s)$ and $\bP_{X\sim q}(X>s)$ for any $s>2$. For the former,
  \begin{align*}
    \bP_{X\sim \pi}(X>s)
    &= Z_\pi \int_s^\infty x^{-p}\diff x
    = \frac{C_1}{s^{p-1}},
    \qquad C_1:=\frac{Z_\pi}{p-1}.
  \end{align*}
  For the latter,
  \begin{align*}
    \bP_{X\sim q}(X>s)
    &= Z_q \int_s^\infty \frac{\log^k(x)}{x^{p+1}}\diff x
    = Z_q s^{-p}\int_1^\infty \log^k(su)u^{-(p+1)}\diff u\\
    &\leq Z_q 2^{k-1}s^{-p}\int_1^\infty \frac{\log^k(s)+\log^k(u)}{u^{p+1}}\diff u\\
    &\leq Z_q 2^{k-1}s^{-p}\left(\frac{\log^k(s)}{p}+\kappa_{p,k}\right)
    \leq \frac{C_2\log^k(s)}{s^p},
  \end{align*}
  where
  \[
  \kappa_{p,k}:=\int_1^\infty \log^k(u)u^{-(p+1)}\diff u<\infty,
  \qquad
  C_2:= Z_q 2^{k-1}\left(\frac{1}{p}+\frac{\kappa_{p,k}}{\log^k(2)}\right),
  \]
  and we used $\log^k(s)\geq \log^k(2)$ for all $s\geq 2$.
  
  Consider an IMH chain $(X_t)_{t\geq 0}$ targeting $\pi$ with proposal $q$, started from $x_0= 3$. Fix any $t\geq 2$, define
  \[
  s_t:= t(\log t)^{k+1}, \qquad A_t:=(s_t,\infty).
  \]
  Then the probability of $A_t$ under $\pi$ is
  \[
  \bP_{X\sim \pi}(X\in A_t)=\frac{C_1}{t^{p-1}(\log t)^{(k+1)(p-1)}}.
  \]
  On the other hand, $X_t\in A_t$ implies that at least one of $x_0,Y_1,\ldots,Y_t$ falls in $A_t$, where $(Y_j)_{j\geq 1}$ are the proposals. For $t\geq 3$, we have $x_0 \notin A_t$, thus, by the union bound,
  \[
  \bP(X_t\in A_t)\leq t\bP_{Y\sim q}(Y\in A_t)
  \leq t\frac{C_2\log^k(s_t)}{s_t^p}.
  \]
  Let $t_1=t_1(k)$ be large enough that $t_1 \geq 3$ and $(k+1)\log\log t\leq \log t$ for all $t\geq t_1$. Then, for $t\geq t_1$,
  \[
  \log(s_t)=\log t +(k+1)\log\log t \leq 2\log t,
  \]
  so that $\log^k(s_t)\leq 2^k(\log t)^k$. Therefore, for all $t\geq t_1$,
  \begin{align*}
    \bP(X_t\in A_t)
    &\leq t\cdot \frac{C_2 2^k(\log t)^k}{t^p(\log t)^{p(k+1)}}
    = \frac{2^{k}C_2}{t^{p-1}(\log t)^{p(k+1)-k}}.
  \end{align*}
  Since
  \[
  p(k+1)-k-(k+1)(p-1)=1,
  \]
  we have
  \[
  \bP(X_t\in A_t)
  \leq \frac{2^{k}C_2}{C_1 \log t}\,\bP_{X\sim \pi}(X\in A_t).
  \]
  Define
  \begin{equation}\label{eq:t0def}
    t_0:= \left\lceil \exp\!\left(\frac{2^{k+1}C_2}{C_1}\right)\right\rceil \vee t_1.
  \end{equation}
  Then for all $t\geq t_0$,
  \[
  \bP(X_t\in A_t)\leq \frac{1}{2}\bP_{X\sim \pi}(X\in A_t).
  \]
  Therefore, for all $t\geq t_0$,
  \begin{align*}
    \left|P^t(x_0, \cdot)-\pi\right|_{\mathrm{TV}}
    &\geq \bP_{X\sim \pi}(X\in A_t)-\bP(X_t\in A_t)\\
    &\geq \frac{1}{2}\bP_{X\sim \pi}(X\in A_t)
    = \frac{C_1/2}{t^{p-1}(\log t)^{(k+1)(p-1)}}.
  \end{align*}
  This proves the claim.

  \subsection{Proofs of Section~\ref{sec:biasremoval_snis}\label{appx:proofs:biasremoval_snis}}

  \subsubsection{Proof of Proposition~\ref{prop:unbiasedsnis_has_finitemoments}}\label{appx:proof:unbiasedsnis_has_finitemoments}
  \begin{proof}
    Note that $\FNu$ is not bounded even if $|f|_{\infty}\le1$, because the sum in \eqref{eq:def:unbiasedsnis} can be arbitrarily large.
    By Minkowski's inequality, for any $s\geq 1$,
    \begin{equation*}
      \mathbb{E}\left[|\FNu|^s\right]^{1/s} \leq 
      \mathbb{E}\left[| \hat{F}(\bfx_0)|^s\right]^{1/s}
      + \mathbb{E}\left[\left|\sum_{t=1}^{\tau - 1} \{\hat{F}(\bfx_t) - \hat{F}(\bfy_{t-1})\}\right|^s\right]^{1/s}.
    \end{equation*}
    Furthermore, if $|f|_{\infty}\le1$ then $|\hat{F}(\bfxsmall)|\leq 1$ for all $\bfxsmall$, thus 
    \begin{equation*}
      \mathbb{E}\left[|\FNu|^s\right]^{1/s} \leq 
      \mathbb{E}\left[| \hat{F}(\bfx_0)|^s\right]^{1/s}
      + 2 \mathbb{E}\left[\mathds{1}(\tau > 1) \left|\tau - 1\right|^s\right]^{1/s}.
    \end{equation*}
    Since $\hat{F}(\bfx_0) \leq 1$ almost surely, $\mathbb{E}[|\hat{F}(\bfx_0)|^s]^{1/s}$ is finite for all $s\geq 1$.
    The latter expectation is smaller than $\mathbb{E}[|\tau|^s]^{1/s}$. Thus, $\FNu$ has $s$ finite moments if $\tau$ has $s$ finite moments. 
    Note that $\FNu$ can have higher moments as well: for example, if $f$ is constant, then $\FNu$ is constant.
    
    Next, in order for $\tau$ to have $s\geq 1$ moments, we can resort to Proposition~\ref{prop:upb_meetingproba}. If Assumption~\ref{asm:pfinitemoments} holds with $p>s$,
    then $\mathbb{P}(\tau>t)\leq CN^{-1/2}t^{-p}$. We can then follow the proof of Proposition 8 in \citet{poisson22}, using Tonelli's theorem:
    \begin{align*}
      \mathbb{E}\left[\tau^s\right] &= \mathbb{E}\left[\int_0^\infty \mathds{1}(u < \tau)s u^{s-1}\mathrm{d}u\right]
      = \int_0^\infty s u^{s-1}\mathbb{P}(\tau > u) \mathrm{d}u\\
      &= \sum_{i=0}^\infty \mathbb{P}(\tau > i) \int_{i}^{i+1} s u^{s-1} \mathrm{d}u
      \; \leq \; \sum_{i=0}^\infty \mathbb{P}(\tau > i) s (i+1)^{s-1}.
    \end{align*}
    The sum is finite under the assumption $p>s$.
  \end{proof}

  \subsubsection{Proof of Proposition~\ref{prop:pimh}\label{appx:proof:thmpimh}}
  
  \begin{proof}[Proof of Proposition~\ref{prop:pimh}]
    
    We now consider the PIMH chain $(\bfx_t)_{t\geq 0}$, started from $\bar{q}$.
    The case $t=0$ corresponds to Theorem~\ref{thm: snis unbounded convergence}. Let $t\geq 1$.
    We can assume that $f$ is non-negative, using the same separate treatment of $f_+$ and $f_{-}$ as in the beginning of the proof of Theorem~\ref{thm: snis unbounded convergence}.
    
    We write
    \begin{equation*}
      \FN^\circ: \bfxsmall \mapsto \FN(\bfxsmall) - \pi(f) = \frac{\sum_{n=1}^N \omega(x_n) \{f(x_n)-\pi(f)\}}{\sum_{m=1}^N \omega(x_m)}.
    \end{equation*}
    We can write 
    \begin{align*}
      &\bE_{\bfx_0 \sim \bar{q}}\left[|\FN(\bfx_t)-\pi(f)|^s\right]\\
      &= \int \bar{q}(\mathrm{d}\bfxsmall_0)P(\bfxsmall_0,\mathrm{d}\bfxsmall_1)\ldots P(\bfxsmall_{t-1},\mathrm{d}\bfxsmall_t) |\FN^\circ(\bfxsmall_t)|^s\\
      &= \int \bar{q}(\mathrm{d}\bfxsmall_0)P(\bfxsmall_0,\mathrm{d}\bfxsmall_1)\ldots P(\bfxsmall_{t-1},\mathrm{d}\bfxsmall_t) |\FN^\circ(\bfxsmall_t)|^s \{\mathds{1}(A_t) + \mathds{1}(A_t^c)\},
    \end{align*}
    where the event $A_t$ represents ``there was an acceptance in the first $t$ steps''. 
    
    In the event $A_t^c$, $\bfx_t = \bfx_0$ so 
    \begin{align*}
      & \int \bar{q}(\mathrm{d}\bfxsmall_0)P(\bfxsmall_0,\mathrm{d}\bfxsmall_1)\ldots P(\bfxsmall_{t-1},\mathrm{d}\bfxsmall_t) |\FN^\circ(\bfxsmall_t)|^s \cdot \mathds{1}(A_t^c) \\
      &= \int \bar{q}(\mathrm{d}\bfxsmall_0)P(\bfxsmall_0,\mathrm{d}\bfxsmall_1)\ldots P(\bfxsmall_{t-1},\mathrm{d}\bfxsmall_t) |\FN^\circ(\bfxsmall_0)|^s \cdot \mathds{1}(A_t^c) \\
      &\leq \int \bar{q}(\mathrm{d}\bfxsmall_0) |\FN^\circ(\bfxsmall_0)|^s,
    \end{align*}
    by bounding the indicator by one, and we can use Theorem~\ref{thm: snis unbounded convergence}
    to obtain a bound in $N^{-s/2}$.
    
    Now we consider the case $A_t$.
    For $1\leq j\leq t$  define the events 
$$A_{j,t} := \{\bfx_{j-1}\neq \bfx_{j}=\bfx_{j+1} = \ldots = \bfx_{t}\},$$
    where $A_{j,t}$ is the event that there is a jump at time $j$ and no jump after that. 
    Then $A_{j,t} \cap A_{j',t}= \emptyset$ for $j\neq j'$ and 
$A_t = \cup_{j=1}^t A_{j,t}$. 
    We can decompose $\mathds{1}(A_t)$ into $\sum_{j=1}^t \mathds{1}(A_{j,t})$ to get
    \begin{align*}
      &\int \bar{q}(\mathrm{d}\bfxsmall_0)P(\bfxsmall_0,\mathrm{d}\bfxsmall_1)\ldots P(\bfxsmall_{t-1},\mathrm{d}\bfxsmall_t) |\FN^\circ(\bfxsmall_t)|^s \mathds{1}(A_t)\\
      &= \sum_{j=1}^t \mathbb{E}_{\bfx_0\sim \bar{q}}\left[|\FN^\circ(\bfx_t)|^s \mathds{1}(A_{j,t}) \right]
      = \sum_{j=1}^t \mathbb{E}_{\bfx_0\sim \bar{q}}\left[\mathbb{E}\left\{ \left.|\FN^\circ(\bfx_j)|^s \mathds{1}(A_{j,t}) \right| \bfx_{j-1} \right\} \right].
    \end{align*}
    Conditional on $\bfxsmall_{j-1}$,
    \begin{align*}
      & \int P(\bfxsmall_{j-1}, \diff \bfxsmall_{j})
      P(\bfxsmall_{j}, \diff \bfxsmall_{j+1}) \ldots P(\bfxsmall_{t-1}, \diff \bfxsmall_{t}) |\FN^\circ(\bfxsmall_{j})|^s \mathds{1}\{\bfxsmall_{j-1}\neq \bfxsmall_{j} = \ldots = \bfxsmall_{t}\}\\
      &= \int P(\bfxsmall_{j-1}, \diff \bfxsmall_{j})
      | \FN^\circ(\bfxsmall_{j})|^s \mathds{1}\{\bfxsmall_{j-1}\neq \bfxsmall_{j}\} \\
      & \quad \cdot \int 
      P(\bfxsmall_{j}, \diff \bfxsmall_{j+1}) \ldots P(\bfxsmall_{t-1}, \diff \bfxsmall_{t})  \mathds{1}\{\bfxsmall_j = \ldots = \bfxsmall_t\}\\ 
      &= \int P(\bfxsmall_{j-1}, \diff \bfxsmall_{j})
      |\FN^\circ(\bfxsmall_{j})|^s \mathds{1}\{\bfxsmall_{j-1}\neq \bfxsmall_{j}\} r(\bfxsmall_{j})^{t-j}\\
      &= \int \bar{q}(\diff \zeta) \alpha(\bfxsmall_{j-1},\zeta)
      | \FN^\circ(\zeta)|^s  r(\zeta)^{t-j}.
    \end{align*}    
    We can then upper bound $\alpha$ by one, and upper bound $\sum_{j=1}^t r(\zeta)^{t-j}$ by $(1-r(\zeta))^{-1}$ to obtain
    \begin{align*}
      &\sum_{j=1}^t \mathbb{E}_{\bfx_0\sim \bar{q}, \zeta\sim \bar{q}}\left[ 
      \alpha(\bfx_{j-1},  \zeta) 
      |\FN^\circ(\zeta)|^s r(\zeta)^{t-j}
      \right]\\
      &\leq  \mathbb{E}_{\zeta\sim \bar{q}}\left[ 
      |\FN^\circ(\zeta)|^s \sum_{j=1}^t r(\zeta)^{t-j}
      \right]\\
      &\leq \mathbb{E}_{\zeta\sim \bar{q}}\left[ 
      |\FN^\circ(\zeta)|^s  \frac{1}{1-r(\zeta)}
      \right].
    \end{align*}
    
    Next, split the expectation into the cases $\ZN(\zeta) > 2$ and $\ZN(\zeta) \leq 2$.
    % \begin{align*}
    %   \mathbb{E}_{\zeta\sim \bar{q}}\left[ 
    %   |\FN^\circ(\zeta)|^s \frac{1}{1-r(\zeta)}
    %   \right] &= \mathbb{E}_{\zeta\sim \bar{q}}\left[ 
    %    \frac{|\FN^\circ(\zeta)|^s}{1-r(\zeta)} \mathds{1}(\ZN(\zeta) \leq 2)
    %   \right] 
    %   + \mathbb{E}_{\zeta\sim \bar{q}}\left[ 
    %    \frac{|\FN^\circ(\zeta)|^s}{1-r(\zeta)} \mathds{1}(\ZN(\zeta) > 2)
    %   \right].
    % \end{align*}
    When $\ZN(\zeta) \leq 2$, since $r$ is increasing with $\ZN$, we have $r(\zeta) \leq r(2)$ and thus $(1-r(\zeta))^{-1} \leq (1-r(2))^{-1}$. This yields:
    \begin{align*}
      \mathbb{E}_{\zeta\sim \bar{q}}\left[ 
      |\FN^\circ(\zeta)|^s \frac{1}{1-r(\zeta)} \mathds{1}(\ZN(\zeta) \leq 2)
      \right] &\leq \frac{1}{1-r(2)}\int \bar{q}(\mathrm{d}\zeta) |\FN^\circ(\zeta)|^s \mathds{1}(\ZN(\zeta) \leq 2)\\
      &\leq \frac{1}{1-r(2)}\mathbb{E}_{\mathbf{x}_0\sim \bar{q}}\left[|\FN^\circ(\bfx_0)|^s\right].
    \end{align*}
    We obtain a bound in $N^{-s/2}$ using Theorem~\ref{thm: snis unbounded convergence}.
    
    When $\ZN(\zeta) > 2$, from Lemma~\ref{lem:rupb}, we have $r(\zeta) \leq 1 - c_p(1/2)/(2\ZN(\zeta))$. 
    Thus $1-r(\zeta) \geq c_p(1/2)/(2\ZN(\zeta))$, and $(1-r(\zeta))^{-1} \leq (2/c_p(1/2))\ZN(\zeta)$. This yields:
    \begin{align*}
      \mathbb{E}_{\zeta\sim \bar{q}}\left[ 
      |\FN^\circ(\zeta)|^s \frac{1}{1-r(\zeta)} \mathds{1}(\ZN(\zeta) > 2)
      \right] \leq \frac{2}{c_p(1/2)} \int \bar{q}(\mathrm{d}\zeta) |\FN^\circ(\zeta)|^s \ZN(\zeta) \mathds{1}(\ZN(\zeta) > 2).
    \end{align*}
    Since we assume $f \geq 0$, we can use the inequality \eqref{eq:upb_max_plus_qwf}:
    \begin{align*}
      |\FN^\circ(\zeta)|^s \leq \left(\max_{1\leq i \leq N} f(\zeta_i) + q(\omega f) \right)^s,
    \end{align*}
    from which we obtain
    \begin{align*}
      \mathbb{E}_{\zeta\sim \bar{q}}&\left[ 
      |\FN^\circ(\zeta)|^s \frac{1}{1-r(\zeta)} \mathds{1}(\ZN(\zeta) > 2)
      \right] \\
      &\leq \frac{2}{c_p(1/2)}\mathbb{E}_{\zeta\sim \bar{q}}\left[ \left(\max_{1\leq i \leq N} f(\zeta_i) + q(\omega f) \right)^s\cdot  \ZN(\zeta) \mathds{1}(\ZN(\zeta) > 2)\right]\\
      &\leq \frac{2^{s}}{c_p(1/2)}\left(\mathbb{E}_{\zeta\sim \bar{q}}\left[ \left(\max_{1\leq i \leq N} f(\zeta_i)\right)^s\cdot  \ZN(\zeta) \mathds{1}(\ZN(\zeta) > 2)\right] + q(\omega f)^s \cdot \mathbb{E}_{\zeta\sim \bar{q}}\left[\ZN(\zeta) \mathds{1}(\ZN(\zeta) > 2)\right]\right).
    \end{align*}
    Using the facts that $\ZN(\zeta)\leq 2(\ZN(\zeta)-1)$ when $\ZN(\zeta) > 2$ and $\mathds{1}(\ZN(\zeta)\geq 2)\leq |\ZN(\zeta) - 1|^{p-1}$, we obtain via Proposition~\ref{prop:zhatconcentration}:
    \begin{align*}
      q(\omega f)^s\cdot \mathbb{E}_{\zeta\sim \bar{q}}\left[   \ZN(\zeta) \mathds{1}(\ZN(\zeta) > 2)\right]
      &\leq 2q(\omega f)^s \mathbb{E}_{\zeta\sim \bar{q}}\left[|\ZN(\zeta) - 1|^p\right] \\
      &\leq \frac{2M(p)^p}{N^{p/2}}q(\omega f)^s.
    \end{align*}
    For the remaining term, using Hölder's inequality yields:
    \begin{align*}
      \mathbb{E}_{\zeta\sim \bar{q}}&\left[ \left(\max_{1\leq i \leq N} f(\zeta_i)\right)^s\cdot  \ZN(\zeta) \mathds{1}(\ZN(\zeta) > 2)\right]\\
      &\leq 2\mathbb{E}_{\zeta\sim \bar{q}}\left[ \left(\max_{1\leq i \leq N} f(\zeta_i)\right)^r\right]^{s/r}\cdot \mathbb{E}_{\zeta\sim \bar{q}}\left[ \left|\ZN(\zeta)-1\right|^{\frac{r}{r-s}}\mathds{1}(\ZN(\zeta) > 2) \right]^{1-s/r}.
    \end{align*}
    Under the assumptions, with $s\leq \frac{pr}{p+r+2}$, we have:
    \begin{align*}
      \frac{r}{r-s}\leq \frac{p+r+2}{r+2}=1+\frac{p}{r+2} \leq p,
    \end{align*}
    where the inequality holds since $r \geq 2$ by assumption.
    This gives us:
    \begin{align*}
      \mathbb{E}_{\zeta\sim \bar{q}}\left[ \left|\ZN(\zeta)-1\right|^{\frac{r}{r-s}}\mathds{1}(\ZN(\zeta) > 2) \right]^{1-s/r}
      &\leq \left(\frac{M(p)^p}{N^{p/2}}\right)^{1-s/r}.
    \end{align*}
    Finally we use 
    the fact that $\max\{a_1,\ldots,a_n\}\leq a_1+\ldots+a_n$ for non-negative $a_i$ to derive
    \begin{align*}
      &\mathbb{E}_{\zeta\sim \bar{q}}\left[ \left(\max_{1\leq i \leq N} f(\zeta_i)\right)^r\right]\\
      &=\mathbb{E}_{\zeta\sim \bar{q}}\left[ \max_{1\leq i \leq N} f(\zeta_i)^r\right]
      \leq \mathbb{E}_{x\sim q}\left[f(x)^r N\right], 
    \end{align*}
    so that 
    \begin{align*}
      &\mathbb{E}_{\zeta\sim \bar{q}}\left[ \left(\max_{1\leq i \leq N} f(\zeta_i)\right)^s\cdot  \ZN(\zeta) \mathds{1}(\ZN(\zeta) > 2)\right]\\
      &\leq 2 \left(q(f^r) N\right)^{s/r} \cdot \left(\frac{M(p)^p}{N^{p/2}}\right)^{1-s/r}.
    \end{align*}
    We end up with an exponent of $N$ equal to $s/r -(p/2)(r-s)/r$,
    which under the assumptions is less than $-s/2$, as detailed in the proof of Theorem~\ref{thm: snis unbounded convergence}.
    Therefore, we obtain an upper bound in $N^{-s/2}$ on all terms.
  \end{proof}
  
  \begin{remark}
    Under Assumption~\ref{asm:abscontinuitypositivity}, PIMH converges in total variation.
    Thus, $(\bfx_t)$ converges weakly to $\bar{\pi}$. 
    We consider the transformation
$\bfxsmall \mapsto |\FN^{\circ}(\bfxsmall)|^s$ and Fatou's lemma as in Theorem 3.4 of \citet{billingsleyconvergence},
    to obtain 
\[\mathbb{E}_{\bar{\pi}}[|\FN^\circ(\bfx)|^s] \leq \lim \inf_t \mathbb{E}[|\FN^\circ(\bfx_t)|^s].\]
    Thus, the bound of Proposition~\ref{prop:pimh}, valid for all $t\geq 0$, applies also to the $s$-th moment
    of $\FN(\bfx) - \pi(f)$ under $\bar{\pi}$.
  \end{remark}
  
  \subsubsection{Proof of Proposition~\ref{prop:unbiasedsnis_has_finitemoments_unbounded}\label{appx:proof:momentsusnis_unbounded}}
  
  \begin{proof}[Proof of Proposition~\ref{prop:unbiasedsnis_has_finitemoments_unbounded}]
    We start as in the proof of Proposition~\ref{prop:unbiasedsnis_has_finitemoments} in Appendix~\ref{appx:proof:unbiasedsnis_has_finitemoments},
    and employ Theorem~\ref{thm: snis unbounded convergence} for the moments of the error of IS
    with unbounded functions.
    Regarding the bias cancellation term, 
    \begin{align*}
      \text{BC}  &=   \sum_{t=1}^{\infty} \Delta_t \mathds{1}(\tau > t),
    \end{align*}
    we use Minkowski with exponent $s \geq 1$:
    \begin{align}\label{eq:EBCs_by_Minkowski}
      \mathbb{E}\left[|\text{BC}|^s\right]^{1/s}  &\le \sum_{t=1}^{\infty} \mathbb{E}\left[|\Delta_t|^s \mathds{1}(\tau > t)\right]^{1/s}.
    \end{align}
    Next, for each time $t$, using H\"older's inequality with an arbitrary $\kappa>1$,
    \begin{equation*}
      \mathbb{E}\left[|\Delta_t|^s \mathds{1}(\tau > t)\right] \leq \mathbb{E}\left[|\Delta_t|^{s\kappa}\right]^{1/\kappa} \mathbb{P}(\tau > t)^{(\kappa-1)/\kappa}.
    \end{equation*}
    For the sum over $t$ in \eqref{eq:EBCs_by_Minkowski} to be finite,
    and using Proposition~\ref{prop:upb_meetingproba} to bound $\mathbb{P}(\tau > t)$, we have the condition on $\kappa$ and $s$,
\[-\frac{p(\kappa - 1)}{s\kappa} < -1 \quad \Leftrightarrow \quad \kappa > p / (p-s).\]
    To establish the finiteness of $\mathbb{E}\left[|\Delta_t|^{s\kappa}\right]$
    we can resort to Proposition~\ref{prop:pimh} if $s\kappa$ satisfies the condition
\[s\kappa \leq \frac{pr}{p+r+2}.\]
    We can find such $\kappa$ if 
\[\frac{ps}{p-s} < \frac{pr}{p+r+2}.\]
  \end{proof}
  
  \subsubsection{Proof of Proposition~\ref{prop:dsnis:mseequiv}\label{appx:proof:dsnis:mseequiv}}
  
  \begin{proof}[Proof of Proposition~\ref{prop:dsnis:mseequiv}]
    We follow the proof of Proposition~\ref{prop:unbiasedsnis_has_finitemoments_unbounded}, 
    with $s=2$. We 
    thus have an exponent $\kappa>1$ that must satisfy 
    $\kappa > p/(p-2)$, and $2\kappa \leq pr/(p+r+2)$. We choose any number $\kappa$
    strictly between $p/(p-2)$ and $pr/(2p+2r+4)$, which is possible by assumption, since
    \[1<\frac{p}{p-2} < \frac{pr}{2p+2r+4} \Leftrightarrow 2p+4r+4 < rp.\]
    For that $\kappa$, we can apply Proposition~\ref{prop:pimh} to bound $\mathbb{E}[|\Delta_t|^{2\kappa}]^{1/\kappa}$
    by a constant times $N^{-1}$.
    Meanwhile, the sum $\sum_{t=1}^{\infty} \mathbb{P}(\tau > t)^{(\kappa-1)/(\kappa s)}$ 
    is finite using Proposition~\ref{prop:upb_meetingproba}, and is of the form $C N^{-a}$ for some positive $a$,
    namely $a = (\kappa-1)/(2\kappa s)$.
    Thus, $\mathbb{E}{|\text{BC}|^2}$ can be bounded by a constant times $N^{-1-a}$ for some positive $a$,
    and finds itself negligible in front of the MSE of IS as $N\to\infty$.
  \end{proof}
  
  \subsubsection{Proof of Proposition~\ref{prop:inverse-CUIS} on the unbiased estimation of \texorpdfstring{$1/Z$}{1/Z}}\label{appx:proof:inverse-uis}

  The proof of Proposition~\ref{prop:inverse-CUIS} relies on the inverse moment bound of Proposition~\ref{prop:inversemoments}, which provides a finite bound on $\bE[\ZN(\bfx)^{-r}]$ for all $r \geq 1$ with no constraint on~$r$ relative to~$p$. We first establish sharper SNIS and PIMH moment bounds for the test function $f = 1/\omega_u$.
  
  \begin{prop}[SNIS moments of $1/\ZN$]\label{prop:snis-inverse}
    Assume that $q(\omega^p) < \infty$ for some $p \geq 2$ and $q(\omega^{-\eta}) < \infty$ for some $\eta > 0$. Then, for any $1 \leq s < p$ and $N \geq sp/((p-s)\eta)$, there exists $C > 0$ such that
    \begin{equation}\label{eq:snis-inverse-bound}
      \bE_{\bar{q}}\!\left[\left|\frac{1}{\ZN(\bfx)} - \frac{1}{Z}\right|^s\right] \leq C\, N^{-s/2}.
    \end{equation}
  \end{prop}
  
  \begin{proof}
    Using $|1/\ZN - 1/Z|^s = |\ZN - Z|^s / (Z^s \ZN^s)$ and H\"older's inequality with exponents $\kappa = p/s$ and $\kappa' = p/(p-s)$:
\[
    \bE\!\left[\frac{|\ZN - Z|^s}{Z^s \ZN^s}\right]
    \leq Z^{-s} \cdot \bE\!\left[|\ZN - Z|^{p}\right]^{s/p} \cdot \bE\!\left[\ZN^{-sp/(p-s)}\right]^{(p-s)/p}.
\]
    The first factor is $O(N^{-s/2})$ by Proposition~\ref{prop:zhatconcentration}. The second is bounded by a constant via Proposition~\ref{prop:inversemoments} with $r = sp/(p-s)$, for $N \geq sp/((p-s)\eta)$.
  \end{proof}
  
  \begin{prop}[PIMH moments of $1/\ZN$]\label{prop:pimh-inverse}
    Under the same conditions as Proposition~\ref{prop:snis-inverse}, there exists $C > 0$ such that for all $t \geq 0$:
    \begin{equation}\label{eq:pimh-inverse-bound}
      \bE_{\bfx_0\sim\bar{q}}\!\left[\left|\frac{1}{\ZN(\bfx_t)} - \frac{1}{Z}\right|^s\right] \leq C\, N^{-s/2},
    \end{equation}
    where $(\bfx_t)_{t\geq 0}$ is the PIMH chain initialized at $\bfx_0 \sim \bar{q}$.
  \end{prop}
  
  \begin{proof}[Proof of Proposition~\ref{prop:pimh-inverse}]
    The case $t = 0$ is Proposition~\ref{prop:snis-inverse}. For $t \geq 1$, we follow the proof of Proposition~\ref{prop:pimh} in Appendix~\ref{appx:proof:thmpimh}. The decomposition into the rejection event ($A_t^c$), small-weight regime ($\ZN(\zeta) \leq 2$), and large-weight regime ($\ZN(\zeta) > 2$) carries over. The $A_t^c$ and small-weight contributions are each bounded by a constant times $\bE_{\bar{q}}[|1/\ZN - 1/Z|^s] \leq C N^{-s/2}$ by Proposition~\ref{prop:snis-inverse}.
    
    The key simplification occurs in the large-weight regime. For $f = 1/\omega_u$, we have $|\hat{F}_N^\circ(\zeta)|^s = |1/\ZN(\zeta) - 1/Z|^s / (Z \cdot \ZN(\zeta))^s$, and using Lemma~\ref{lem:rupb} with $\theta = 1/2$, $(1-r(\zeta))^{-1} \leq 2\ZN(\zeta)/c_p(1/2)$ when $\ZN(\zeta) > 2$. The product becomes
\[
    \frac{|1/\ZN(\zeta) - 1/Z|^s}{Z^s \ZN(\zeta)^s} \cdot \frac{2\ZN(\zeta)}{c_p(1/2)} = \frac{2}{c_p(1/2)} \cdot \frac{|1/\ZN(\zeta) - 1/Z|^s}{Z^s \ZN(\zeta)^{s-1}},
\]
    where one power of $\ZN$ cancels. This is the simplification afforded by the inverse weight as a test function: for a general test function~$f$, the analogous step requires H\"older's inequality to separate $|\hat{F}_N^\circ|^s$ from $\ZN \cdot \mathds{1}(\ZN > 2)$, introducing the test function moment $q(|f|^r)$ and the constraint $s \leq pr/(p+r+2)$. For $f = 1/\omega_u$, the partial cancellation leaves only inverse moments of $\ZN$, controlled by Proposition~\ref{prop:inversemoments}.
    
    Applying H\"older's inequality with exponents $p/s$ and $p/(p-s)$:
\[
    \bE\!\left[\frac{|1/\ZN - 1/Z|^s}{Z^s \ZN^{s-1}}\right]
    \leq Z^{-s} \cdot \bE[|1/\ZN - 1/Z|^{p}]^{s/p} \cdot \bE[\ZN^{-(s-1)p/(p-s)}]^{(p-s)/p}.
\]
    The first factor is $O(N^{-s/2})$ by Proposition~\ref{prop:zhatconcentration}, and the second is bounded by a constant via Proposition~\ref{prop:inversemoments} with $r = (s-1)p/(p-s)$. Combining the three contributions ($A_t^c$, small weight, large weight) yields the bound.
  \end{proof}

  \begin{proof}[Proof of Proposition~\ref{prop:inverse-CUIS}]
    \emph{Finite $s$-th moments under $p > s$.}
    We follow the proof of Proposition~\ref{prop:unbiasedsnis_has_finitemoments_unbounded} in Appendix~\ref{appx:proof:momentsusnis_unbounded} with $f = 1/\omega_u$. By Minkowski's inequality, it suffices to bound $\bE[|\text{BC}|^s]^{1/s}$, where $\text{BC} = \sum_{t=1}^{\infty} \Delta_t \mathds{1}(\tau > t)$. Applying H\"older's inequality with exponent $\kappa > 1$ and conjugate $\kappa' = \kappa/(\kappa-1)$:
    \begin{equation}\label{eq:uis-inv-holder}
      \bE[|\Delta_t|^s \mathds{1}(\tau > t)]
      \leq \bE[|\Delta_t|^{s\kappa}]^{1/\kappa} \cdot \bP(\tau > t)^{(\kappa-1)/\kappa}.
    \end{equation}
    For general $f$, bounding $\bE[|\Delta_t|^{s\kappa}]$ requires $s\kappa \leq pr/(p+r+2)$, imposing an upper bound on $\kappa$. For $f = 1/\omega_u$, we show that a constant bound on $\bE[|\Delta_t|^{s\kappa}]$ holds for \emph{all} $\kappa$, so that $\kappa$ can be chosen freely.
    
    By the triangle inequality,
    \begin{equation}\label{eq:uis-inv-triangle}
      |\Delta_t|
      = \left|\frac{1}{\ZN(\bfx_t)} - \frac{1}{\ZN(\bfy_{t-1})}\right|
      \leq \frac{1}{\ZN(\bfx_t)} + \frac{1}{\ZN(\bfy_{t-1})},
    \end{equation}
    so by Minkowski's inequality in $L^{s\kappa}$,
    \begin{equation}\label{eq:uis-inv-mink}
      \bE[|\Delta_t|^{s\kappa}]^{1/(s\kappa)}
      \leq \bE\!\left[\ZN(\bfx_t)^{-s\kappa}\right]^{1/(s\kappa)}
      + \bE\!\left[\ZN(\bfy_{t-1})^{-s\kappa}\right]^{1/(s\kappa)}.
    \end{equation}
    We now show that $\bE[\ZN(\bfx_t)^{-s\kappa}] \leq \bE_{\bar{q}}[\ZN(\bfx)^{-s\kappa}]$ for all $t \geq 0$, by induction on the PIMH chain. The base case $t = 0$ is immediate since $\bfx_0 \sim \bar{q}$. For the inductive step, recall that the Metropolis--Hastings update at time~$t$ draws a proposal $\boldsymbol{\zeta}_t \sim \bar{q}$ independently and accepts it with probability $\min(1, \ZN(\boldsymbol{\zeta}_t)/\ZN(\bfx_{t-1}))$. Conditioning on $\bfx_{t-1}$ and $\boldsymbol{\zeta}_t$:
    \begin{align}\label{eq:uis-inv-onestep}
      \bE\!\left[\ZN(\bfx_t)^{-s\kappa} \mid \bfx_{t-1}, \boldsymbol{\zeta}_t\right]
      &= \min\!\left(1, \frac{\ZN(\boldsymbol{\zeta}_t)}{\ZN(\bfx_{t-1})}\right) \ZN(\boldsymbol{\zeta}_t)^{-s\kappa} \nonumber\\
      &\quad + \left(1 - \min\!\left(1, \frac{\ZN(\boldsymbol{\zeta}_t)}{\ZN(\bfx_{t-1})}\right)\right) \ZN(\bfx_{t-1})^{-s\kappa}.
    \end{align}
    This is a convex combination. If $\ZN(\boldsymbol{\zeta}_t) \geq \ZN(\bfx_{t-1})$, the proposal is accepted with certainty and the expression equals $\ZN(\boldsymbol{\zeta}_t)^{-s\kappa}$. If $\ZN(\boldsymbol{\zeta}_t) < \ZN(\bfx_{t-1})$, then $\ZN(\bfx_{t-1})^{-s\kappa} \leq \ZN(\boldsymbol{\zeta}_t)^{-s\kappa}$, so the convex combination is also bounded by $\ZN(\boldsymbol{\zeta}_t)^{-s\kappa}$. In both cases,
    \begin{equation}\label{eq:uis-inv-onestep-bound}
      \bE\!\left[\ZN(\bfx_t)^{-s\kappa} \mid \bfx_{t-1}, \boldsymbol{\zeta}_t\right]
      \leq \ZN(\boldsymbol{\zeta}_t)^{-s\kappa}.
    \end{equation}
    Since $\boldsymbol{\zeta}_t \sim \bar{q}$ independently of $\bfx_{t-1}$, taking expectations yields
    \begin{equation}\label{eq:uis-inv-induction}
      \bE\!\left[\ZN(\bfx_t)^{-s\kappa}\right]
      \leq \bE_{\bar{q}}\!\left[\ZN(\bfx)^{-s\kappa}\right]
      \leq q(\omega^{-\eta})^{s\kappa/\eta},
    \end{equation}
    where the last inequality is Proposition~\ref{prop:inversemoments} with $r = s\kappa$, for $N \geq s\kappa/\eta$. The same bound holds for $\bfy_{t-1}$. Combining~\eqref{eq:uis-inv-mink} and~\eqref{eq:uis-inv-induction}:
    \begin{equation}\label{eq:uis-inv-delta-constant}
      \bE[|\Delta_t|^{s\kappa}]^{1/(s\kappa)}
      \leq 2\,q(\omega^{-\eta})^{1/\eta}
      =: \bar{C}_\eta < \infty,
    \end{equation}
    for all $\kappa$ with $N \geq s\kappa/\eta$, with \emph{no constraint} on $s\kappa$ relative to $p$.
    
    Substituting~\eqref{eq:uis-inv-delta-constant} and the meeting-time bound $\bP(\tau > t) \leq C/(N^{1/2} t^p)$ of Proposition~\ref{prop:upb_meetingproba} into~\eqref{eq:uis-inv-holder}, and summing over $t$:
\[
    \sum_{t=1}^{\infty} \bE[|\Delta_t|^s \mathds{1}(\tau > t)]^{1/s}
    \leq \bar{C}_\eta \sum_{t=1}^{\infty} \left(\frac{C}{N^{1/2} t^p}\right)^{(\kappa-1)/(s\kappa)}.
\]
    The series converges if and only if $p(\kappa-1)/(s\kappa) > 1$, i.e.\ $\kappa > p/(p-s)$. Since $p > s$, we have $p/(p-s) < \infty$, and since~\eqref{eq:uis-inv-delta-constant} imposes no upper bound on $\kappa$, such a choice exists. This gives $\bE[|\text{BC}|^s] < \infty$.
    
    \medskip
    
    \emph{Rate $N^{-s/2}$ under $p > 2s$.}
    To upgrade from finiteness to the rate $N^{-s/2}$, replace the constant bound~\eqref{eq:uis-inv-delta-constant} by the PIMH bound of Proposition~\ref{prop:pimh-inverse}: for $s\kappa < p$,
    \begin{equation}\label{eq:uis-inv-delta-pimh}
      \bE[|\Delta_t|^{s\kappa}]^{1/(s\kappa)} \leq C_1 N^{-1/2}.
    \end{equation}
    Two constraints on $\kappa$ must now hold simultaneously:
    \begin{equation}\label{eq:uis-inv-compatibility}
      \frac{p}{p-s} < \kappa < \frac{p}{s}.
    \end{equation}
    The first ensures the convergence of the series; the second ensures that the PIMH bound applies. This interval is non-empty if and only if $p/(p-s) < p/s$, i.e.\ $p > 2s$. Choosing such $\kappa$ and separating the $N$-dependent factors from the convergent $t$-series gives $\bE[|\text{BC}|^s] = o(N^{-s/2})$, and the MSE equivalence for $s = 2$ follows from~\eqref{eq:MSE_unbiasedSNIS_decompCS}.
  \end{proof}

  \section{Alternative unbiased importance sampling estimators\label{app:other_uis}}

Here we recall two other approaches to unbiased importance sampling mentioned in the beginning of Section~\ref{sec:biasremoval_snis}, which are not based on coupled PIMH chains. These alternatives are implemented in the experiments of Section~\ref{sec:applications}.
  
  \subsection{MLMC-UIS}

  One approach to removing the bias from a sequence of consistent estimates, such as SNIS estimates as $N\to\infty$,
  is based on multilevel Monte Carlo and randomization; we refer to it as MLMC-UIS.
This is the approach described in \citet{blanchet2015unbiasedmultilevel,wang2023unbiased},
and specifically applied to SNIS in \citet[][Section 4]{ShiCornish}.
Let $\rho$ in $(0,1)$ be the parameter of a Geometric distribution, with 
$\mathbb{P}(R = k) = (1-\rho)^k \rho$ for $k\geq 0$, and $\mathbb{P}(R\geq k) = (1-\rho)^k$.
Introduce the following notation:
\begin{itemize}
  \item For a sample $x_1,\ldots,x_{n}$, with $n$ even, we write 
  $E(x_1,\ldots,x_n)$ for the sample $(x_{2i})_{i=1}^{n/2}$ of the even-indexed variables, and $O(x_1,\ldots,x_n)$ for the sample $(x_{2i-1})_{i=1}^{n/2}$ of the odd-indexed variables.
  \item We write $\hat{F}(x_1,\ldots,x_n)$ for the SNIS estimator of $\pi(f)$ based on the sample $x_1,\ldots,x_n$, and $\hat{F}\circ E (x_1,\ldots,x_n)$ for the SNIS estimator of $\pi(f)$ based on the even-indexed variables, and $\hat{F}\circ O (x_1,\ldots,x_n)$ for the SNIS estimator of $\pi(f)$ based on the odd-indexed variables.
\end{itemize}
The MLMC-UIS estimator is obtained with the following procedure:
\begin{enumerate}
  \item Choose an initial level $\ell \geq 0$. You can e.g. take $\ell = 1$.
  \item Draw $R \sim \text{Geometric}(\rho)$.
  \item Draw $x_1,\ldots,x_{2^{\ell+R+1}}$ i.i.d.\ samples from $q$. You may compute once and for all $\omega(x_i)$ and $f(x_i)$ for $i=1,\ldots,2^{\ell+R+1}$, as these will be used in the different SNIS estimators computed below.
  \item Compute
  \begin{align*}
    S_{R} &= \hat{F}(x_1,\ldots,x_{2^{\ell+R+1}}),\\
    S_{R}^O &= \hat{F}\circ O (x_1,\ldots,x_{2^{\ell+R+1}}),\\
    S_{R}^E &= \hat{F}\circ E (x_1,\ldots,x_{2^{\ell+R+1}}).
  \end{align*}
  \item Compute
  \[ \Delta_R = S_{R} - \frac{S_{R}^O + S_{R}^E}{2},\]
  whose expectation is equal the bias difference between SNIS estimators with $2^{\ell+R+1}$ and $2^{\ell+R}$ samples.
  \item Return either the ``single sample'' or the ``Russian roulette'' estimator, defined as follows:
  \begin{align*}
    \mathrm{SS} &= \frac{1}{2^{R+1}} \sum_{n=1}^{2^{R+1}} \hat{F}(x_{(n-1)2^\ell + 1},\ldots,x_{n 2^{\ell}}) + \frac{\Delta_R}{(1-\rho)^{R} \rho },\\
    \mathrm{RR} &=  \frac{1}{2^{R+1}} \sum_{n=1}^{2^{R+1}} \hat{F}(x_{(n-1)2^\ell + 1},\ldots,x_{n 2^{\ell}})  + \sum_{r=0}^R \frac{\Delta_r}{(1-\rho)^{r}}.
  \end{align*}
\end{enumerate} 
The first term in both the SS and RR estimators represents the average of many SNIS estimators of $\pi(f)$
each based on $2^{\ell}$ samples. This is biased for $\pi(f)$. The second terms in SS and RR are bias correction terms.

\begin{remark}
  The tuning parameters are $\ell$, which can be set to e.g. one, and the parameter $\rho$ of the Geometric distribution, which controls the tail of the  distribution of $R$, and thus both the cost and variance of the estimator.
\end{remark}

Theorem 3 of \citet{ShiCornish} provides conditions under which the variance of the SS and RR estimators is finite, and the expected cost is finite.
There are two sets of conditions:
\begin{itemize}
  \item 
Assume that $q(\omega^a + |f|^b)<\infty$ with $a>4$ and $b>2a/(a-4)$.
Then set $\alpha = 1-2/b$.
\item 
Assume that  $q(\omega^a + |f|^b)<\infty$ with $2< a \leq 4$ and $b>(2a+4)/(a-2)$,
or $a>4, b \in ((2a+4)/(a-2), 2a/(a-4))$. Then set $\alpha = a/2 - (a+2)/b - 1$.
\end{itemize} 
Then the estimator has finite expected cost and finite variance for $\rho \in (2^{-1}, 1 - 2^{-1-\alpha})$. The larger the $\alpha$, the larger the range of $\rho$ for which the estimator has finite variance and finite expected cost. 

\begin{remark} If the test function $f$ is chosen to be $x\mapsto \omega_u(x)^{-1}$, where $Z = \int \omega_u(x) \, dx$, then the unbiased MLMC estimators are unbiased estimators of $1/Z$.
\end{remark}

\begin{remark} In the Exponential example considered in Section~\ref{subsec:applications:inverse_Z},
  with $p = 3$, we are in the case $2 < a \leq 4$. Considering the function $x\mapsto \omega_u(x)^{-1}$,
  we can take $b$ arbitrarily large. Therefore, we select $\alpha = 1/2$, define $\rho_{\max} = 1 - 2^{-1-\alpha} = 1 - 2^{-3/2} \approx 0.6464$, and choose $\rho$ as $0.95 \rho_{\max}$ in the experiments to reflect a slightly conservative but well-informed choice.
\end{remark}

\subsection{Taylor-UIS}

Another approach to removing the bias of SNIS is based on Taylor expansions and randomization, and was proposed in \citet{blanchet2015unbiasedtaylor}, with useful improvements and analyses in \citet{chopin_crucinio_singh_2025}. In these articles,
the goal is to estimate a smooth function of the expectation of a random variable that we can sample from. In the context of importance sampling, the smooth function is the inverse function, and we apply the method of \citet{chopin_crucinio_singh_2025} to estimate $1/Z$ without bias. We can then multiply the resulting estimator by an unbiased estimator of $q(f\omega_u)$ to obtain an unbiased estimator of $\pi(f)$. We refer to this approach as Taylor-UIS. 

From $g:m\mapsto 1/m$, the starting point is the identity
\begin{equation}\label{eq:ccs_identity}
  g(m) = \sum_{k=0}^\infty \gamma_k \left(\frac{m}{x_0} - 1\right)^k,
\end{equation}
where $x_0$ is a tuning parameter, and the coefficients $\gamma_k$ are given by
\begin{align*}
  \forall k \geq 1 \quad \gamma_k &= \frac{g^{(k)}(x_0)}{k!} x_0^k = \frac{(-1)^{k}}{x_0}.
\end{align*}
The basic idea is that:
\begin{enumerate}
  \item The infinite sum in \eqref{eq:ccs_identity} can be estimated without bias with a truncated sum of reweighted terms. 
  Introduce $R \sim \text{Geometric}(\rho)$, such that $\bP(R = k) = (1-\rho)^k\rho$ for $k\geq 0$, where $\rho\in(0,1)$. Note that $\mathbb{P}(R\geq k) = (1-\rho)^k$.
  \item 
Each term $(m/x_0 - 1)^k$ in \eqref{eq:ccs_identity} can be unbiasedly estimated using draws from $Y$ with expectation $m$. The most natural estimator is \citep{blanchet2015unbiasedtaylor}:
\begin{align}\label{eq:ccs_urk}
  U_{r,k}&= \prod_{i=1}^k \left(\frac{Y_i}{x_0} - 1\right).
\end{align}
However, a better estimator is as follows: if one has access to $r$ random variables $Y_1,\ldots,Y_r$ with $r>k$, then 
one can average over different possible products of $k$ distinct variables. \citet{chopin_crucinio_singh_2025}
consider averaging over $r$ circular shifts of the variables:
\begin{align}\label{eq:ccs_urkC}
  U^C_{r,k}&= \frac{1}{r}\left\{ \sum_{j=0}^{r-1} \prod_{i=1}^{k} \left(\frac{Y_{j+i}}{x_0} - 1\right)\right\},
\end{align}
where the index $j+i$ is taken modulo $r$, i.e. $Y_{r+1} = Y_1$, $Y_{r+2} = Y_2$, etc.
\end{enumerate}
Combining the above considerations, an unbiased estimator of $1/m$ is
\begin{equation}\label{eq:ccs_estimator}
  R\sim \text{Geometric}(\rho), \quad Y_1,\ldots,Y_R \sim \mathcal{L}(m,\sigma^2), \quad \hat{I} := \sum_{k=0}^R \frac{(-1)^{k} U^C_{R,k}}{x_0 \mathbb{P}(R \geq k)}.
\end{equation}
where $\mathcal{L}(m,\sigma^2)$ is the distribution of $Y_i$ with expectation $m$ and variance $\sigma^2$.
In our setting each $Y_i$ would correspond to $\hat{Z}=N^{-1}\sum_{n=1}^N \omega_u(x_n)$ where $x_n$ are i.i.d.\ samples from $q$, and thus we need to choose a tuning parameter $N$. The two additional parameters are $x_0$ and $\rho$.
\citet{chopin_crucinio_singh_2025} provide guidance on tuning parameters in their Section 3. Introduce the following notation:
\begin{align*}
  &\beta_0 = |m/x_0 - 1|, \sigma^2 = \mathbb{V}[Y_i], \\
  &\hat{m} = \frac{1}{n_0} \sum_{i=1}^{n_0} Y_i, \quad \hat{\sigma}^2 = \frac{1}{n_0-1} \sum_{i=1}^{n_0} (Y_i - \hat{m})^2,
\end{align*}
where $n_0$ is the length of a pilot sample of $Y_i$'s, used to estimate $m$ and $\sigma^2$. 
\citet{chopin_crucinio_singh_2025} choose $n_0 = 10$.
Then define
\begin{align*}
  &\hat{x_0^\star} = (\hat{m}^2 + \hat\sigma^2)/ \hat{m},\\
  & x_{0,\text{ub}} = \text{bootstrap percentile at level } (1 - \alpha) \text{ of } (\hat{m}^2 + \hat\sigma^2)/(2\hat{m}),\\
  & x_0 = \max(\hat{x_0^\star}, x_{0,\text{ub}}),
\end{align*}
where the bootstrap is performed by resampling with replacement from the pilot sample of $Y_i$'s. 
By default we follow \citet{chopin_crucinio_singh_2025} and use 100 bootstrap samples and $\alpha = 0.01$.
Having defined $x_0$, we can define
\begin{equation*}
  \hat{\beta}^2 = \frac{\hat\sigma^2}{x_0^2} + \left(\frac{\hat{m}}{x_0} - 1\right)^2,
\end{equation*}
and 
\begin{equation*}
  \rho = \min(1-\hat{\beta}^2, 1/(n_0+1)).
\end{equation*}
See Section 3.2 of \citet{chopin_crucinio_singh_2025} for details on the choice of $\rho$.
Having obtained $\rho$ and $x_0$ from $n_0$ samples of $X_i$, we can then run the $\hat{I}$ estimator as in \eqref{eq:ccs_estimator}:
\begin{enumerate}
  \item Draw $R \sim \text{Geometric}(\rho)$.
  \item Draw $R$ i.i.d.\ samples $Y_1,\ldots,Y_R$.
  \item Compute $U^C_{R,k}$ for $k=1,\ldots,R$ using \eqref{eq:ccs_urkC}.
  \item Compute $\hat{I} = x_0^{-1} + \sum_{k=1}^R \frac{(-1)^{k} U^C_{R,k}}{x_0 (1-\rho)^k}$, an unbiased estimator of $1/m$.
\end{enumerate}

A natural estimator of $q(\omega_u \cdot f)$ is given by $M^{-1}\sum_{n=1}^M \omega_u(x'_n) f(x'_n)$, where $x'_1,\ldots,x'_M$ are i.i.d.\ samples from $q$, independent of all the samples used to obtain the unbiased estimator of $Z^{-1}$. Finally we obtain
\begin{equation*}
  \hat{\pi}(f) = \hat{I} \times \left(\frac{1}{M} \sum_{n=1}^M \omega_u(x'_n) f(x'_n)\right).
\end{equation*}
The parameter $M$ should be tuned, relative to the tuning of each $Y_i$ (which itself is a $\ZN$ estimator). 
In the application of Section~\ref{subsec:applications:inverse_Z}, 
our interest is in estimating $\pi(f)$ for $f(x) = \omega_u(x)^{-1}$, which is equivalent to estimating $1/Z$,
so we do not need to estimate $q(\omega_u \cdot f)$, and we can simply use the $\hat{I}$ estimator of $1/Z$ in \eqref{eq:ccs_estimator}. 
We leave the question of how to optimally split the budget between estimating $Z^{-1}$ and estimating $q(\omega_u \cdot f)$ for future work.

\end{document}